\documentclass[a4paper,12pt]{article}

\usepackage{latexsym}
\usepackage{amsmath}
\usepackage{amsthm}
\usepackage{amsbsy}
\usepackage{amssymb}
\usepackage{graphicx}
\usepackage{epstopdf}
\usepackage{mathrsfs}
\usepackage{times}
\usepackage{afterpage}
\usepackage[prefix]{nomencl}
\usepackage{psfrag}
\usepackage{multido}
\usepackage{algorithm} 
\usepackage{algpseudocode}
\usepackage{mathabx}

\usepackage{amsfonts}
\usepackage{arydshln}
\usepackage{booktabs}	
\usepackage{url}
\usepackage{hyperref} 

\input epsf

\renewcommand{\vec}[1]{\boldsymbol{#1}}
\newcommand{\dif}{\mathrm{d}}

\usepackage{color}

\DeclareMathOperator*{\Res}{Res}

\usepackage{amsthm}
\newtheorem{theorem}{Theorem}[section]
\newtheorem{remark}[theorem]{Remark}

\newtheorem{lemma}[theorem]{Lemma}
\newtheorem{corollary}[theorem]{Corollary}
\allowdisplaybreaks[3]

\newcommand{\im}{\mathrm{i}}

\begin{document}
\begin{center}
{\Large \bf 
The Spectral Properties of the Magnetic Polarizability Tensor for Metallic Object Characterisation}
\\
P.D. Ledger$^*$ and W.R.B. Lionheart$^\dagger$ \\
$^*$Zienkiewicz Centre for Computational Engineering, College of Engineering, \\Swansea University Bay Campus, Swansea. SA1 8EN UK\\
$^\dagger$School of Mathematics, Alan Turing Building, \\The University of Manchester, Oxford Road, Manchester, M13 9PL UK\\
31st May 2019
\end{center}

\section*{Abstract}
The measurement of  time-harmonic perturbed field data, at a range of frequencies, is beneficial for practical metal detection where the goal is to locate and identify hidden targets.
 In particular, these benefits are realised  when frequency dependent magnetic polarizability tensors (MPTs)  are used to provide an economical characterisation of  conducting permeable objects and a dictionary based classifier is employed. However, despite the advantages shown in dictionary based classifiers, the behaviour of the MPT coefficients with frequency is not properly understood.  In this paper, we rigorously analyse, for the first time, 
 the spectral properties of the coefficients of the MPT. This analysis has the potential to improve existing algorithms and design new approaches for object location and identification in metal detection. Our analysis also enables the response transient response from a conducting permeable object to be predicted for more general forms of excitation.

{\bf Keywords:} Eddy current, Metal Detection, Inverse Problems, Asymptotic Analysis, Spectral Problems, Magnetic Polarizability Tensor. \\

MSC Classification: 35R30; 35B30

\section{Introduction}

In metal detection, there is considerable interest in being able to locate and identify conducting permeable objects from the measurements of mutual inductance between a transmitting and a measurement coil. Applications include  security screening, archaeology excavations, ensuring food safety as well as the search for landmines and unexploded ordnance. There are also closely related topics such as magnetic induction tomography for medical imaging and eddy current testing for monitoring the corrosion of steel reinforcement in concrete structures.

Within the metal detection community, magnetic polarizability tensors (MPTs) have attracted considerable interest to assist with the identification of objects when the transmitting coil is excited by a sinusoidal signal e.g~\cite{abdel,das1990,dekdouk,Gregorczy,marsh2014b,marsh2013,zhao2016}. Engineers believe that a rank 2 MPT provides an economical characterisation of a conducting permeable object  that is invariant of position. An asymptotic formula providing the leading order term for the perturbed magnetic field due to the presence of a small conducting permeable object has been obtained by Ammari, Chen, Chen, Garnier and Volkov~\cite{ammarivolkov2013}, which characterises the object in terms of a rank 4 tensor. We have shown that this simplifies for orthonormal coordinates and allows an object to be characterised by a complex symmetric rank 2 MPT, with an explicit formula for its coefficients, thus, justifying the earlier engineering conjecture~\cite{ledgerlionheart2012}. We have extended Ammari {\it et al.}'s work to provide a complete asymptotic expansion for the perturbed magnetic field, which allows an object to be characterised by generalised MPTs, of which the rank 2 MPT is the simplest case~\cite{ledgerlionheart2017}. In~\cite{ledgerlionheart2018mathmeth}, we have developed asymptotic expansions for the perturbed magnetic field that describe, 1), the leading order term for a small inhomogeneous conducting permeable object  and, 2), the response in the case of multiple small conducting permeable objects, thus, extending the rank 2 MPT description of each of the objects to these cases. Some properties of the MPT are described in~\cite{ledgerlionheart2016} and the availability of the explicit formula for the isolated single object case has been communicated to the engineering community~\cite{ledgerlionheart2018}. 

The measurement of  time-harmonic perturbed field data, at a range of frequencies, has already been shown to be 
beneficial for object location and object identification with several dictionary based algorithms being proposed for metal detection~\cite{ammarivolkov2013,ammarivolkov2013b,ledgerlionheart2018mathmeth} and electro-sensing~\cite{ammarielectrosensing2013}. These benefits are realised due to the frequency dependence of the tensor coefficients.
 In the case of an object with homogeneous conductivity and permeability, computational results (e.g.~\cite{ledgerlionheart2016}) have shown, for a range of object shapes and topologies, that the eigenvalues of the real part of the MPT is
 monotonic and bounded with logarithmic frequency while the eigenvalues of the imaginary part of the MPT
 has a single local maximum
 with  logarithmic frequency. The behaviour was found to be similar for objects with inhomogeneous permeability, however, for objects with inhomogeneous conductivity, computational results show that the eigenvalues of the real part have multiple non-stationary infection points and the imaginary part has multiple local maxima~\cite{ledgerlionheart2018mathmeth}. This was also found to be the case in the measurements of the MPTs of US coins, which are made up of different conducting materials~\cite{davidson2018}.  

An insight in to the spectral behaviour of MPT coefficients is provided by the analytical solution for a conducting permeable sphere obtained by Wait and Spies~\cite{wait1969} who also provide a description for the transient perturbed magnetic field when the excitation is through a step or impulse function. 
Baum~\cite{baum499} has suggested the form that the spectral response for an MPT for a homogeneous conducting object should take a similar form, but does give a formal proof or an explicit expansion. Instead, he uses heuristic arguments to justify its existence. He uses this as a basis for the so called singularity expansion method~\cite{baum521,baumbook} where he proposes that the transient response to a conducting permeable object can be characterised by a series of resonant frequencies, rather than its MPT coefficients. This approach, however, has been received with scepticism when applied to objects other than spheres due to its lack of rigour~\cite{ramm1982}. In this work, we present the first rigorous spectral analysis of the MPT coefficients and its eigenvalues, which, we anticipate, will lead to improvements to existing algorithms and the design of new approaches for object location and object identification. Furthermore, the improved understanding the spectral behaviour of the MPT coefficients also allows the transient response of conducting permeable objects to be understood when the excitation is not harmonic (e.g. in the case of delta function) extending the work of Wait and Spies and making rigorous the work of Baum. This is of practical value for metal detectors, which use impulse or other non-harmonic forms of excitation. 
See also related work in the field of electro-sensing~\cite{ammarielectrosensing2014}.
These points are addressed in our work through the following novelties:

\begin{enumerate}
\item A new alternative invariant form of the MPT is introduced where the coefficients of the MPT follow from symmetric bilinear forms.
\item Explicit expressions for the MPT coefficients for the limiting cases of an inhomogeneous permeable object at low frequency and an inhomogeneous object with  infinite conductivity (perfectly conducting) for multiply connected topologies.
\item A new alternative view point is introduced where the MPT coefficients are derived from an energy functional expressed as  a sum of three inner products, describing  the magnetostatic and time varying magnetic and electric energies. This leads to explicit expressions for the real and imaginary parts of the MPT in the form
\begin{align}
{\mathcal M} = {\mathcal N}^0 + {\mathcal R}^{\sigma_*}+ \im {\mathcal I}^{\sigma_*} ,
\end{align}
for a general inhomogeneous object. In the above, ${\mathcal N}^0$, ${\mathcal R}^{\sigma_*}$ and ${\mathcal I }^{\sigma_*}$ are all real symmetric rank 2 tensors, the former describes the magnetostatic response and the latter two are frequency dependent. The frequency behaviour of the coefficients of  ${\mathcal R}^{\sigma_*}$ and $ {\mathcal I}^{\sigma_*}$ is also explicitly derived.

\item The introduction of an eigenvalue problem that allows the spectral behaviour of ${\mathcal M}$, and hence the frequency response of ${\mathcal R}^{\sigma_*}$ and ${\mathcal I}^{\sigma_*}$,  to be understood as a convergent infinite series using the Mittag-Leffler theorem.

\item Explicit forms of the transient response from a homogeneous conducting permeable object when the excitation is a step function or an impulse function are derived.  For the step function, this rigorously shows that the long time response  is that of a permeable object and, for an impulse function, the short time response is that of a perfect conductor.
\end{enumerate}

The paper is organised as follows: In Section~\ref{sect:background} some background on the characterisation of a conducting permeable object by an MPT is briefly reviewed. Then, in Section~\ref{sect:invariantform}, a new invariant form of the MPT is presented. Section~\ref{sect:limitcase} describes the explicit expressions for the MPT coefficients for the limiting cases of an inhomogeneous  non-conducting permeable multiply connected object at low frequency and an inhomogeneous multiply connected object with  infinite conductivity (perfectly conducting).  In Section~\ref{sect:enfun}, an energy functional is defined, from which the  MPT coefficients follow, leading to explicit expressions for ${\mathcal R}^{\sigma_*}$ and ${\mathcal I }^{\sigma_*}$.  Section~\ref{sect:bounds} provides bounds on ${\mathcal R}^{\sigma_*}$ and ${\mathcal I }^{\sigma_*}$, generalising the results already known for ${\mathcal N}^0$. Then, in Section~\ref{sect:eig}, explicit expressions for the eigenvalues of the tensors ${\mathcal N}^0$, ${\mathcal R}^{\sigma_*}$ and ${\mathcal I}^{\sigma_*}$ are derived. Section~\ref{sect:spectrum} presents a spectral analysis of the MPT coefficients allowing their behaviour with frequency to be understood. Using this analysis, the transient response for several different forms of excitation is obtained in Section~\ref{sect:tranresp}. 

\section{Characterisation of conducting permeable objects} \label{sect:background}
We begin by considering the characterisation of a single homogenous conducting permeable object.
Following~\cite{ammarivolkov2013,ledgerlionheart2014}, we describe a single inclusion by $B_\alpha := \alpha B + {\vec z}$, which means that it can be thought of a unit-sized object $B$ located at the origin, scaled by $\alpha$ and translated by ${\vec z}$. We assume the background is non-conducting and non-permeable and introduce the position dependent conductivity and permeability as
\begin{align}
\sigma_\alpha= \left \{ \begin{array}{ll} \sigma_* & \text{in $B_\alpha$} \\
                                                             0 & \text{in $B_\alpha^c:= {\mathbb R}^3 \setminus\overline{ B_\alpha}$}
                                                             \end{array},  \right . \qquad
\mu_\alpha= \left \{ \begin{array}{ll} \mu_* & \text{in $B_\alpha$} \\
                                                             \mu_0 & \text{in $B_\alpha^c$}
                                                             \end{array} \right .  , \label{eqn:eddymodel}
\end{align}
where $\mu_0:= 4 \pi \times 10^{-7} \text{H/m}$ is the permeability of free space, $0 < \mu_* < \infty$ and $0\le \sigma_* < \infty$. For metal detection, the relevant mathematical model is the eddy current approximation of Maxwell's equations since $\sigma_*$ is large and the angular frequency $\omega=2 \pi f$ is small (see Ammari, Buffa and N\'ed\'elec~\cite{ammaribuffa2000} for a detailed justification). The electric and magnetic interaction fields, ${\vec E}_\alpha$ and ${\vec H}_\alpha$, respectively, satisfy the curl equations
\begin{equation}
\nabla \times {\vec H}_\alpha = \sigma_\alpha {\vec E}_\alpha +{\vec J}_0, \qquad 
\nabla \times {\vec E}_\alpha = \im \omega \mu_\alpha {\vec H}_\alpha, \label{eqn:eddyqns}
\end{equation}
in ${\mathbb R}^3$ and decay as $O(|{\vec x}|^{-1})$ for $|{\vec x}| \to \infty$. In the above, ${\vec J}_0$ is an external current source with support in $B_\alpha^c$. In absence of an object, the background fields ${\vec E}_0$ and ${\vec H}_0$ satisfy (\ref{eqn:eddyqns}) with $\alpha=0$.

The task is to find an economical description for the perturbed magnetic  field $({\vec H}_\alpha - {\vec H}_0)({\vec x})$ due to the presence of $B_\alpha$, which characterises the object's shape and material parameters by a small number of parameters separately to its location ${\vec z}$.  For ${\vec x}$ away from $B_\alpha$, the leading order term in an asymptotic expansion for  $({\vec H}_\alpha - {\vec H}_0)({\vec x})$ as $\alpha \to 0$ has been derived by Ammari {\it et al.}~\cite{ammarivolkov2013}. We have shown that this reduces to the simpler form~\cite{ledgerlionheart2014,ledgerlionheart2018}~\footnote{In order to simplify notation, we drop the double check on ${\mathcal M}$ and the single check on ${\mathcal C}$, which was used in~\cite{ledgerlionheart2014} to denote two and one reduction(s) in rank, respectively. We recall that ${\mathcal M}= ({\mathcal M})_{ij} {\vec e}_i \otimes {\vec e}_j$, by the Einstein summation convention, where we use the notation ${\vec e}_i$ to denote the $i$th unit  orthonormal basis vector and repeated indices imply summation unless otherwise stated. We will denote the $i$th component of a vector ${\vec u}$ by ${\vec u} \cdot {\vec e}_i = ({\vec u})_i$ and the $i,j$th coefficient of a rank 2 tensor ${\mathcal M}$ by $({\mathcal M})_{ij}$.} 
\begin{align}
({\vec H}_\alpha - {\vec H}_0)({\vec x})_i =&( {\vec D}_{x}^2 G({\vec x},{\vec z}))_{ij}  ( {\mathcal M}[\alpha B])_{jk} ({\vec H}_0({\vec z}) )_k +( {\vec R}({\vec x}))_i   \nonumber \\
= &  \frac{1}{4 \pi r^3} \left ( 3\hat{\vec r} \otimes \hat{\vec r} - {\mathbb I} \right)_{ij}  ( {\mathcal M} [\alpha B])_{jk} ({\vec H}_0({\vec z}) )_k + ({\vec R}({\vec x}))_i    \label{eqn:asymp}.
\end{align}
In the above, $G({\vec x},{\vec z}) := 1/ ( 4 \pi |{\vec x}-{\vec z}|)$  is the free space Laplace Green's function, ${\vec r}: ={\vec x}- {\vec z}$, $r=|{\vec r}|$ and $\hat{\vec r}= {\vec r}/ r$ and ${\mathbb I}$ is the rank 2 identity tensor.
The term ${\vec R}({\vec x})$ quantifies the remainder and it is known that  $|{\vec R}| \le C \alpha^4 \| {\vec H}_0 \|_{W^{2,\infty}(B_\alpha) }$.  The result holds when $\nu \in {\mathbb R}^+  := \sigma_* \mu_0 \omega \alpha^2 = O(1)$ as $\alpha \to 0$ (this includes the case of fixed $\sigma_*$, $\mu_*$, $\omega$ as $\alpha \to 0$). Note that (\ref{eqn:asymp}) involves the evaluation of the background field within the object, usually at it's centre i.e. ${\vec H}_0({\vec z}) $,  and requires it to be analytic at this location. In addition, the notation ${\mathcal M}[\alpha B]$ is used to denote that ${\mathcal M}$ is evaluated for the configuration $\alpha B$. 
 In the following, we write ${\mathcal M}$ for ${\mathcal M}[\alpha B]$ where no confusion arises.

The  rank 2 tensor ${\mathcal M}:= - {\mathcal C} + {\mathcal N}= (- ({\mathcal C})_{ij} + ({\mathcal N})_{ij} ){\vec e}_i \otimes {\vec e}_j $ depends on $\omega$, $\sigma_*$, $\mu_*/\mu_0$, $\alpha$ and the shape of $B$, but is independent of ${\vec z}$. This is the MPT and its coefficients can be computed from vectorial solutions ${\vec \theta}_j( {\vec \xi}) $, $j=1,2,3$, to a transmission problem, which we will state shortly, using
\begin{subequations}  \label{eqn:mcheck} 
\begin{align}
( {\mathcal C} )_{ij} :=& -\frac{\im  \alpha^3 }{4}{\vec e}_i \cdot \int_B\nu  {\vec \xi} \times ({\vec \theta}_j + {\vec e}_j \times {\vec \xi} ) \dif {\vec \xi},  \\
({\mathcal N} )_{ij} := &  \alpha^3  \int_B \left ( 1- \frac{\mu_0}{\mu_*} \right ) \left (
{\vec e}_i \cdot {\vec e}_j + \frac{1}{2}  {\vec e}_i \cdot  \nabla \times {\vec \theta}_j \right ) \dif {\vec \xi}.
\end{align}
\end{subequations}
If the object is inhomogeneous, with possibly different piecewise-constant values of $\mu_*$ and $\nu$ in different regions of the object, then (\ref{eqn:asymp}) and (\ref{eqn:mcheck}) still hold if
 we replace $B$ with ${\vec B} = \bigcup_{n=1}^N B^{(n)}$ and $B_\alpha = \alpha B + {\vec z}$ by ${\vec B}_\alpha =\alpha {\vec B} + {\vec z}$  to describe the fact that ${\vec B}$ is made up of $N$ regions~\cite{ledgerlionheart2018mathmeth}. We require that  $B$ (and  $B^{(n)}$) have Lipschitz boundaries and note that
\begin{align}
\mu = \left \{ \begin{array}{ll}  \mu_*( {\vec \xi}) & \text{in ${\vec B}$}  \\
\mu_0 & \text{in ${\vec B}^c := {\mathbb R}^3 \setminus \overline{\vec B}$}  \end {array} \right . ,
\qquad
\sigma  = \left \{ \begin{array}{ll}  \sigma_*( {\vec \xi}) & \text{in ${\vec B}$ } \\
0 & \text{in ${\vec B}^c$}  \end {array} \right . ,  \nonumber
\end{align} 
and
\begin{align}
\mu_* =  \begin{array}{cc}
\mu_*^{(n)} & \text{in $B^{(n)}$ }\end{array} ,  \qquad
\sigma_* =  \begin{array}{cc} \sigma_*^{(n)} & \text{in $B^{(n)}$} \end{array} ,  \qquad
\nu =  \begin{array}{cc} \omega \alpha^2\mu_0  \sigma_*^{(n)} & \text{in $B^{(n)}$} \end{array}   \nonumber ,
\end{align}
where $0 < \mu_*^{(n)} < \infty$, $0 \le \sigma_*^{(n)} < \infty$.
In addition,  $\alpha$  denotes the size of the (combined) configuration and ${\vec z}$ its location. Throughout the following, we concentrate on results  for the case of ${\vec B}$, but these readily simplify to the case of ${ B}$.
The aforementioned  transmission problem is
\begin{subequations}
\begin{align}
\nabla_\xi \times \mu_*^{-1} \nabla_\xi \times {\vec \theta}_j - \im \omega \sigma_* \alpha^2 {\vec \theta }_j & = \im \omega \sigma_* \alpha^2 {\vec e}_j \times {\vec \xi}  && \text{in ${\vec B}$ } ,\\
\nabla_\xi \cdot {\vec \theta}_j  = 0 , \qquad \nabla_\xi \times \mu_0^{-1} \nabla_\xi \times {\vec \theta}_j  & = {\vec 0} && \text{in ${\vec B}^c$ }, \\
[{\vec n} \times {\vec \theta}_j ]_\Gamma = {\vec 0},  \qquad [{\vec n} \times \mu^{-1} \nabla_\xi \times {\vec \theta}_j ]_\Gamma & = -2 [\mu^{-1 } ]_\Gamma {\vec n} \times {\vec e}_j  && \text{on $\Gamma$},\\
{\vec \theta}_j & = O( | {\vec \xi} |^{-1}) && \text{as $|{\vec \xi} | \to \infty$ },
\end{align}\label{eqn:transproblemthetar}%
\end{subequations}
which is solved for ${\vec \theta}_j({\vec \xi})$, $j=1,2,3$. In the above,  $[\cdot ]_\Gamma $ denotes the  jump of the function over $\Gamma$, with $\Gamma=\partial B$  for the homogeneous case or $\Gamma= \partial {\vec B} \cup \{ \partial B^{(n)} \cap \partial B^{(m)}, n,m=1,\cdots,N, n \ne m \}$ otherwise, and ${\vec \xi}$ is measured from an origin chosen to be in $B$ or ${\vec B}$, respectively. 

\section{Invariant form of ${\mathcal M}$} \label{sect:invariantform}
We define ${\vec \Theta}({\vec u})$, for a constant real vector ${\vec u}$, to be the complex vector field solution of the transmission problem
\begin{subequations}
\begin{align}
\nabla_\xi \times \mu_r^{-1} \nabla_\xi \times {\vec \Theta} - \im \nu   {\vec \Theta } & = \im\nu {\vec u} \times {\vec \xi}  && \text{in ${\vec B}$ } ,\\
\nabla_\xi \cdot {\vec \Theta}  = 0 , \qquad \nabla_\xi \times  \nabla_\xi \times {\vec \Theta}  & = {\vec 0} && \text{in ${\vec B}^c$ }, \\
[{\vec n} \times {\vec \Theta} ]_\Gamma = {\vec 0},  \qquad [{\vec n} \times \tilde{\mu}_r^{-1} \nabla_\xi \times {\vec \Theta} ]_\Gamma & = -2 [\tilde{\mu}_r^{-1 } ]_\Gamma {\vec n} \times {\vec u}  && \text{on $\Gamma$},\\
{\vec \Theta} & = O( | {\vec \xi} |^{-1}) && \text{as $|{\vec \xi} | \to \infty$ },
\end{align}\label{eqn:transproblemthetainv}%
\end{subequations}
where, here, and in the following, the dependence of ${\vec \Theta}({\vec u})$ on position ${\vec \xi}$ is not stated explicitly for compactness of notation. In addition,
\begin{align}
\mu_r({\vec \xi}) : =\frac{ \mu_* ({\vec \xi})}{\mu_0}, \qquad  \nu({\vec \xi}) := \omega  \mu_0 \alpha^2 \sigma_*({\vec \xi}) , \qquad   {\vec \xi} \in {\vec B},  \qquad 
\tilde{\mu}_r: = \left \{ \begin{array}{ll} \mu_r & \text{in ${\vec B}$} \\
1 & \text{in ${\vec B}^c$} 
\end{array} \right . . \nonumber
\end{align}
Thus, it clear that ${\vec \theta}_j = {\vec \Theta}({\vec e}_j)$. In addition, setting
\begin{subequations}  \label{eqn:mcheckinv} 
\begin{align}
{ C} ({\vec u},{\vec v}):=& -\frac{\im  \alpha^3 }{4}{\vec u} \cdot \int_B \nu  {\vec \xi} \times ({\vec \Theta} ({\vec v}) + {\vec v} \times {\vec \xi} ) \dif {\vec \xi},  \\
N ({\vec u},{\vec v}) := &  \alpha^3  \int_B \left ( 1- \mu_r^{-1}  \right ) \left (
{\vec u} \cdot {\vec v} + \frac{1}{2}  {\vec u} \cdot  \nabla \times {\vec \Theta} ({\vec v}) \right ) \dif {\vec \xi} , \\
M ({\vec u},{\vec v}):= & N({\vec u},{\vec v}) - { C} ({\vec u},{\vec v}), 
\end{align}
\end{subequations}
where ${\vec v}$ is also a constant real vector then, obviously, $( {\mathcal C} )_{i j} = C ({\vec e}_i, {\vec e}_j)$, $( {\mathcal N} )_{i j} = N ({\vec e}_i, {\vec e}_j)$, and $({\mathcal M})_{i j} = M({\vec e}_i , {\vec e}_j)$  are the aforementioned tensor coefficients.

To provide an alternative splitting of $ {\mathcal M} $, we generalise Lemma 1 of~\cite{ledgerlionheart2016}, which was for a homogenous object, to the inhomogeneous case, in terms of
\begin{align}
{\vec \Theta} ( {\vec u}) =&  {\vec \Theta}^{(0)} ({\vec u}) + {\vec \Theta}^{(1)} ({\vec u}) - {\vec u} \times {\vec \xi},  \nonumber \\
=&  \tilde{\vec \Theta}^{(0)} ({\vec u}) + {\vec \Theta}^{(1)} ({\vec u}),  \nonumber
\end{align}
with $ \tilde{\vec \Theta}^{(0)} ({\vec u}) :=  {\vec \Theta}^{(0)} ({\vec u}) - {\vec u} \times {\vec \xi}$,
as follows:

\begin{lemma} \label{lemma:msplitting1}
The coefficients of  $ {\mathcal M}$ in a orthonormal basis ${\vec e}_i$, $i=1,\ldots,3$, can be expressed as $ ({\mathcal M})_{ij} := M({\vec e}_i , {\vec e}_j) = 
 { N}^{\sigma_*} ({\vec e}_i,{\vec e}_j) + {N}^0  ({\vec e}_i,{\vec e}_j) - {{ C}}^{\sigma_*} ({\vec e}_i,{\vec e}_j)$ where
\begin{subequations}
\begin{align}
{{C}}^{\sigma_*}  ({\vec u},{\vec v}) & : = - \frac{\im  \alpha^3 }{4} {\vec u} \cdot \int_{\vec B} \nu {\vec \xi} \times ({\vec \Theta}^{(0)}({\vec v}) + {\vec \Theta}^{(1)} ({\vec v}) ) \dif {\vec \xi} , \label{eqn:definecconvar}\\
{ N}^{\sigma_*} ({\vec u},{\vec v} ) & := \frac{\alpha^3}{2}  \int_{\vec B} \left ( 1- \mu_r ^{-1} \right ) \left (
    {\vec u} \cdot  \nabla \times {\vec \Theta}^{(1)} ({\vec v})  \right ) \dif {\vec \xi} \label{eqn:definenconvar}, \\
{ N}^0 ({\vec u},{\vec v} ) & := \frac{\alpha^3}{2}  \int_{\vec B}  \left ( 1- \mu_r^{-1} \right )  \left (
  {\vec u} \cdot  \nabla \times {\vec \Theta}^{(0)}({\vec v}) \right ) \dif {\vec \xi}  , \label{eqn:definen0var}
\end{align}
\end{subequations}
and ${\vec u}$, ${\vec v}$ are constant real vectors.
Note that $ {\mathcal N}^{\sigma_*} -{{\mathcal C}}^{\sigma_*}= ({{N}}^{\sigma_*}  ({\vec e}_i,{\vec e}_j)  -   C^{\sigma_*} ({\vec e}_i,{\vec e}_j ) ) {\vec e}_i \otimes {\vec e}_j$ is a complex rank 2 tensor and ${\mathcal N}^0= { N}^0 ({\vec e}_i,{\vec e}_j ){\vec e}_i \otimes {\vec e}_j$ is a real rank 2 tensor. The forms ${{C}}^{\sigma_*}$, ${ N}^{\sigma_*} $ and ${ N}^0 $  depend on the solutions ${\vec \Theta}^{(0)} ({\vec u})$, ${\vec \Theta}^{(1)}({\vec u})$ to the transmission problems
\begin{subequations}
\begin{align} 
\nabla_\xi \times \tilde{\mu}_r^{-1} \nabla_\xi \times {\vec \Theta}^{(0)} & ={\vec 0}  && \hbox{in ${\vec B} \cup {\vec B}^c$}  , \\
\nabla_\xi \cdot {\vec \Theta}^{(0)} &= 0  && \hbox{in ${\vec B} \cup {\vec B} ^c$} , \\
\left [ {\vec \Theta} ^{(0)} \times {\vec n} \right ]_\Gamma   &= {\vec 0}   && \hbox{on $\Gamma$} ,\\ 
\,  \left [   \tilde{\mu}_r^{-1}   \nabla_\xi  \times {\vec \Theta}^{(0)}  \times {\vec n} \right ]_\Gamma  & = 
{\vec 0}  &&
\hbox{on $\Gamma$} ,\\ 
{\vec \Theta}^{(0)} - {\vec u} \times {\vec \xi} & = O(|{\vec \xi} |^{-1})  && \hbox{as $|{\vec \xi}| \to \infty$} ,
\end{align} \label{eqn:transproblem0var}
\end{subequations}
and
\begin{subequations}
\begin{align} 
\nabla_\xi \times \mu_r^{-1} \nabla_\xi \times {\vec \Theta} ^{(1)} - \im \nu  ({\vec \Theta}^{(1)} +{\vec \Theta}^{(0)} )  & = {\vec 0}   && \hbox{in ${\vec B}  $}  , \\
\nabla_\xi \times  \nabla_\xi \times {\vec \Theta} ^{(1)} & = {\vec 0}   && \hbox{in ${\vec B}^c$}  , \\
\nabla_\xi \cdot {\vec \Theta}^{(1)} &= 0 && \hbox{in ${\vec B}^c$} , \\
\left [ {\vec \Theta} ^{(1)} \times {\vec n} \right ]_\Gamma  & = {\vec 0}  && \hbox{on $\Gamma$} ,\\ 
  \left [   \tilde{\mu}_r^{-1}   \nabla_\xi  \times {\vec \Theta}^{(1)} \times {\vec n} \right ]_\Gamma  & = 
 {\vec 0} &&
\hbox{on $\Gamma$} ,\\ 
{\vec \Theta}^{(1)} & = O(|{\vec \xi} |^{-1}) && \hbox{as $|{\vec \xi}| \to \infty$} ,
\end{align} \label{eqn:transproblem1var}
\end{subequations}
respectively, where ${\vec \Theta}^{(0)} ({\vec u})$ is a real vector field and ${\vec \Theta}^{(1)}({\vec u})$ is a complex vector field.
\end{lemma}

\begin{proof}
The proof is analogous to Lemma 1 of~\cite{ledgerlionheart2016}.
\end{proof}
We now consider the symmetry of the forms ${{N}}^{0}  ({\vec u},{\vec v}) $, $M({\vec u},{\vec v})$ and $M({\vec u},{\vec v})-N^0({\vec u},{\vec v}) = {N}^{\sigma_*} ({\vec u},{\vec v}) - C^{\sigma_*} ({\vec u},{\vec v})$ and, hence, the tensors ${\mathcal N}^0$, ${\mathcal M}$ and ${\mathcal M}-{\mathcal N}^0={\mathcal N}^{\sigma_*} - {\mathcal C}^{\sigma_*}$ for the inhomogeneous case. In the homogeneous case, the tensor ${\mathcal N}^0$ can be shown to be equivalent to the P\'olya-Szeg\"o tensor parameterised by the contrast in permeability, ${\mathcal  T}(\mu_r)$, (see Lemma 3 of~\cite{ledgerlionheart2016}). In addition, we have that ${\mathcal M}={\mathcal N}^0 + O(\omega)={\mathcal T}(\mu_r)  + O(\omega) $ as $\omega \to 0$ (by Theorem 9 of~\cite{ledgerlionheart2016}) and ${\mathcal N}^0$ is known to be real symmetric. Consequently, the tensor ${\mathcal N}^0$ provides an object characterisation for magnetostatic problems. In Lemma 4.4 of~\cite{ledgerlionheart2014}, we have previously shown that ${\mathcal M}$ is  complex symmetric and provides a characterisation of homogeneous conducting permeable objects.
 In order to extend these results to the inhomogeneous case, for square integrable complex vector fields ${\vec a}$, ${\vec b}$, we will use the notation   
\begin{align}
\left <{\vec a},{\vec b}  \right >_{L^2({\vec B} )}= \overline{\left <{\vec b},{\vec a} \right >_{L^2({\vec B} )}} := \int_{\vec B} {\vec a} \cdot \overline{\vec b} \dif {\vec \xi} \label{eqn:l2innerprod} ,
\end{align}
to denote the $L^2$ inner product over ${\vec B}$, where the overbar denotes the complex conjugate. This reduces to $\left <{\vec a},{\vec b}  \right >_{L^2({\vec B})} =\left <{\vec b},{\vec a} \right  >_{L^2({\vec B})}$ if ${\vec a}$, ${\vec b}$ are  square integrable real vector fields.  Hence, $\| {\vec u} \|_{L^2({\vec B})} := \left < {\vec u} , {\vec u} \right >_{L^2({\vec B})}^{1/2}$ is the $L^2$ norm of ${\vec u}$ over ${\vec B}$. We also define  $\|{\vec u}\|_{W(c,{\vec B})}: =\left <c {\vec u} , {\vec u}  \right >_{L^2({\vec B})}^{1/2}$, for a piecewise constant $c> 0$ in ${\vec B}$, as a weighted $L^2$ norm of ${\vec u}$ over ${\vec B}$.
The following theorem reveals insights into ${N}^0({\vec u},{\vec v})$ for inhomogeneous objects.

\begin{theorem}\label{thm:formsn0}
 ${ N}^0 ({\vec u}, {\vec v}): {\mathbb R}^3 \times {\mathbb R}^3 \to {\mathbb R} $ is a symmetric bilinear form on real vectors that can be expressed as
\begin{align}
{N}^0  ({\vec u}, {\vec v}) 
  =&\alpha^3  \left <   \left ( 1- \mu_r^{-1} \right )   {\vec u}, {\vec v} \right >_{L^2({\vec B} )}+  \frac{\alpha^3}{4} \left < \tilde{\mu}_r^{-1}  \nabla  \times \tilde{\vec \Theta}^{(0)}({\vec u})  , \nabla  \times \tilde{\vec \Theta}^{(0)} ({\vec v}) \right >_{L^2({\vec B}\cup{\vec B}^c )} 
\label{eqn:n0form2},  
\end{align}
and also defines an inner product provided that $\mu_r ({\vec \xi}) \ge 1$ for ${\vec \xi} \in {\vec B}$. In the above,
$\tilde{\vec \Theta}^{(0)} ({\vec u}):= {\vec \Theta}^{(0)} ({\vec u})- {\vec u} \times {\vec \xi}$ is a real vector field, which satisfies the transmission problem
\begin{subequations}
\begin{align} 
\nabla_\xi \times \tilde{\mu}_r^{-1} \nabla_\xi \times \tilde{\vec \Theta}^{(0)} & ={\vec 0}  && \hbox{in ${\vec B}  \cup {\vec B}^c$}  , \\
\nabla_\xi \cdot \tilde{\vec \Theta}^{(0)} &= 0  && \hbox{in ${\vec B} \cup {\vec B}^c$} , \\
\left [ \tilde{\vec \Theta}^{(0)} \times {\vec n} \right ]_\Gamma   &= {\vec 0}   && \hbox{on $\Gamma$} ,\\ 
\,  \left [   \tilde{\mu}_r^{-1}   \nabla_\xi  \times \tilde{\vec \Theta}^{(0)}  \times {\vec n} \right ]_\Gamma  & = -2 [\tilde{\mu}_r^{-1}]_\Gamma {\vec u} \times {\vec n} &&
\hbox{on $\Gamma$} , \\ 
\tilde{\vec \Theta}^{(0)}  & = O(|{\vec \xi} |^{-1})  && \hbox{as $|{\vec \xi}| \to \infty$} .
\end{align} \label{eqn:transproblem0til}
\end{subequations}
\end{theorem}

\begin{proof}
We first rewrite $N^0({\vec u},{\vec v}) $ as 
\begin{align}
N^0({\vec u},{\vec v}) 
=& \alpha^3  \int_{{\vec B}}   \left ( 1- \mu_r^{-1}  \right )  \left ( {\vec u}\cdot  {\vec v}+ \frac{1}{2} {\vec u} \cdot  \nabla  \times \tilde{\vec \Theta}^{(0)}({\vec v}) \right ) \dif {\vec \xi}  , \nonumber
\end{align}
where we have used $ {\vec \Theta}^{(0)}({\vec u})= \tilde{\vec \Theta}^{(0)} ({\vec u})+ {\vec u} \times {\vec \xi}$ in (\ref{eqn:definen0var}). The transmission problem for $\tilde{\vec \Theta}^{(0)} ({\vec u})$ is also easily derived. 

To obtain (\ref{eqn:n0form2}), we notice that
\begin{align}
\int_{{\vec B}}   (1- \mu_r^{-1})   {\vec u} \cdot  \nabla  \times \tilde{\vec \Theta}^{(0)} ({\vec v}) \dif {\vec \xi}  =  & \sum_{n=1}^N  (1- (\mu_r(B^{(n)}))^{-1}) \int_{B^{(n)}} 
{\vec u} \cdot  \nabla  \times \tilde{\vec \Theta}^{(0)} ({\vec v})\dif {\vec \xi} 
\nonumber \\
=& \sum_{n=1}^N  (1- (\mu_r(B^{(n)}))^{-1}) \int_{\partial B^{(n)}} 
{\vec u} \cdot  {\vec n}  \times \tilde{\vec \Theta}^{(0)} ({\vec v})\dif {\vec \xi},  \nonumber 
\end{align}
where $\mu_r(B^{(n)})= \mu_r ({\vec \xi}) $ with ${\vec \xi} \in B^{(n)}$.
Then, using $[{\vec n} \times \tilde{\vec \Theta}^{(0)} ({\vec v})]_{\partial B^{(n)}\cup\partial B^{(m)}}  = {\vec 0}$ for $n,m=1,\ldots,N$, $n\ne m$, it follows that
\begin{align}
\int_{{\vec B}}   (1- \mu_r^{-1})   {\vec u} \cdot  \nabla  \times \tilde{\vec \Theta}^{(0)} ({\vec v}) \dif {\vec \xi}  
=& \sum_{n,m =1, n \ne m}^N
 \int_{\partial B^{(n)} \cap \partial B^{(m)}}   [ {\mu}_r^{-1} ]_{\partial B^{(n)} \cap \partial B^{(m)}}  {\vec u} \cdot  {\vec n}  \times \tilde{\vec \Theta}^{(0)} ({\vec v})  
\dif {\vec \xi} \nonumber \\
&+\int_{\partial {\vec B} }   [\tilde{\mu}_r^{-1} ]_{\partial {\vec B}}   {\vec u} \cdot  {\vec n}^-  \times \tilde{\vec \Theta}^{(0)} ({\vec v}) 
\dif {\vec \xi}  \nonumber \\
= &  - \sum_{n,m =1, n \ne m}^N
 \int_{\partial B^{(n)} \cap \partial B^{(m)}}   [{\mu}_r^{-1} ]_{\partial B^{(n)} \cap \partial B^{(m)}}    \tilde{\vec \Theta}^{(0)} ({\vec v})\cdot {\vec n} \times {\vec u}
\dif {\vec \xi} \nonumber \\
&-\int_{\partial {\vec B} }   [\tilde{\mu}_r^{-1} ]_{\partial {\vec B}}    \tilde{\vec \Theta}^{(0)} ({\vec v}) \cdot {\vec n}^- \times {\vec u}
\dif {\vec \xi}  \nonumber .
\end{align}
By application of the transmission conditions in (\ref{eqn:transproblem0til}) and integration by parts, this becomes
\begin{align}
 -\frac{1}{2}   & \left ( - \sum_{n,m =1, n \ne m}^N
 \int_{\partial B^{(n)} \cap \partial B^{(m)}}    ( [  {\vec n}  \times \mu_r^{-1}   \nabla \times  \tilde{\vec \Theta}^0({\vec u}) ]_{\partial B^{(n)} \cap \partial B^{(m)}}  \cdot \tilde{\vec \Theta}^{(0)} ({\vec v}) 
\dif {\vec \xi} \right .  \nonumber \\
&  \left . + \int_{\partial {\vec B}}  {\vec n}^+   \times   \nabla \times  \tilde{\vec \Theta}^0({\vec u})  \cdot \tilde{\vec \Theta}^0({\vec v}) |_+ \dif {\vec \xi} + \int_{\partial {\vec B}}  {\vec n}^- \times \mu_r^{-1} \nabla \times \tilde{\vec \Theta}^0({\vec u}) \cdot \tilde{\vec \Theta}^0 ({\vec v})  |_- \dif {\vec \xi} \right ) \nonumber \\
& = \frac{1}{2} \left (  \int_{{\vec B}}  \mu_r^{-1} \nabla \times \tilde{\vec \Theta}^0 ({\vec u}) \cdot \nabla \times \tilde{\vec \Theta}^0 ({\vec v})  \dif {\vec \xi} +  \int_{{\vec B}^c}  \nabla \times \tilde{\vec \Theta}^0({\vec u})  \cdot \nabla \times \tilde{\vec \Theta}^0 ({\vec v})  \dif {\vec \xi}  \right ) \nonumber ,
\end{align}
from which  (\ref{eqn:n0form2})   immediately follows. We observe, from (\ref{eqn:n0form2}), and the linearity of the transmission problem  (\ref{eqn:transproblem0til}), that $N^0({\vec u},{\vec v})= N^0({\vec v},{\vec u})$ $\forall {\vec u},{\vec v}\in {\mathbb R}^3$, $N^0({\vec u}+{\vec w},{\vec v})= N^0({\vec u},{\vec v}) +N^0({\vec w},{\vec v})$ $\forall {\vec u},{\vec v}, {\vec w}\in {\mathbb R}^3$ and $N^0(c{\vec u},d{\vec v})=cd N^0({\vec u},{\vec v})$ $\forall {\vec u},{\vec v}\in {\mathbb R}^3$ and $c,d\in {\mathbb R}$ . Thus,  $N^0({\vec u},{\vec v}): {\mathbb R}^3 \times {\mathbb R}^3 \to {\mathbb R}$  is a symmetric bilinear form. Provided that $\mu_r ({\vec \xi}) \ge 1$ for ${\vec \xi} \in {\vec B}$ then $N^0({\vec u},{\vec u})\ge 0$ and $N^0({\vec u},{\vec u}) =0$ only if ${\vec u}={\vec 0}$, hence, $N^0({\vec u},{\vec v})$ defines an inner product. 
\end{proof}

\begin{corollary} \label{coll:diagonalcoeffn0}
It immediately follows from Theorem~\ref{thm:formsn0} that  ${\mathcal N}^{0}=N^0({\vec e}_i,{\vec e}_j) {\vec e}_i \otimes {\vec e}_j $ is a symmetric tensor extending the known result for a homogenous object proved in Lemma 1 of~\cite{ledgerlionheart2016}. In particular, the diagonal coefficients of the associated tensor ${\mathcal N}^{0}$  are
\begin{align}
({\mathcal N}^0)_{ii}  =  N^0({\vec e}_i,{\vec e}_i)=&
 \alpha^3  \int_{{\vec B}}  \left ( 1- \mu_r^{-1}  \right )    \dif {\vec \xi} \nonumber \\
&+ \frac{\alpha^3}{4} \left   (    \| \nabla \times \tilde{\vec \Theta}^{(0)} ({\vec e}_i) \|_{L^2({\vec B}^c)}^2 +   \| \nabla \times \tilde{\vec \Theta}^{(0)}({\vec e}_i)  \|_{W(\mu_r^{-1},{\vec B})}^2  \right )  , \nonumber
\end{align}
where the repeated index $i$ does not imply summation.
 In addition, we see that ${\mathcal N}^0_{ii} > 0$ provided that $\mu_r({\vec \xi})> 1$ for ${\vec \xi} \in {\vec B}$.
\end{corollary}

The following result provides further insights in to $N^0 ({\vec u}, {\vec v}) $ when the object is homogeneous:
\begin{lemma}\label{lemma:formsn0h}
For the homogeneous case, where ${\vec B}$ becomes $B$,  ${N}^0({\vec u}, {\vec v})$  can also be expressed in the following alternative forms
\begin{subequations}
\begin{align}
{ N}^0({\vec u}, {\vec v})  = &  \frac{\alpha^3}{4}  \int_\Gamma (1- \mu_r^{-1}) \nabla \times ({\vec u} \times {\vec \xi}) \cdot {\vec n}^- \times {\vec \Theta}^{(0)} ({\vec v}) \dif {\vec \xi} \label{eqn:n0form1h},  \\
 = &  -\frac{\alpha^3}{4}  \int_\Gamma ( \mu_r -1 )   ({\vec u} \times {\vec \xi}) \cdot {\vec n}^- \times \nabla \times  {\vec \Theta}^{(0)} ({\vec v}) |_+ \dif {\vec \xi} \label{eqn:n0form2h} , \\
 = & \frac{\alpha^3}{4} \frac{\mu_r+1}{\mu_r -1}  \left ( {\vec u} \cdot {\vec v} |B| + \left <  {\mu}_r^{-1} \nabla \times {\vec \Theta}^{(0)} ({\vec u}), \nabla \times {\vec \Theta }^{(0)} ({\vec v}) \right >_{L^2({\vec B}) } \right . 
 \nonumber\\
&\left .  +  \left <   \nabla \times \tilde{\vec \Theta}^{(0)} ({\vec u}), \nabla \times \tilde{\vec \Theta }^{(0)} ({\vec v}) \right >_{L^2({\vec B}^c)} \right ) 
\label{eqn:n0form3h} ,
\end{align}
\end{subequations}
where $\mu_r= \mu_* /\mu_0$ is now a constant.
\end{lemma}

\begin{proof}
To obtain (\ref{eqn:n0form1h}), we replace ${\vec B}$ by $B$ and transform the volume integral over ${ B}$  in (\ref{eqn:definen0var}) to a surface integral over $\Gamma=\partial B$ and use $2 {\vec u} = \nabla \times ({\vec u} \times {\vec \xi})$.

As $\mu_r$ is constant in ${B}$ for this case, then, to obtain  (\ref{eqn:n0form2h}), we subtract the following from (\ref{eqn:definen0var}) 
\begin{align}
0 =&  \frac{\alpha^3}{4}  \int_{{B}} ( 1- \mu_r^{-1} )  \nabla \times \nabla \times {\vec \Theta}^{(0)} ({\vec v})\cdot {\vec u} \times {\vec \xi} \dif {\vec \xi} \nonumber \\
= & \frac{\alpha^3}{4}  \left (  \int_{{ B}} ( 1- \mu_r^{-1} ) \nabla \cdot  (\nabla \times {\vec \Theta}^{(0)} ({\vec v})  \times ({\vec u} \times {\vec \xi} ))  \dif {\vec \xi} +
2 \int_{{ B}} ( 1- \mu_r^{-1} ) {\vec u} \cdot \nabla \times {\vec \Theta}^{(0)}({\vec v})  \dif {\vec \xi} \right ) \nonumber.
\end{align}
The result then follows by transforming the remaining volume integral to a surface integral over $\Gamma$ and using the transmission condition in (\ref{eqn:transproblem0var}).

For the third form, we use  (\ref{eqn:n0form1h}) and  (\ref{eqn:n0form2h}) to give
\begin{align}
(1+\mu_r^{-1} ) { N}^0({\vec u},{\vec v}) =& \frac{\alpha^3}{4} ( 1- \mu_r^{-1}) \left ( \int_\Gamma \nabla \times ({\vec u} \times {\vec \xi} ) \cdot {\vec n}^- \times {\vec \Theta} ^{(0)} ({\vec v})  \dif {\vec \xi} \right . \nonumber \\
& \left . - \int_\Gamma {\vec u} \times {\vec \xi} \cdot {\vec n}^- \times \nabla \times {\vec \Theta}^{(0)} ({\vec v}) |_+ \dif {\vec \xi}
\right ) \nonumber .
\end{align}
Then, writing ${\vec \Theta}^{(0)}({\vec v}) = \tilde{\vec \Theta}^{(0)}({\vec v})+ {\vec v} \times {\vec \xi}$, so that $\nabla \times {\vec \Theta}^{(0)} ({\vec v}) = \nabla \times \tilde{\vec \Theta}^{(0)}({\vec v})+ 2 {\vec v}$, we have
\begin{align}
 \int_\Gamma  ({\vec u} \times {\vec \xi} ) \cdot {\vec n}^- \times  \nabla \times  {\vec \Theta}^{(0)} ({\vec v}) |_+ \dif {\vec \xi} = &  \int_\Gamma {\vec u} \times {\vec \xi} \cdot {\vec n}^- \times \nabla \times \tilde{\vec \Theta}^{(0)} ({\vec v}) |_+ \dif {\vec \xi} + 2 \int_\Gamma {\vec n}^- \cdot {\vec v} \times ( {\vec u} \times {\vec \xi}) \dif {\vec \xi} \nonumber \\
 =&   \int_\Gamma {\vec u} \times {\vec \xi} \cdot {\vec n}^- \times \nabla \times \tilde{\vec \Theta}^{(0)} ({\vec v}) |_+ \dif {\vec \xi} - 4 |B| {\vec u}\cdot {\vec v} \nonumber .
\end{align}
This means that
\begin{align}
4 \left ( \frac{\mu_r +1 }{\mu_r-1}  \frac{{N}^0({\vec u},{\vec v})}{\alpha^3} -|B|  {\vec u}\cdot {\vec v}  \right ) = & -\int_\Gamma {\vec \Theta}^{(0)}({\vec v}) \cdot {\vec n}^- \times( \nabla \times (  {\vec u} \times {\vec \xi}))|_+ \dif {\vec \xi} \nonumber \\
&
-\int_\Gamma {\vec u} \times {\vec \xi} \cdot {\vec n}^- \times \nabla \times \tilde{\vec \Theta}^{(0)} ({\vec v}) |_+ \dif {\vec \xi}\nonumber \\
= & - \mu_r^{-1} \int_\Gamma {\vec \Theta}^{(0)} ({\vec v}) \cdot {\vec n}^- \times \nabla \times {\vec \Theta}^{(0)} ({\vec u})
 |_- \dif {\vec \xi}\nonumber \\
 & + \int_\Gamma \tilde{\vec \Theta}^{(0)} ({\vec u}) \cdot {\vec n}^-  \times \nabla \times \tilde{\vec \Theta}^{(0)}({\vec v}) |_+ \dif {\vec \xi} \nonumber ,
  \end{align}
which follows from  first using ${\vec n} \times \nabla \times  ( {\vec u} \times {\vec \xi})  = -{\vec n} \times \nabla \times  \tilde{\vec \Theta}^{(0)} ({\vec u}) |_+ + \mu_r^{-1} {\vec n} \times \nabla \times  {\vec \Theta}^{(0)} ({\vec u}) |_-$, then using  ${\vec n} \times  {\vec \Theta}^{(0)} ({\vec u}) |_+ =  {\vec n} \times {\vec \Theta}^{(0)} ({\vec u}) |_- ={\vec n} \times( {\vec u} \times {\vec \xi}) +  {\vec n} \times  \tilde{\vec \Theta}^{(0)} ({\vec u})  |_+$  and simplifying.  The final result follows from integration by parts and using the far field decay conditions of $\tilde{\vec \Theta}^{(0)} ({\vec u}) $.
\end{proof}

\begin{corollary} \label{coll:diagonalcoeffn02}
It immediately follows from Lemma~\ref{lemma:formsn0h} that the diagonal coefficients of ${\mathcal N}^{0}$ for homogeneous  case  are
\begin{align}
({\mathcal N}^0)_{ii} & = \frac{\alpha^3}{4} \frac{\mu_r+1}{\mu_r -1}  \left ( |B| +  \| \nabla \times  \tilde{\vec \Theta}^{(0)} ({\vec e}_i) \|_{L^2(B^c)}^2 +  {\mu}_r^{-1} \| \nabla \times {\vec \Theta}^{(0)} ({\vec e}_i) \|_{L^2(B)}^2  \right )  , \nonumber
\end{align}
where the repeated index $i$ does not imply summation, and, hence, ${\mathcal N}^0_{ii}>0$ if $\mu_r >1$ and ${\mathcal N}^0_{ii}< 0 $ if $\mu_r <1 $.
\end{corollary}

We now consider the symmetry of the bilinear forms $M({\vec u},{\vec v})$ and $M({\vec u},{\vec v})-N^0({\vec u},{\vec v}) = {\vec N}^{\sigma_*} ({\vec u},{\vec v}) - C^{\sigma_*} ({\vec u},{\vec v})$ and, hence, the symmetry of the tensors ${\mathcal M}$ and ${\mathcal M}-{\mathcal N}^0={\mathcal N}^{\sigma_*} - {\mathcal C}^{\sigma_*}$.

\begin{theorem} \label{thm:symmfull} 
$M({\vec u},{\vec v}): {\mathbb R}^3 \times {\mathbb R}^3 \to {\mathbb C}$ is a symmetric bilinear form on real vectors, which can be expressed as
\begin{align}
M({\vec u},{\vec v}) =& \frac{\alpha^3}{4} \left (\int_{\vec B} \frac{1}{\im \nu}  \nabla \times \mu_r^{-1} \nabla \times {\vec \Theta}({\vec u}) \cdot \nabla \times \mu_r^{-1} \nabla \times {\vec \Theta}({\vec v}) \dif {\vec \xi}\right . \nonumber \\
 &+\int_{\vec B} \left ( 1- \mu_r^{-1} \right ) \left (
4{\vec u} \cdot {\vec v} + 2 \nabla \times {\vec \Theta}({\vec u})\cdot {\vec v} + 2 {\vec u} \cdot  \nabla \times {\vec \Theta} ({\vec v}) \right ) \dif {\vec \xi} \nonumber\\
&\left .  -\int_{{\vec B} \cup {\vec B}^c  } \tilde{\mu}_r^{-1}  \nabla \times {\vec \Theta}({\vec u}) \cdot  \nabla \times {\vec \Theta}({\vec v}) \dif {\vec \xi} \right )
\label{eqn:symmform} ,
\end{align}
and $N^{\sigma_*}({\vec u},{\vec v})-C^{\sigma_*}({\vec u},{\vec v})= M({\vec u},{\vec v})- N^0({\vec u},{\vec v}): {\mathbb R}^3 \times {\mathbb R}^3 \to {\mathbb C}$ is also a symmetric bilinear form on real vectors.
\end{theorem}

\begin{proof}
The first part of the proof applies similar arguments to Lemma 4.4 of~\cite{ledgerlionheart2014},  which showed that ${\mathcal M}$ is a symmetric tensor for a homogenous object. Here, we will apply these arguments to the form $M({\vec u},{\vec v})$ and consider an object with possibly inhomogeneous materials. We begin by noting from (\ref{eqn:mcheckinv}a) that
\begin{align}
-\frac{4C ({\vec u},{\vec v})}{\alpha^3} =& \int_{\vec B} \im \nu ( {\vec \Theta}({\vec v}) + {\vec v} \times {\vec \xi}) \cdot ( {\vec u } \times {\vec \xi}) \dif {\vec \xi} \nonumber \\
= &\int_{\vec B} \nabla \times \mu_r^{-1} \nabla \times {\vec \Theta}({\vec v}) \cdot \left ( \frac{1}{\im \nu} \nabla \times \mu_r^{-1} \nabla \times {\vec \Theta}({\vec u}) - {\vec \Theta}({\vec u})
\right ) \dif {\vec \xi} \nonumber \\
= &\int_{\vec B} \frac{1}{\im \nu}  \nabla \times \mu_r^{-1} \nabla \times {\vec \Theta}({\vec v}) \cdot \nabla \times \mu_r^{-1} \nabla \times {\vec \Theta}({\vec u}) \dif {\vec \xi} \nonumber \\
 & -\int_{\vec B} \nabla \times \mu_r^{-1} \nabla \times {\vec \Theta}({\vec v}) \cdot {\vec \Theta}({\vec u}) \dif {\vec \xi} , \nonumber
\end{align}
by use of the transmission problem (\ref{eqn:transproblemthetainv}). Next, by integration by parts, we have
\begin{align}
\int_{\vec B} &\nabla \times \mu_r^{-1} \nabla \times {\vec \Theta}({\vec v}) \cdot {\vec \Theta}({\vec u}) \dif {\vec \xi} = \int_{\partial {\vec B}} {\vec n}^-  \times \mu_r^{-1} \nabla \times {\vec \Theta} ({\vec v}) \cdot  {\vec \Theta}({\vec u})   |_- \dif {\vec \xi} \nonumber \\
&+\int_{\vec B} \mu_r^{-1} \nabla \times {\vec \Theta}({\vec u}) \cdot \nabla \times {\vec \Theta}({\vec v}) \dif {\vec \xi}
 - \sum_{n,m=1, n\ne m}^N \int_{\partial B^{(n)} \cup \partial B^{(m)}} [ {\vec n} \times \mu_r^{-1} \nabla \times {\vec \Theta}({\vec u}) ]_{\partial B^{(n)} \cap \partial B^{(m)}}  \cdot {\vec \Theta}({\vec u}) \dif {\vec \xi} \nonumber \\
&=  \int_{\partial {\vec B} }{\vec n}^- \times \nabla \times {\vec \Theta}({\vec v}) \cdot {\vec \Theta}({\vec u}) |_+ + 2[\tilde{\mu}_r^{-1} ]_{\partial {\vec B}} {\vec n}^- \times {\vec v} \cdot {\vec \Theta} ({\vec u}) \dif {\vec \xi} \nonumber \\
&+\int_{\vec B} \mu_r^{-1} \nabla \times {\vec \Theta}({\vec u}) \cdot \nabla \times {\vec \Theta}({\vec v}) \dif {\vec \xi}
 +2 \sum_{n,m=1, n\ne m}^N \int_{\partial B^{(n)} \cup \partial B^{(m)}} [  \mu_r^{-1} ]_{\partial B^{(n)} \cap \partial B^{(m)}}   {\vec n} \times {\vec v} \cdot {\vec \Theta}({\vec u}) \dif {\vec \xi} \nonumber \\
 &=  \int_{\partial {\vec B} }{\vec n}^- \times \nabla \times {\vec \Theta}({\vec v}) \cdot {\vec \Theta}({\vec u}) |_+ + 2[\tilde{\mu}_r^{-1} ]_{\partial {\vec B}} {\vec n}^- \times {\vec v} \cdot {\vec \Theta} ({\vec u}) \dif {\vec \xi} \nonumber \\
&+\int_{\vec B} \mu_r^{-1} \nabla \times {\vec \Theta}({\vec u}) \cdot \nabla \times {\vec \Theta}({\vec v}) \dif {\vec \xi}
 +2 \sum_{n=1}^N (1-(\mu_r(B^{(n)}))^{-1}) \int_{\partial B^{(n)} \setminus \partial {\vec B}} {\vec n}^- \times {\vec v} \cdot {\vec \Theta}({\vec u}) \dif {\vec \xi} \nonumber ,
\end{align}
since $[{\vec n} \times {\vec \Theta} ({\vec v})]_{\partial B^{(n)}\cup\partial B^{(m)}}  = {\vec 0}$ for $n,m=1,\ldots,N$, $n\ne m$. It then follows that
\begin{align}
\int_{\vec B} &\nabla \times \mu_r^{-1} \nabla \times {\vec \Theta}({\vec u}) \cdot {\vec \Theta}({\vec u}) \dif {\vec \xi} =\int_{{\vec B}  } {\mu}_r^{-1}  \nabla \times {\vec \Theta}({\vec u}) \cdot  \nabla \times {\vec \Theta}({\vec v}) \dif {\vec \xi} \nonumber \\
&+ \int_{{\vec B}^c } \nabla \times {\vec \Theta}({\vec u}) \cdot  \nabla \times {\vec \Theta}({\vec v}) \dif {\vec \xi} -2 \int_{\vec B} (1-\mu_r^{-1}) \nabla \times {\vec \Theta}({\vec u})\cdot {\vec v} \dif {\vec \xi}  \nonumber .
\end{align}
From the above, and $M({\vec u},{\vec v}) = N ({\vec u},{\vec v}) - C ({\vec u},{\vec v}) $, the result in (\ref{eqn:symmform}) immediately follows.
On consideration of (\ref{eqn:symmform}), and the linearity of the transmission problem (\ref{eqn:transproblem1var}), we see 
 that $M({\vec u},{\vec v}) = M({\vec v},{\vec u})$ $\forall {\vec u}, {\vec v} \in {\mathbb R}^3$,  $M({\vec u}+{\vec w},{\vec v}) = M({\vec u},{\vec v})+M({\vec w},{\vec v}) $ $\forall {\vec u}, {\vec v}, {\vec w} \in {\mathbb R}^3$ and  $M(c{\vec u},d{\vec v}) = cd  M({\vec u},{\vec v}) $ $\forall {\vec u}, {\vec v} \in {\mathbb R}^3$ and $\forall c,d \in {\mathbb R}$ and, thus, $M({\vec u},{\vec v}):{\mathbb R}^3 \times {\mathbb R}^3 \to {\mathbb C}$ is a symmetric bilinear form on real vectors. By using Theorem~\ref{thm:formsn0}, it follows that $M({\vec u},{\vec v})- N^0({\vec u},{\vec v}):{\mathbb R}^3 \times {\mathbb R}^3 \to {\mathbb C}$ is also a symmetric bilinear form on real vectors.
\end{proof}

\begin{corollary}
It immediately follows from Theorem~\ref{thm:symmfull} that ${\mathcal M}= M({\vec e}_i,{\vec e}_j) {\vec e}_i \otimes {\vec e}_j$ is a complex symmetric tensor and ${\mathcal M}-{\mathcal N}^0 = {\mathcal N}^{\sigma^*}-{\mathcal C}^{\sigma_*} =( N^{\sigma^*}({\vec e}_i, {\vec e}_j) - C^{\sigma^*}({\vec e}_i,{\vec e}_j) ){\vec e}_i \otimes {\vec e}_j$ is a complex symmetric tensor, extending the known results in Lemma 4.4
of~\cite{ledgerlionheart2014} and Lemma 1 of~\cite{ledgerlionheart2016} for a homogeneous object to the inhomogeneous case.
\end{corollary}
\begin{remark}
Note that the first and last terms in (\ref{eqn:symmform})  cannot be expressed in terms of the notation introduced in (\ref{eqn:l2innerprod}) since  ${\vec \Theta}({\vec u})$ and ${\vec \Theta}({\vec v})$ are complex valued and the integrands each lack a complex conjugate.
\end{remark}

\section{Limiting cases of ${\mathcal M}$} \label{sect:limitcase}
Recall that the asymptotic formula (\ref{eqn:asymp}) is valid for $\nu =O(1)$ as $\alpha \to 0$ and so care needs to be exercised when interpreting the limiting cases of ${\mathcal M}$. Still further, recall that the eddy current model (\ref{eqn:eddymodel}) is a low-frequency approximation of the Maxwell system and so the limit of $\nu \to \infty$ for fixed $\sigma_*$, $\alpha$ would break break both (\ref{eqn:asymp}) and (\ref{eqn:eddymodel}). The case of a perfect conductor with sufficiently small $\omega$, $\alpha$  and $\sigma_*\to \infty$ is permitted by the eddy current model, provided topological requirements on ${\vec B}$ are satisfied~\cite{ammaribuffa2000,schmidt2008}, but invalidates (\ref{eqn:asymp}) as we still have $\nu \to \infty$.
In the following, we compute $M ({\vec u},{\vec v};\nu)$ when $\nu =0$ and $\nu \to \infty$. From the former, we can deduce ${\mathcal M}(0)$, which provides a magnetostatic characterisation of ${\vec B}$ for a permeable object, and, from the latter, we can obtain
 ${\mathcal M}(\infty)$, which we denote as the characterisation of a perfectly conducting object. The coefficients of ${\mathcal M}(\infty)$ can not be substituted in to (\ref{eqn:asymp}) and, instead, should be viewed as the limiting characterisation of ${\vec B}$ provided by (\ref{eqn:asymp})   as $\sigma_*\to \infty$ and $\alpha \to 0$.   


\begin{lemma} \label{lemma:limitcase}
The limiting cases of $M ({\vec u},{\vec v};\nu)$ when $\nu =0$ and $\nu \to \infty$ are
\begin{align}
{ M}({\vec u},{\vec v};0) ={ N}^0 ({\vec u},{\vec v}) =&  \alpha^3  \left <  \left ( 1- \mu_r^{-1} \right )   {\vec u},  {\vec v} \right >_{L^2({\vec B})}  \nonumber\\
&+ \frac{\alpha^3}{4} \left <  \tilde{\mu}_r^{-1}  \nabla  \times \tilde{\vec \Theta}^{(0)}({\vec u}) , \nabla  \times \tilde{\vec \Theta}^{(0)} ({\vec v})  \right >_{L^2({\vec B}\cup{\vec B}^c )} \label{eqn:limit0} ,\\
 { M}({\vec u},{\vec v};\infty)  =& - \alpha^3 \left <  {\vec u} , {\vec v} \right >_{L^2({\vec B})} - \frac{\alpha^3}{4} \left < \nabla \times {\vec \Theta}^{(\infty)}({\vec u}) , \nabla \times {\vec \Theta}^{(\infty)}({\vec v}) \right  >_{L^2({\vec B}^c)} \label{eqn:limitinf} ,
\end{align}
where $\tilde{\vec \Theta}^{(0)}({\vec u})$ is the real vector field solution to (\ref{eqn:transproblem0til}) and ${\vec \Theta}^{(\infty)}({\vec u})$ is the real vector field solution to
\begin{subequations} \label{eqn:tpthetainf}
\begin{align}
\nabla_\xi \times \nabla_\xi \times {\vec \Theta}^{(\infty)}({\vec u} ) = & {\vec 0} &&\text{in ${\vec B}^c$}, \\
\nabla_\xi \cdot  {\vec \Theta}^{(\infty)} = & 0 &&\text{in ${\vec B}^c$}, \\
{\vec n} \times \nabla_\xi \times {\vec \Theta}^{(\infty)}({\vec u} ) |_+ = & - 2{\vec n} \times {\vec u} &&\text{on $\partial {\vec B}$}, \\
{\vec \Theta}^{(\infty)}   = & O(|{\vec \xi}|^{-1}) && \text{as $|{\vec \xi}| \to \infty$}.
\end{align}
\end{subequations}
\end{lemma}

\begin{proof}
Using Lemma~\ref{lemma:msplitting1}, we immediately establish that ${C}({\vec u},{\vec v})$ vanishes for $\nu=0$ and, from (\ref{eqn:transproblem1var}), find that ${\vec \Theta}^{(1)}({\vec u}) ={ \vec 0}$ for $\nu =0$. Thus,  ${ M}({\vec u},{\vec v};0) ={ N}^0 ({\vec u},{\vec v})$ and we quote the form of ${ N}^0 ({\vec u},{\vec v})$ given in Theorem~\ref{thm:formsn0}.

To obtain ${ M}({\vec u},{\vec v};\infty)$, we see, from (\ref{eqn:transproblemthetainv}), that ${\vec \Theta}({\vec u}) = - {\vec u} \times {\vec \xi}$ in ${\vec B}$ when $\nu\to \infty$. Using Theorem~\ref{thm:symmfull} for this case, we have
\begin{align}
M({\vec u},{\vec v}; \infty)=& \frac{\alpha^3}{4} \left (- 4 \int_{\vec B} \mu_r^{-1} {\vec u}\cdot {\vec v} \dif {\vec \xi} \right . \nonumber \\
& -\int_{{\vec B^c}  }  \nabla \times {\vec \Theta}({\vec u}) \cdot  \nabla \times {\vec \Theta}({\vec v}) \dif {\vec \xi} 
\nonumber \\
&\left . +\int_{\vec B} \left ( 1- \mu_r^{-1} \right ) \left (
4{\vec u} \cdot {\vec v} - 4 {\vec u} \cdot  {\vec v} - 4 {\vec u} \cdot  {\vec v}  \right ) \dif {\vec \xi} \right ), \nonumber
\end{align}
which immediately simplifies to (\ref{eqn:limitinf}) by realising that ${\vec \Theta}({\vec u})$ becomes ${\vec \Theta}^{(\infty)} ({\vec u})\in {\mathbb R}^3$ when $\nu \to \infty$. This is because,  on the interior interfaces of ${\vec B}$,  we observe, from (\ref{eqn:transproblemthetainv}), that $[ {\vec n} \times {\mu}^{-1} \nabla \times {\vec \Theta}({\vec u})]_{\partial B^{(m)} \cup \partial B^{(n)}}  = - 2 [ {\mu}^{-1} ]_{\partial B^{(m)} \cup \partial B^{(n)}} {\vec n} \times {\vec u}$, $n,m =1,\ldots,N$, $n\ne m$ is now automatically satisfied since $\nabla \times {\vec \Theta}({\vec u}) = -2 {\vec u}$ in ${\vec B}$ and, on $\partial {\vec B}$, the jump condition 
$[ {\vec n} \times \tilde{\mu}^{-1} \nabla \times {\vec \Theta}({\vec u})]_{\partial {\vec B}}  = - 2 [ \tilde{\mu}^{-1} ]_{\partial {\vec B}} {\vec n} \times {\vec u}$
simplifies to the boundary condition ${\vec n} \times \nabla \times {\vec \Theta}^{(\infty)}({\vec u} ) |_+ =  - 2{\vec n} \times {\vec u}$.
\end{proof}

\begin{corollary}\label{corollary:n0tops}
For an object with homogeneous $\mu_*$,  ${\mathcal M}(0)= M({\vec e}_i,{\vec e}_j,0) {\vec e}_i \otimes {\vec e}_j  = {\mathcal N}^0=N^0({\vec e}_i , {\vec e}_j ) {\vec e}_i \otimes {\vec e}_j$ is just the P\'oyla-Szeg\"o tensor parameterised by the contrast in permeability ${\mathcal T}(\mu_r)$, independent of the object's topology.
\end{corollary}
\begin{proof}
The result follows from Lemma~\ref{lemma:limitcase}  and by applying similar arguments to Lemma 3 of~\cite{ledgerlionheart2016}. The latter discusses the contractibility of loops associated with holes in the object  so that the result is independent of the object's topology.
\end{proof}

\begin{corollary}
In the case where ${\vec B}$ becomes a single object $B$, with Betti numbers such that $\beta_1(B) =\beta_1(B^c) =0$, then
$({\mathcal M} (\infty))_{ij} ={ M}({\vec e}_i,{\vec e}_j;\infty) $ can be expressed as
\begin{subequations}
\begin{align}
({\mathcal M} (\infty))_{ij} =  &- \alpha^3|B|  \delta_{ij} - \alpha^3 \int_{B^c} \nabla \psi_i\cdot \nabla \psi_j \dif {\vec \xi}   \\
= & -\alpha^3 |B| \delta_{ij} - \alpha^3 \int_\Gamma {\vec n}^+ \cdot {\vec e}_j  \psi_i \dif {\vec \xi}   \\
= & -\alpha^3 |B| \delta_{ij} + \alpha^3 \int_\Gamma {\vec n}^- \cdot \nabla \psi_i \xi_j \dif {\vec \xi} = ({ \mathcal T} (0) )_{ij}, 
\end{align}\label{eqn:minfdifferentforms}
\end{subequations}
and coincides with the coefficients of the P\'oyla-Szeg\"o tensor parameterised by  0, $({ \mathcal T} (0) )_{ij}$. In the above,
 $\delta_{ij}$ denotes the Kronecker delta and $\psi_i({\vec \xi})$ solves
\begin{subequations}
\begin{align}
\nabla^2 \psi_i & = 0 && \text{in $B^c$}, \\
{\vec n} \cdot \nabla \psi_i|_+  & = {\vec n} \cdot \nabla \xi_i && \text{on $\Gamma$}, \\
\psi_i & = O(|{\vec\xi}|) &&\text{as $|{\vec\xi}|\to \infty$}.
\end{align} \label{eqn:psitranprob}
\end{subequations}
\end{corollary}
\begin{proof}
For the case of  $\beta_1(B) =\beta_1(B^c) =0$, it follows from  Lemma~\ref{lemma:limitcase} that we can set $\nabla \times {\vec \Theta}^{(\infty)}({\vec e} _i)= - 2 \nabla \psi_i$ where $\psi_i$  solves (\ref{eqn:psitranprob}). Also, by applying integration by parts, the different forms of $({\mathcal M} (\infty))_{ij} ={ M}({\vec e}_i,{\vec e}_j;\infty) $ in (\ref{eqn:minfdifferentforms}) can be easily obtained. We see this coincides with $({ \mathcal T} (0) )_{ij}$ by comparing (\ref{eqn:minfdifferentforms}c) to the expression for
 ${\mathcal T}(c)$ given in (9) of~\cite{ledgerlionheart2016}  in the case where the contrast becomes $0$.
\end{proof}

\begin{remark}
For objects, with $\beta_1(B^c)\ne 1$ then $\nabla \times {\vec \Theta}^{(\infty)}({\vec e} _i)= - 2 \nabla \psi_i+{\vec h}_i$ in $B^c$ where ${\vec h}_i$ is a curl free function that is not a gradient with dimension $\text{dim}({\vec h}_i)= \beta_1(B^c)$. Unlike in Lemma 3 of~\cite{ledgerlionheart2016}, the loops $\gamma_k(B^c)$, $k=1,\ldots, \beta_1(B^c)$ associated with the holes passing through the object are no longer contractable and so ${\vec h}_i \ne {\vec 0}$ in this case.  Thus,
 $( {\mathcal M}  (\infty) )_{ij}$ does not coincide with $({\mathcal T} (0))_{ij}$ for objects with holes and the more general form $ ({\mathcal M} (\infty) )_{ij}= M({\vec e}_i, {\vec e}_j; \infty)$  following from (\ref{eqn:limitinf}) must be used. Numerical examples illustrating this for single multiply connected objects with loops were presented in~\cite{ledgerlionheart2016}.
\end{remark}

\section{The energy functional associated with ${\mathcal M}$} \label{sect:enfun}
An important alternative representation of $M({\vec u},{\vec v})- N^0({\vec u},{\vec v}) = {N}^{\sigma_*} ({\vec u},{\vec v})-{C}^{\sigma_*}({\vec u},{\vec v})$ is provided in the following theorem.
\begin{theorem} \label{thm:realandimagpts}
The bilinear form  $ M({\vec u},{\vec v})- N^0({\vec u},{\vec v}) = {N}^{\sigma_*} ({\vec u},{\vec v})-{C}^{\sigma_*}({\vec u},{\vec v})$ can be written as $ {R}^{\sigma_*} ({\vec u},{\vec v})+  \im {I}^{\sigma_*}({\vec u},{\vec v})$ where ${R}^{\sigma_*} ({\vec u},{\vec v}):{\mathbb R}^3 \times {\mathbb R}^3\to {\mathbb R}$,  ${I}^{\sigma_*} ({\vec u},{\vec v}):{\mathbb R}^3 \times {\mathbb R}^3\to {\mathbb R}$ are the following symmetric bilinear forms  on real vectors 
\begin{subequations}
\begin{align}
{R}^{\sigma_*} ({\vec u},{\vec v})=&  \mathrm{Re}( {N}^{\sigma_*} ({\vec u},{\vec v})-{C}^{\sigma_*}({\vec u},{\vec v}))\nonumber \\
=&- \frac{\alpha^3}{4}    \int_{{\vec B}\cup{\vec B}^c}  \tilde{\mu}_r ^{-1}\nabla \times {\vec \Theta}^{(1)} {(\vec v}) \cdot  
\nabla \times \overline{{\vec \Theta}^{(1)}({\vec u}) } \dif {\vec \xi} \nonumber \\
=&- \frac{\alpha^3}{4}   \left <   \tilde{\mu}_r ^{-1}\nabla \times {\vec \Theta}^{(1)} {(\vec v}) ,
\nabla \times {{\vec \Theta}^{(1)}({\vec u}) }\right >_{L^2( {\vec B} \cup {\vec B}^c) }   \nonumber \\
=&- \frac{\alpha^3}{4}   \left <   \tilde{\mu}_r ^{-1}\nabla \times {\vec \Theta}^{(1)} {(\vec u}) ,
\nabla \times {{\vec \Theta}^{(1)}({\vec v}) }\right >_{L^2( {\vec B} \cup {\vec B}^c) }   
,\label{eqn:rform2}  \\
{ I}^{\sigma_*} ({\vec u},{\vec v}) = &  \mathrm{Im}( {N}^{\sigma_*} ({\vec u},{\vec v})-{C}^{\sigma_*}({\vec u},{\vec v}))\nonumber \\
=&\frac{\alpha^3}{4}   \int_{\vec B }\frac{1}{\nu}  \nabla \times {\mu}_r^{-1} \nabla \times {\vec \Theta}^{(1)}({\vec v})  \cdot \nabla \times \mu_r ^{-1} 
\nabla \times \overline{{\vec \Theta}^{(1)} ({\vec u}) }\dif {\vec \xi}  \nonumber \\
=&\frac{\alpha^3}{4}  \left < \frac{1}{\nu}  \nabla \times {\mu}_r^{-1} \nabla \times {\vec \Theta}^{(1)}({\vec v})  ,  \nabla \times \mu_r ^{-1} 
\nabla \times {{\vec \Theta}^{(1)} ({\vec u}) } \right >_{L^2({\vec B})} \nonumber \\
=&\frac{\alpha^3}{4}  \left < \frac{1}{\nu}  \nabla \times {\mu}_r^{-1} \nabla \times {\vec \Theta}^{(1)}({\vec u})  ,  \nabla \times \mu_r ^{-1} 
\nabla \times {{\vec \Theta}^{(1)} ({\vec v}) } \right >_{L^2({\vec B})}
.\label{eqn:iform2}
\end{align} 
\end{subequations}
 Additionally, 
$-{\mathcal R}^{\sigma_*} ({\vec u},{\vec v})$ and ${\mathcal I}^{\sigma_*} ({\vec u},{\vec v})$ define inner products on real vectors.  \end{theorem}

\begin{proof}
Using  the definitions in (\ref{eqn:definecconvar}) and (\ref{eqn:definenconvar}) and the transmission problem (\ref{eqn:transproblem1var}) we have
\begin{align}
({N}^{\sigma_*} -{{C}}^{\sigma_*} ) ({\vec u},{\vec v})  = &
 \frac{\im \alpha^3 }{4} {\vec u} \cdot \int_{\vec B} \nu {\vec \xi} \times  \frac{1}{\im \nu} \nabla \times \mu_r^{-1} \nabla \times {\vec \Theta}^{(1)} ({\vec v}) \dif {\vec \xi}\nonumber\\
 & + \frac{\alpha^3}{2} \int_{\vec B}   \left ( 1- \mu_r^{-1}  \right ) \left (
    {\vec u} \cdot  \nabla \times {\vec \Theta}^{(1)} ({\vec v}) \right ) \dif {\vec \xi} \nonumber \\
    = &  \frac{\alpha^3}{4} \int_{\vec B} \nabla \times \mu_r^{-1} \nabla \times {\vec \Theta}^{(1)} ({\vec v}) \cdot {\vec u} \times {\vec \xi} \dif {\vec \xi} \nonumber \\
    &+ \frac{\alpha^3}{2} \int_{\vec B} \left ( 1- \mu_r^{-1}  \right )  \left (
    {\vec u} \cdot  \nabla \times {\vec \Theta}^{(1)}({\vec v})  \right ) \dif {\vec \xi} \nonumber .
\end{align}
Then, using  ${\vec u } \times {\vec \xi}= \frac{1}{\im \nu} \nabla \times \mu_r^{-1} \nabla \times {\vec \Theta}^{(1)}({\vec u})  - ( {\vec \Theta}^{(0)} ({\vec u})+ {\vec \Theta}^{(1)} ({\vec u}) - {\vec u} \times {\vec \xi})$ in ${\vec B}$ we have, for ${\vec u}\in {\mathbb R}^3$,
\begin{align}
\overline{{\vec u} \times {\vec \xi}}= {\vec u} \times {\vec \xi} =\frac{\im}{\nu} \nabla \times \mu_r^{-1} \nabla \times \overline{{\vec \Theta}^{(1)} ({\vec u})} -( {\vec \Theta}^{(0)}({\vec u}) + \overline{{\vec \Theta}^{(1)}({\vec u})} - {\vec u} \times {\vec \xi}) \nonumber ,
\end{align}
in ${\vec B}$ since ${\vec \Theta}^{(0)}({\vec u})\in {\mathbb R}^3$. Thus, it follows that
\begin{align}
({N}^{\sigma_*} -{C}^{\sigma_*} )({\vec u},{\vec v})  = & \frac{\im  \alpha^3 }{4} \int_{\vec B} \frac{1}{\nu}  \nabla \times \mu_r^{-1} \nabla \times {\vec \Theta}^{(1)}({\vec v}) \cdot
 \nabla \times \mu_r^{-1} \nabla \times \overline{{\vec \Theta}^{(1)}({\vec u})} \dif {\vec \xi} \nonumber \\
 &- \frac{\alpha^3}{4} \int_{\vec B}  \nabla \times \mu_r^{-1} \nabla \times {\vec \Theta}^{(1)}({\vec v})  \cdot ( {\vec \Theta}^{(0)}({\vec u}) +\overline{ {\vec \Theta}^{(1)} ({\vec u}) }- {\vec u} \times {\vec \xi}) \dif {\vec \xi} \nonumber \\
 & +\frac{\alpha^3}{2} \int_{\vec B}   \left ( 1- \mu_r^{-1}  \right ) \left (
    {\vec u} \cdot  \nabla \times {\vec \Theta}^{(1)} ({\vec v}) \right ) \dif {\vec \xi} \nonumber .
 \end{align}
 Denoting the latter two terms by $-\frac{\alpha^3}{4}A_1$ and $\frac{\alpha^3}{2}A_2$, respectively, then, by integration by parts, we have
 \begin{align}
 A_1 = &\int_{\vec B} \mu_r^{-1} \nabla \times {\vec \Theta}^{(1)}({\vec v}) \cdot \nabla \times \overline{{\vec \Theta}^{(1)}({\vec u})} \dif {\vec \xi} + \int_{\vec B} \mu_r^{-1} \nabla \times {\vec \Theta}^{(1)}({\vec v}) \cdot \nabla \times {\vec \Theta}^{(0)}({\vec u}) \dif {\vec \xi}
 \nonumber \\
 & -2 {\vec u} \cdot \int_{\vec B} \mu_r^{-1}  \nabla \times {\vec \Theta }^{(1)}({\vec v}) \dif {\vec \xi} \nonumber \\
 &+\int_{\vec B}  \nabla \cdot \mu_r^{-1}  \nabla \times {\vec \Theta}^{(1)}({\vec v}) \times ( {\vec \Theta}^{(0)}({\vec u}) + \overline{{\vec \Theta}^{(1)}({\vec u})}- {\vec u} \times {\vec \xi} ) \dif {\vec \xi} \nonumber \\
= &\int_{\vec B}  \mu_r^{-1} \nabla \times {\vec \Theta}^{(1)}({\vec v}) \cdot \nabla \times \overline{{\vec \Theta}^{(1)}({\vec u})} \dif {\vec \xi}+
\int_{\vec B}  {\vec \Theta}^{(1)}({\vec v}) \cdot \nabla \times \mu_r^{-1}\nabla \times {\vec \Theta}^{(0)}({\vec u}) \dif {\vec \xi}  
\nonumber\\
&+ \int_{\vec B} \nabla \cdot  ( {\vec \Theta}^{(1)}({\vec v})  \times \mu_r^{-1}\nabla \times {\vec \Theta}^{(0)}({\vec u}) ) \dif {\vec \xi} \nonumber \\
&-2{\vec u} \cdot  \int_{\vec B} \mu_r^{-1}   \nabla \times {\vec \Theta}^{(1)}({\vec v})  \dif {\vec \xi}- \int_{{\vec B}^c} \nabla \cdot  (\nabla \times {\vec \Theta}^{(1)}({\vec v}) \times ( {\vec \Theta}^{(0)}({\vec u}) + \overline{{\vec \Theta}^{(1)}({\vec u})}- {\vec u} \times {\vec \xi} )) \dif {\vec \xi} . \nonumber 
\end{align}
Next, using $\nabla \times \mu_r^{-1}\nabla \times {\vec \Theta}^{(0)}({\vec u}) = {\vec 0} $ in ${\vec B}$, integrating by parts the third integral over ${\vec B}$, and expanding the integral over ${\vec B}^c$, we have
\begin{align}
A_1
=&\int_{\vec B} \mu_r^{-1} \nabla \times {\vec \Theta}^{(1)}({\vec v}) \cdot \nabla \times \overline{{\vec \Theta}^{(1)}({\vec u})} \dif {\vec \xi} -  \int_{{\vec B}^c} \nabla \cdot  ( {\vec \Theta}^{(1)}({\vec v})  \times \nabla \times {\vec \Theta}^{(0)}({\vec u}) ) \dif {\vec \xi} \nonumber\\
&-2  {\vec u} \cdot  \int_{\vec B} \mu_r^{-1} \nabla \times {\vec \Theta}^{(1)}({\vec v})  \dif {\vec \xi} \nonumber \\
& + \int_{{\vec B}^c} \nabla \times {\vec \Theta}^{(1)}({\vec v}) \cdot \nabla \times {\vec \Theta}^{(0)}({\vec u}) \dif {\vec \xi} + \int_{{\vec B}^c}  \nabla \times {\vec \Theta}^{(1)}({\vec v}) \cdot \nabla \times \overline{{\vec \Theta}^{(1)}({\vec u})}\dif {\vec \xi}\nonumber\\
& -2 {\vec u} \cdot \int_{{\vec B}^c} \nabla \times {\vec \Theta}^{(1)}({\vec v}) \dif {\vec \xi}\nonumber \\
=&\int_{{\vec B} \cup {\vec B}^c} \tilde{\mu_r}^{-1} \nabla \times {\vec \Theta}^{(1)}({\vec v}) \cdot \nabla \times \overline{{\vec \Theta}^{(1)}({\vec u}) }\dif {\vec \xi} 
-  \int_{{\vec B}^c} \nabla \times {\vec \Theta}^{(0)}({\vec u})  \cdot \nabla \times {\vec \Theta}^{(1)}({\vec v}) \dif {\vec \xi} \nonumber\\
&+ \int_{{\vec B}^c} \nabla \times \nabla \times {\vec \Theta}^{(0)}({\vec u}) \cdot {\vec \Theta}^{(1)}({\vec v}) \dif {\vec \xi} \nonumber\\
& -2  {\vec u} \cdot  \int_{\vec B}\mu_r^{-1} \nabla \times {\vec \Theta}^{(1)}({\vec v})  \dif {\vec \xi} + \int_{{\vec B}^c} \nabla \times {\vec \Theta}^{(1)}({\vec v}) \cdot \nabla \times {\vec \Theta}^{(0)}({\vec u}) \dif {\vec \xi}
-2 {\vec u} \cdot \int_{{\vec B}^c} \nabla \times {\vec \Theta}^{(1)}({\vec v}) \dif {\vec \xi}\nonumber \\
=&\int_{{\vec B} \cup {\vec B}^c} \tilde{\mu_r}^{-1} \nabla \times {\vec \Theta}^{(1)}({\vec v}) \cdot \nabla \times \overline{{\vec \Theta}^{(1)}({\vec u})}\dif {\vec \xi} 
 -2 {\vec u} \cdot  \int_{\vec B}\mu_r^{-1}  \nabla \times {\vec \Theta}^{(1)}({\vec v})  \dif {\vec \xi} -2 {\vec u} \cdot \int_{{\vec B}^c} \nabla \times {\vec \Theta}^{(1)}({\vec v}) \dif {\vec \xi}\nonumber .
 \end{align}
Note that the final equality follows by cancelling terms and using $\nabla \times \mu_r^{-1}\nabla \times {\vec \Theta}^{(0)}({\vec u}) = {\vec 0} $ in ${\vec B}^c$.
In addition, by transforming the surface integral in $A_2$, we have
\begin{align}
A_2= & -  {\vec u}  \cdot \int_{\vec B} \mu_r^{-1} \nabla \times {\vec \Theta}^{(1)} ({\vec v})\dif {\vec \xi} - {\vec u} \cdot \int_{\partial {\vec B}} {\vec n}^+ \times {\vec \Theta}^{(1)} ({\vec v}) \dif {\vec \xi} \nonumber \\
= &- {\vec u}  \cdot \int_{\vec B} \mu_r^{-1}  \nabla \times {\vec \Theta}^{(1)} ({\vec v}) \dif {\vec \xi} - {\vec u}\cdot \int_{{\vec B}^c} \nabla \times {\vec \Theta}^{(1)} ({\vec v}) \dif {\vec \xi} \nonumber  ,
\end{align}
so that
\begin{align}
( N^{\sigma_*} -C^{\sigma_*} ) ({\vec u},{\vec v}) = & \frac{\im  \alpha^3 }{4 } \int_{\vec B} \frac{1}{\nu} \nabla \times \mu_r^{-1} \nabla \times {\vec \Theta}^{(1)}({\vec v}) \cdot
 \nabla \times \mu_r^{-1} \nabla \times \overline{{\vec \Theta}^{(1)}({\vec u})} \dif {\vec \xi} \nonumber \\
& -\frac{\alpha^3}{4} \int_{{\vec B} \cup {\vec B}^c} \tilde{\mu_r}^{-1} \nabla \times {\vec \Theta}^{(1)}({\vec v}) \cdot \nabla \times \overline{{\vec \Theta}^{(1)}({\vec u})} \dif {\vec \xi} \nonumber .
\end{align}
 Denoting the real part of $(N^{\sigma_*} -C ^{\sigma_*})({\vec u},{\vec v}) $ as ${R}^{\sigma_*}({\vec u},{\vec v})$ and its imaginary part by ${I} ^{\sigma_*}({\vec u},{\vec v})$ then we have 
\begin{subequations}
\begin{align}
{R}^{\sigma_*}({\vec u},{\vec v}) = & -  \frac{  \alpha^3 }{4} \mathrm{Im} \left ( \int_{\vec B} \frac{1}{\nu} \nabla \times \mu_r^{-1} \nabla \times {\vec \Theta}^{(1)} ({\vec v}) \cdot
 \nabla \times \mu_r^{-1} \nabla \times \overline{{\vec \Theta}^{(1)}({\vec u})} \dif {\vec \xi} \right ) \nonumber \\
 & -\frac{\alpha^3}{4} \mathrm{Re} \left ( \int_{{\vec B} \cup {\vec B}^c} \tilde{\mu_r}^{-1} \nabla \times {\vec \Theta}^{(1)}({\vec v}) \cdot \nabla \times \overline{{\vec \Theta}^{(1)}({\vec u})} \dif {\vec \xi} 
 \right ) \label{eqn:rform1} , \\
{I}^{\sigma_*} ({\vec u},{\vec v}) = &  \frac{  \alpha^3 }{4} \mathrm{Re} \left ( \int_{\vec B} \frac{1}{\nu} \nabla \times \mu_r^{-1} \nabla \times {\vec \Theta}^{(1)}({\vec v}) \cdot
 \nabla \times \mu_r^{-1} \nabla \times \overline{{\vec \Theta}^{(1)} ({\vec u})}\dif {\vec \xi} \right ) \nonumber \\
 & -\frac{\alpha^3}{4} \mathrm{Im} \left ( \int_{{\vec B} \cup {\vec B}^c} \tilde{\mu_r}^{-1} \nabla \times {\vec \Theta}^{(1)} ({\vec v}) \cdot \nabla \times \overline{{\vec \Theta}^{(1)} ({\vec u }) } \dif {\vec \xi} 
 \right )\label{eqn:iform1}  .
\end{align} 
\end{subequations}
By the  properties of the complex conjugate, we get that
\begin{align}
{R}^{\sigma_*}({\vec v},{\vec u} )= & -  \frac{  \alpha^3 }{4} \mathrm{Im} \left ( \int_{\vec B} \frac{1}{\nu} \nabla \times \mu_r^{-1} \nabla \times {\vec \Theta}^{(1)} ({\vec u}) \cdot
 \nabla \times \mu_r^{-1} \nabla \times \overline{{\vec \Theta}^{(1)}({\vec v})} \dif {\vec \xi} \right ) \nonumber \\
 & -\frac{\alpha^3}{4} \mathrm{Re} \left ( \int_{{\vec B} \cup {\vec B}^c} \tilde{\mu_r}^{-1} \nabla \times {\vec \Theta}^{(1)} ({\vec u}) \cdot \nabla \times \overline{{\vec \Theta}^{(1)}({\vec v})} \dif {\vec \xi} 
 \right ) \nonumber \\
=& -  \frac{  \alpha^3 }{4} \mathrm{Im} \left (\overline{ \int_{\vec B} \frac{1}{\nu} \nabla \times \mu_r^{-1} \nabla \times {\vec \Theta}^{(1)}({\vec v}) \cdot
 \nabla \times \mu_r^{-1} \nabla \times \overline{{\vec \Theta}^{(1)}{(\vec u}) } \dif {\vec \xi}}  \right ) \nonumber \\
 & -\frac{\alpha^3}{4} \mathrm{Re} \left ( \overline{\int_{{\vec B} \cup {\vec B}^c} \tilde{\mu_r}^{-1} \nabla \times {\vec \Theta}^{(1)}({\vec v}) \cdot \nabla \times \overline{{\vec \Theta}^{(1)}({\vec u}) } \dif {\vec \xi} }
 \right ) \nonumber \\
 = &   \frac{  \alpha^3 }{4} \mathrm{Im} \left ( \int_{\vec B} \frac{1}{\nu}  \nabla \times \mu_r^{-1} \nabla \times {\vec \Theta}^{(1)} ({\vec v}) \cdot
 \nabla \times \mu_r^{-1} \nabla \times \overline{{\vec \Theta}^{(1)} ({\vec u}) } \dif {\vec \xi} \right ) \nonumber \\
 & -\frac{\alpha^3}{4} \mathrm{Re} \left ( \int_{{\vec B} \cup {\vec B}^c} \tilde{\mu_r}^{-1} \nabla \times {\vec \Theta}^{(1)} ({\vec v}) \cdot \nabla \times \overline{{\vec \Theta}^{(1)}({\vec u})} \dif {\vec \xi} 
 \right ) \nonumber \\
 = & -\frac{\alpha^3}{4} \mathrm{Re} \left ( \int_{{\vec B} \cup {\vec B}^c} \tilde{\mu_r}^{-1} \nabla \times {\vec \Theta}^{(1)} ({\vec v})\cdot \nabla \times \overline{{\vec \Theta}^{(1)}({\vec u})} \dif {\vec \xi} 
 \right ) =  {R}^{\sigma_*}({\vec u},{\vec v} )\label{eqn:symstep1} ,
 \end{align}
where, in the final step, we have used
\begin{equation}
\frac{  \alpha^3 }{4} \mathrm{Im} \left ( \int_{\vec B} \frac{1}{\nu}  \nabla \times \mu_r^{-1} \nabla \times {\vec \Theta}^{(1)} ({\vec v}) \cdot
 \nabla \times \mu_r^{-1} \nabla \times \overline{{\vec \Theta}^{(1)} ({\vec u}) } \dif {\vec \xi} \right )=0, \label{eqn:symstep2} 
\end{equation}
 which follows since we know that   ${R}^{\sigma_*} ({\vec u},{\vec v}) = {\mathcal R}^{\sigma_*}({\vec v},{\vec u}) $  by the symmetry of ${N}^{\sigma_*} ({\vec u},{\vec v})-{ C}^{\sigma_*}({\vec u},{\vec v})$ in Theorem \ref{thm:symmfull} and, hence, the symmetry of its real and imaginary parts. By applying similar arguments to ${\mathcal I}^{\sigma_*}$ we get that
 \begin{equation}
 -\frac{\alpha^3}{4} \mathrm{Im} \left ( \int_{{\vec B} \cup {\vec B}^c} \tilde{\mu_r}^{-1} \nabla \times {\vec \Theta}^{(1)} ({\vec v}) \cdot \nabla \times \overline{{\vec \theta}^{(1)} ({\vec u }) } \dif {\vec \xi} \right )=0, \label{eqn:symstep3} 
 \end{equation}
 and
\begin{equation}
 {I}^{\sigma_*} ({\vec u},{\vec v}) =   \frac{  \alpha^3 }{4} \mathrm{Re} \left ( \int_{\vec B} \frac{1}{\nu} \nabla \times \mu_r^{-1} \nabla \times {\vec \Theta}^{(1)}({\vec v}) \cdot
 \nabla \times \mu_r^{-1} \nabla \times \overline{{\vec \theta}^{(1)} ({\vec u})}\dif {\vec \xi} \right ) \label{eqn:symstep4} .
 \end{equation}
Still further, using (\ref{eqn:symstep3}),
 (\ref{eqn:symstep1}) becomes (\ref{eqn:rform2}) and, in a similar manner, using (\ref{eqn:symstep2}), (\ref{eqn:symstep4}) becomes (\ref{eqn:iform2}) as desired. 
 
 It also follows from (\ref{eqn:rform2}), and the linearity of the transmission problem (\ref{eqn:transproblem1var}),
 that $R^{\sigma_*} ({\vec u},{\vec v}) = R^{\sigma_*} ({\vec v},{\vec u})$ $\forall {\vec u}, {\vec v} \in {\mathbb R}^3$,  $R^{\sigma_*} ({\vec u}+{\vec w},{\vec v}) = R^{\sigma_*} ({\vec u},{\vec v})+R^{\sigma_*} ({\vec w},{\vec v}) $ $\forall {\vec u}, {\vec v}, {\vec w} \in {\mathbb R}^3$ and  $R^{\sigma_*}(c{\vec u},d{\vec v}) = c {d} R^{\sigma_*}({\vec u},{\vec v}) $ $\forall {\vec u}, {\vec v} \in {\mathbb R}^3$ and $\forall c,d \in {\mathbb R}$ and, thus, $R^{\sigma_*}({\vec u},{\vec v}):{\mathbb R}^3 \times {\mathbb R}^3 \to {\mathbb R}$ is a symmetric bilinear form on real vectors. Similarly, $I^{\sigma_*}({\vec u},{\vec v}):{\mathbb R}^3 \times {\mathbb R}^3 \to {\mathbb R }$ is a symmetric bilinear form on real vectors. In addition, since $-{\mathcal R}^{\sigma_*} ({\vec u},{\vec u})\ge 0$,  ${\mathcal I}^{\sigma_*} ({\vec u},{\vec u})\ge 0$ and ${\mathcal R}^{\sigma_*} ({\vec u},{\vec u})= {\mathcal I}^{\sigma_*} ({\vec u},{\vec u})=0$ only if ${\vec u}={\vec 0}$, $- R^{\sigma_*} ({\vec u},{\vec v}) $ and $I^{\sigma_*} ({\vec u},{\vec v}) $ define inner products on real vectors.
\end{proof}

\begin{corollary} \label{coll:diagonalcoeff}
An alternative splitting of the MPT is ${\mathcal M} = {\mathcal N}^0 + {\mathcal R}^{\sigma_*} + \im {\mathcal I}^{\sigma_*}$ where  ${\mathcal N}^0 =N^0 ( {\vec e}_i ,{\vec e}_j ) {\vec e}_i \otimes {\vec e}_j$,
 ${\mathcal R}^{\sigma_*} =R^{\sigma_*} ( {\vec e}_i ,{\vec e}_j ) {\vec e}_i \otimes {\vec e}_j$ and ${\mathcal I}^{\sigma_*} =I^{\sigma_*} ( {\vec e}_i ,{\vec e}_j ) {\vec e}_i \otimes {\vec e}_j$ are real symmetric tensors. In addition,  $({\mathcal N}^0)_{ii}\ge 0$, $({\mathcal R}^{\sigma_*} )_{ii}\le 0$ and $ ({\mathcal I}^{\sigma_*})_{ii} \ge 0$ where the repeated index $i$ does not imply summation.
 \end{corollary}
 
 \begin{proof}
 The splitting ${\mathcal M} = {\mathcal N}^0 + {\mathcal R}^{\sigma_*} + \im {\mathcal I}^{\sigma_*}$ immediately follows from Theorem~\ref{thm:realandimagpts}. The symmetry of ${\mathcal N}^0 =N^0 ( {\vec e}_i ,{\vec e}_j ) {\vec e}_i \otimes {\vec e}_j$ follows from (\ref{eqn:n0form2}) and the symmetries of ${\mathcal R}^{\sigma_*} =R^{\sigma_*} ( {\vec e}_i ,{\vec e}_j ) {\vec e}_i \otimes {\vec e}_j$ and ${\mathcal I}^{\sigma_*} =I^{\sigma_*} ( {\vec e}_i ,{\vec e}_j ) {\vec e}_i \otimes {\vec e}_j$ from (\ref{eqn:rform2})
 and (\ref{eqn:iform2}), respectively. 
 The diagonal coefficients of ${\mathcal N}^0$ are quoted in Corollary~\ref{coll:diagonalcoeffn0} and 
  those of  ${\mathcal R}^{\sigma_*}$ and ${\mathcal I}^{\sigma_*}$ are
\begin{subequations}
\begin{align}
({\mathcal R}^{\sigma_*})_{ii} = & - \frac{\alpha^3}{4} \left  (  \| \nabla \times {\vec \Theta}^{(1)} ({
\vec e}_i)\|_{W(\mu_r^{-1}, {\vec B})}^2 +  \| \nabla \times {\vec \Theta}^{(1)} ({
\vec e}_i) \|_{L^2({\vec B}^c )}^2 \right ) , \\
({\mathcal I}^{\sigma_*})_{ii} = &  \frac{\alpha^3}{4}  \left (  \|\nabla \times \mu_r^{-1}\nabla \times {\vec \Theta}^{(1)} ({
\vec e}_i) \|_{W(\nu^{-1},{\vec B})}^2  \right ) ,
\end{align} 
\end{subequations}
leading to the quoted result.
\end{proof}

\begin{remark}
In Theorems~\ref{thm:formsn0} and~\ref{thm:realandimagpts} we have established that the MPT follows from the symmetric bilinear form
\begin{equation}
M({\vec u}, {\vec v} ) =  N^0 ({\vec u},{\vec v}) +{R}^{\sigma_*} ({\vec u}, {\vec v} ) + \im I^{\sigma_*} ({\vec u}, {\vec v} ), \label{eqn:newsplitmptfun}
\end{equation}
where $ {\vec u}$ and $ {\vec v}$ are real vectors and
 $N^0 ({\vec u},{\vec v})$, $-R^{\sigma_*} ({\vec u}, {\vec v} )$ and $I^{\sigma_*} ({\vec u}, {\vec v} )$ are symmetric bilinear forms and  inner products. This suggests that another possible route to the derivation of the asymptotic formula for $({\vec H}_\alpha-{\vec H}_0)({\vec x})$ could be through through the approach of topological derivatives~\cite{novotnybook}, where, through the definition of an appropriate energy functional, its topological derivative is the leading order term of (\ref{eqn:asymp}).
Still further, $N^0 ({\vec u},{\vec v}) $ defines a magnetostatic type energy, $-{R}^{\sigma_*} ({\vec u}, {\vec v} ) $ a magnetic type energy and $ I^{\sigma_*} ({\vec u}, {\vec v} )$ an electric (Ohmic) type energy functional for pairs of solutions ${\vec \Theta}({\vec u})$ and ${\vec \Theta}({\vec v})$, which provides a concrete interpretation of the three contributions in (\ref{eqn:newsplitmptfun}).
\end{remark}

We complete this section by establishing an alternative form of ${R}^{\sigma_*} ({\vec u}, {\vec v} )$ and $ I^{\sigma_*} ({\vec u}, {\vec v} )$. To do this, we first remark that the weak form for the transmission problem  (\ref{eqn:transproblem1var}) is: Find ${\vec \Theta}^{(1)} ({\vec u})  \in X$ such that
\begin{align}
\left <  \tilde{\mu}_r^{-1} \nabla \times {\vec \Theta}^{(1)} ({\vec u}) , \nabla \times {\vec \psi}\right >_{L^2( {\vec B}\cup{\vec B}^c)} = \im \left < \nu ( {\vec \Theta}^{(1)} ({\vec u}) + {\vec \Theta}^{(0)} ({\vec u}) ),{\vec \psi}\right >_{L^2({\vec B})} \qquad \forall {\vec \psi} \in X, \label{eqn:weakformtheta1}
\end{align}
where
\begin{align}
X:= \{  {\vec \varphi} \in {\vec H}(\hbox{curl}): \nabla \cdot {\vec \varphi}=0 \ \text{ in ${\vec B}^c$} , \ {\vec \varphi} = O( | {\vec \xi}|^{-1} )  \text{ as $|{\vec \xi}| \to \infty $}  \} .\nonumber
\end{align}
We can then establish the following result:

\begin{lemma} \label{lemma:altformiandr}
An alternative form of the symmetric bilinear forms  ${R}^{\sigma_*} ({\vec u}, {\vec v} )$ and $ I^{\sigma_*} ({\vec u}, {\vec v} )$, introduced in (\ref{eqn:rform2}) and (\ref{eqn:iform2}), respectively, is
\begin{subequations}
\begin{align}
{R}^{\sigma_*} ({\vec u}, {\vec v} ) & = - \frac{\alpha^3}{4} \left <  \nu \mathrm{Im}( {\vec \Theta}^{(1)} ({\vec u}) ) , {\vec \Theta}^{(0)} ({\vec v}) \right>_{L^2({\vec B})}, \label{eqn:raltform}\\
{I}^{\sigma_*} ({\vec u}, {\vec v} ) & = \frac{\alpha^3}{4}  \left ( \left <  \nu  \mathrm{Re}( {\vec \Theta}^{(1)} ({\vec u}) ) ,  {\vec \Theta}^{(0)} ({\vec v}) \right>_{L^2({\vec B})}  + \left < \nu  {\vec \Theta}^{(0)} ({\vec u})  , {\vec \Theta}^{(0)} ({\vec v}) \right >_{L^2({\vec B})} 
 \right ).\label{eqn:ialtform}
\end{align}
\end{subequations}
\end{lemma}

\begin{proof}
Choosing ${\vec \psi} = {\vec \Theta}^{(1)} ({\vec v}) $ in (\ref{eqn:weakformtheta1}) then
\begin{align}
\left <  \tilde{\mu}_r^{-1} \nabla \times {\vec \Theta}^{(1)} ({\vec u}) , \nabla \times {\vec \Theta}^{(1)} ({\vec v}) \right >_{L^2({\vec B}\cup{\vec B}^c)} = \im \left < \nu ( {\vec \Theta}^{(1)} ({\vec u}) + {\vec \Theta}^{(0)} ({\vec u}) ), {\vec \Theta}^{(1)} ({\vec v}) \right >_{L^2({\vec B})} \nonumber,
\end{align}
and, hence, from (\ref{eqn:rform2}) we obtain that
\begin{align}
{R}^{\sigma_*} ({\vec u}, {\vec v} ) &= - \im \frac{\alpha^3}{4}  \left < \nu ( {\vec \Theta}^{(1)} ({\vec u}) + {\vec \Theta}^{(0)} ({\vec u}) ), {\vec \Theta}^{(1)} ({\vec v}) \right > _{L^2( {\vec B} )} \nonumber ,
\end{align}
which must be real by definition.
Also, by using the transmission problem (\ref{eqn:transproblem1var}) and recalling $ {\vec \Theta}^{(0)} ({\vec u}) \in {\mathbb R}^3 $, we have that
\begin{align}
{I}^{\sigma_*} ({\vec u}, {\vec v} ) =& \frac{\alpha^3}{4}\left < \nu ( {\vec \Theta}^{(1)} ({\vec u}) + {\vec \Theta}^{(0)} ({\vec u}) ), {\vec \Theta}^{(1)} ({\vec v})+  {\vec \Theta}^{(0)} ({\vec v})\right >_{L^2({\vec B})}  \nonumber\\
= & \frac{\alpha^3}{4}  \left (   \left  < \nu ( {\vec \Theta}^{(1)} ({\vec u}) + {\vec \Theta}^{(0)} ({\vec u}) ) , {\vec \Theta}^{(1)} ({\vec v}) \right >_{L^2({\vec B})} +  \left < \nu  {\vec \Theta}^{(1)} ({\vec u}) ,  {\vec \Theta}^{(0)} ({\vec v}) \right >_{L^2({\vec B})} \right . \nonumber \\
&\left . +\left <  \nu {\vec \Theta}^{(0)} ({\vec u}) , 
{\vec \Theta}^{(0)} ({\vec v}) \right >_{L^2({\vec B})} \right) \nonumber \\
= & - \frac{1}{\im} {R}^{\sigma_*}  ({\vec u}, {\vec v} )+ \frac{\alpha^3}{4}  \left (  \left < \nu  {\vec \Theta}^{(1)} ({\vec u}) ,  {\vec \Theta}^{(0)} ({\vec v}) \right >_{L^2({\vec B})} + \left < \nu {\vec \Theta}^{(0)} ({\vec u}) ,{\vec \Theta}^{(0)} ({\vec v}) \right >_{L^2({\vec B})}  \right) \nonumber,
\end{align}
which must be real by definition.
Then, since ${R}^{\sigma_*}  ({\vec u}, {\vec v} )\in {\mathbb R}$,  ${\vec \Theta}^{(0)} ({\vec u}) \in {\mathbb R}^3 $ and ${\vec \Theta}^{(1)} ({\vec u}) \in {\mathbb C}^3 $, it follows that the first term is purely imaginary, the second is complex and the last term is real and, hence, an alternative form of ${I}^{\sigma_*} ({\vec u}, {\vec v} )$
is
 given by (\ref{eqn:ialtform}). Still further, we have
\begin{align}
\mathrm{Im} \left (  - \frac{1}{\im} {R}^{\sigma_*}  ({\vec u}, {\vec v} ) \right ) + 
 \frac{\alpha^3}{4}  \left <  \nu \mathrm{Im}( {\vec \Theta}^{(1)} ({\vec u}) ),  {\vec \Theta}^{(0)} ({\vec v}) \right >_{L^2({\vec B})}  =0, \nonumber
\end{align}
and, from this, we immediately obtain (\ref{eqn:raltform}).
\end{proof}
\begin{corollary}
In a similar way to Corollary~\ref{coll:diagonalcoeff}, the expressions (\ref{eqn:raltform}) and (\ref{eqn:ialtform}), obtained in Lemma~\ref{lemma:altformiandr}, can be used to obtain alternative expressions for the tensors  ${\mathcal R}^{\sigma_*} =R^{\sigma_*} ( {\vec e}_i ,{\vec e}_j ) {\vec e}_i \otimes {\vec e}_j$ and ${\mathcal I}^{\sigma_*} =I^{\sigma_*} ( {\vec e}_i ,{\vec e}_j ) {\vec e}_i \otimes {\vec e}_j$.
\end{corollary}

\section{Bounds on the off--diagonal coefficients of ${\mathcal R}^{\sigma_*} $ and ${\mathcal I}^{\sigma_*}$} \label{sect:bounds}
Bounds on the off-diagonal coefficients of the P\'olya-Szeg\"o tensor, and hence ${\mathcal N}^0$ for homogenous $\mu_*$, have previously been established e.g.~\cite{ammarikangbook,kleinmansenior,ledgerlionheart2016}. The following provides a bound on the magnitudes of the off-diagonal coefficients of ${\mathcal R}^{\sigma_*}$ and ${\mathcal I}^{\sigma_*}$.
\begin{lemma}
For $i\ne j$ then
\begin{subequations}
\begin{align}
|({\mathcal R}^{\sigma_*} )_{ij}| \le & |  \mathrm{Tr}({\mathcal R}^{\sigma_*}) |  = \left |   ( {\mathcal R}^{\sigma_*})_{kk}\right | , \\
|({\mathcal I}^{\sigma_*})_{ij}| \le & \mathrm{Tr}({\mathcal I}^{\sigma_*})  =  ( {\mathcal I}^{\sigma_*} )_{kk} .
\end{align}
\end{subequations}
\end{lemma}

\begin{proof}
First we construct an upper bound on $|({\mathcal R}^{\sigma_*})_{ij}|$ for $i \ne j$ as
\begin{align}
| ({\mathcal R}^{\sigma_*})_{ij}| = &  \frac{\alpha^3}{4}\left |    \int_{{\vec B}  \cup {\vec B}^c} \tilde{\mu_r}^{-1} \nabla \times {\vec \Theta}^{(1)}({\vec e}_j) \cdot \nabla \times \overline{{\vec \Theta}^{(1)}({\vec e}_i) } \dif {\vec \xi} 
  \right | \nonumber  \\
 \le & \frac{\alpha^3}{4}  \left (  \left | \int_{\vec B} {\mu_r}^{-1} \nabla \times {\vec \Theta}^{(1)}({\vec e}_j) \cdot \nabla \times \overline{{\vec \Theta}^{(1)}({\vec e}_i)} \dif {\vec \xi} \right |  +   \left | \int_{{\vec B}^c} \nabla \times {\vec \Theta}^{(1)} ({\vec e}_j) \cdot \nabla \times \overline{{\vec \Theta}^{(1)}({\vec e}_i)} \dif {\vec \xi} \right | \right ) \nonumber \\
 \le & \frac{\alpha^3}{4}  \left (   \|  \nabla \times {\vec \Theta}^{(1)} ({\vec e}_j)\|_{W(\mu_r^{-1}, {\vec B})} \|  \nabla \times {\vec \Theta}^{(1)} ({\vec e}_i ) \|_{W(\mu_r^{-1}, {\vec B} )}  \right . \nonumber \\
 & \left . + \|  \nabla \times {\vec \Theta}^{(1)}({\vec e}_j) \|_{L^2({\vec B}^c)} \|  \nabla \times {\vec \Theta}^{(1)}({\vec e}_i) \|_{L^2({\vec B}^c)} \right ), \nonumber 
 \end{align}
which follows by application of the Cauchy-Schwartz inequality. From $(c-d)^2+2cd =c^2+d^2$ we have $cd <2cd < c^2+d^2 $ for real $c>0$ and $d>0$ and so
\begin{align}
|({\mathcal R}^{\sigma_*})_{ij}|  \le &  \frac{\alpha^3}{4}  \left (  \|  \nabla \times {\vec \Theta}^{(1)}({\vec e}_i) \|_{W(\mu_r^{-1}, {\vec B} )}^2 +  \|  \nabla \times {\vec \Theta}^{(1)} ({\vec e}_i) \|_{L^2({\vec B}^c)} ^2 \right . \nonumber \\
& \left . +
  \|  \nabla \times {\vec \Theta}^{(1)}({\vec e}_j) \|_{W(\mu_r^{-1}, {\vec B})} ^2+  \|  \nabla \times {\vec \Theta}^{(1)}({\vec e}_j) \|_{L^2({\vec B}^c)} ^2 \right ) \nonumber\\
 \le &  \frac{\alpha^3}{4}  \left (  \sum_{k=1}^3  \|  \nabla \times {\vec \Theta}^{(1)} ({\vec e}_k) \|_{W(\mu_r^{-1}, {\vec B} )} ^2 + \|  \nabla \times {\vec \Theta}^{(1)}({\vec e}_k) \|_{L^2({\vec B} )} ^2  \right ) \nonumber \\
 \le & \left |  \sum_{k=1}^3 ( {\mathcal R}^{\sigma_*})_{kk} \right  | ,  \nonumber
 \end{align}
 as desired. In a similar fashion, for $i\ne j$,
 \begin{align}
 |( {\mathcal I}^{\sigma_*})_{ij}  | = &  \frac{  \alpha^3 }{4} \left |   \int_{\vec B} \frac{1}{\nu} \nabla \times \mu_r^{-1} \nabla \times {\vec \Theta}^{(1)} ({\vec e}_j)\cdot
 \nabla \times \mu_r^{-1} \nabla \times \overline{{\vec \Theta}^{(1)}({\vec e}_i)} \dif {\vec \xi}  \right | \nonumber \\
  \le & \frac{  \alpha^3 }{4}  \left ( \| \nabla \times \mu_r^{-1} \nabla \times {\vec \Theta}^{(1)} ({\vec e}_i) \|_{W(\nu^{-1}, {\vec B} )}  \| \nabla \times \mu_r^{-1} \nabla \times {\vec \Theta}^{(1)} ({\vec e}_j) \|_{W(\nu^{-1}, {\vec B} )}  \right ) \nonumber \\
 \le & \frac{  \alpha^3 }{4}  \left ( \| \nabla \times \mu_r^{-1} \nabla \times {\vec \Theta}^{(1)}  ({\vec e}_i)\|_{W(\nu^{-1}, {\vec B} )}^2+  \| \nabla \times \mu_r^{-1} \nabla \times {\vec \Theta}^{(1)} ({\vec e}_j) \|_{W(\nu^{-1}, {\vec B} )}  ^2 \right ) \nonumber \\
 \le & \frac{  \alpha^3 }{4} \sum_{k=1}^3 ( {\mathcal I}^{\sigma_*})_{kk} \nonumber ,
  \end{align}
  since $({\mathcal I}^{\sigma_*})_{kk}  > 0$ giving the result as desired.
\end{proof}

\section{Eigenvalues of ${\mathcal R}^{\sigma_*}$, ${\mathcal I}^{\sigma_*} $ and ${\mathcal N}^0 $} \label{sect:eig}
As ${\mathcal R}^{\sigma_*}$ and ${\mathcal I}^{\sigma_*}$ are real symmetric tensors, their coefficients, when arranged in the form of a $3\times 3$ matrices, can be diagonalised by  orthogonal matrices $Q^{{\mathcal R}^{\sigma_*}}$ and $Q^{{\mathcal I}^{\sigma_*}}$, respectively, so that $\Lambda^{{\mathcal R}^{\sigma_*}}$ and $\Lambda^{{\mathcal I}^{\sigma_*}}$ are diagonal and
\begin{subequations}
\begin{align}
(\Lambda^{{\mathcal R}^{\sigma_*}})_{ij} =  & ( (Q^{{\mathcal R}^{\sigma_*}})^T  {\mathcal R}^{\sigma_*} Q^{{\mathcal R}^{\sigma_*}})_{ij} = (Q^{{\mathcal R}^{\sigma_*}})_{ki} ({\mathcal R}^{\sigma_*})_{kp} (Q^{{\mathcal R}^{\sigma_*}})_{pj} , \\
(\Lambda^{{\mathcal I}^{\sigma_*}})_{ij} =  & ( (Q^{{\mathcal I}^{\sigma_*}})^T  {\mathcal I}^{\sigma_*} Q^{{\mathcal I}^{\sigma_*}})_{ij} = (Q^{{\mathcal I}^{\sigma_*}})_{ki} ({\mathcal I}^{\sigma_*})_{kp} (Q^{{\mathcal I}^{\sigma_*}})_{pj} .
\end{align}
\end{subequations}
Moreover, the diagonal entries of $\Lambda^{{\mathcal R}^{\sigma_*}}$ and $\Lambda^{{\mathcal I}^{\sigma_*}}$ are the eigenvalues of ${\mathcal R}^{\sigma_*}$ and ${\mathcal I}^{\sigma_*}$, respectively, and the columns of the matrices $Q^{{\mathcal R}^{\sigma_*}}$ and $Q^{{\mathcal I}^{\sigma_*}}$ are their eigenvectors.  In a similar way, ${\mathcal N}^0$ can be diagonalised by the orthogonal matrix $Q^{{\mathcal N}^0}$ containing the eigenvectors of ${\mathcal N}^0$ so that
\begin{equation}
(\Lambda^{{\mathcal N}^0})_{ij} =   ( (Q^{{\mathcal R}^{0 }})^T  {\mathcal N}^{0} Q^{{\mathcal R}^{ 0}})_{ij} = (Q^{{\mathcal R}^{ 0 }})_{ki} ({\mathcal N}^{0})_{kp} (Q^{{\mathcal R}^{0}})_{pj} ,
\end{equation}
are the elements of a diagonal matrix containing the eigenvalues of ${\mathcal N}^0$.

The orthogonal matrices $Q^{{\mathcal R}^{\sigma_*}}$, $Q^{{\mathcal I}^{\sigma_*}}$ and $Q^{{\mathcal N}^{0}}$ can also be viewed as rotations of the object ${\vec B}$ such that 
 ${\mathcal R}^{\sigma_*}[Q^{{\mathcal R}^{\sigma_*}}({\vec B})]$, ${\mathcal I}^{\sigma_*}[Q^{{\mathcal I}^{\sigma_*}}({\vec B})]$ and ${\mathcal N}^0[Q^{{\mathcal N}^{0}}({\vec B}) ]$~\footnote{In a similar way (\ref{eqn:asymp})  the square brackets used here emphasise the the object for which the tensor is evaluated.} are diagonal and their entries being the associated eigenvalues. We summarise this as the  main result of this section:
\begin{theorem} \label{thm:tensoreigs}
The eigenvalues of  ${\mathcal R}^{\sigma_*}$, ${\mathcal I}^{\sigma_*}$ and ${\mathcal N}^0$ can be explicitly expressed as the diagonal coefficients
\begin{subequations}
\begin{align}
(\Lambda^{{\mathcal R}^{\sigma_*}})_{ii} =& ({\mathcal R}^{\sigma_*} [ Q^{{\mathcal R}^{\sigma_*}}({\vec B})])_{ii}  \nonumber \\
=&  - \frac{\alpha^3}{4} \left  (  \| \nabla \times {\vec \Theta}^{(1)} ({\vec e}_i) \|_{W(\mu_r^{-1}, Q^{{\mathcal R}^{\sigma_*}}({\vec B}))}^2 +  \| \nabla \times {\vec \Theta}^{(1)} ({\vec e}_i) \|_{L^2(Q^{{\mathcal R}^{\sigma_*}}({\vec B}^c) )}^2 \right ), \\
(\Lambda^{{\mathcal I}^{\sigma_*}})_{ii} & =( {\mathcal I}^{\sigma_*} [ Q^{{\mathcal I}^{\sigma_*}}({\vec B})])_{ii} \nonumber \\
 = &  \frac{\alpha^3}{4 }  \left (  \|\nabla \times \mu_r^{-1}\nabla \times {\vec \Theta}^{(1)} ({\vec e}_i) \|_{W(\nu^{-1},Q^{{\mathcal I}^{\sigma_*}}({\vec B}))}^2  \right ) , \\
(\Lambda^{{\mathcal N}^{0}})_{ii} & = ( {\mathcal N}^{0} [ Q^{{\mathcal N}^{0}}({\vec B})])_{ii} \nonumber \\
 = &\frac{\alpha^3}{4}   \left ( 4\int_{ Q^{{\mathcal R}^{0}}({\vec B})}  ( 1- \mu_r^{-1} ) \dif {\vec \xi} +   \| \nabla \times \tilde{\vec \Theta}^{(0)}({\vec e}_i) \|_{L^2 ( {Q^{{\mathcal R}^{0}}({\vec B}^c } )} ^2 \right . \nonumber \\
& \left . +\|  \nabla \times \tilde{\vec \Theta}^{(0)}({\vec e}_i) \|^2_{W( \mu_r^{-1} ,Q^{{\mathcal R}^{0}}({\vec B}))}  \right ) ,
\end{align}
\end{subequations}
where the repeated index $i$ does not imply summation and ${\vec \Theta}^{(1)}({\vec u})$  is the solution to
\begin{subequations}
\begin{align} 
\nabla_\xi \times \mu_r^{-1} \nabla_\xi \times {\vec \Theta} ^{(1)} ({\vec u})- \im \nu ({\vec \Theta}^{(1)} ({\vec u}))+{\vec \Theta}^{(0)}({\vec u}))  & = {\vec 0}   && \hbox{in $Q({\vec B} )$}  , \\
\nabla_\xi \times  \nabla_\xi \times {\vec \Theta} ^{(1)} ({\vec u}) & = {\vec 0}   && \hbox{in $Q(  {\vec B}^c)$} , \\
\nabla_\xi \cdot {\vec \Theta}^{(1)} ({\vec u}) &= 0 && \hbox{in $Q({\vec B}^c)$} , \\
\left [ {\vec \Theta}^{(1)} ({\vec u}) \times {\vec n} \right ]_{ Q(\Gamma)}  & = {\vec 0}  && \hbox{on $ Q(\Gamma)$} ,\\ 
  \left [   \tilde{\mu}_r^{-1}   \nabla_\xi  \times {\vec \Theta}^{(1)} ({\vec u}) \times {\vec n} \right ]_{ Q(\Gamma)}  & = 
 {\vec 0} &&
\hbox{on ${ Q(\Gamma)}$} ,\\ 
{\vec \Theta}^{(1)}  & = O(|{\vec \xi} |^{-1}) && \hbox{as $|{\vec \xi}| \to \infty$} ,
\end{align} 
\end{subequations}
with $Q=Q^{{\mathcal R}^{\sigma_*}}$, $Q=Q^{{\mathcal I}^{\sigma_*}}$, respectively. In addition,  ${\vec \Theta}^{(0)}({\vec u})= \tilde{\vec \Theta}^{(0)}({\vec u}) +{\vec u} \times {\vec \xi}$  is the solution to
\begin{subequations}
\begin{align} 
\nabla_\xi \times \tilde{\mu}_r^{-1} \nabla_\xi \times {\vec \Theta}^{(0)}({\vec u}) & ={\vec 0}  && \hbox{in ${Q ({\vec B})} \cup {Q({\vec B}^c)}$}  , \\
\nabla_\xi \cdot {\vec \Theta}^{(0)}({\vec u} ) &= 0  && \hbox{in ${Q ({\vec B})}\cup {Q ({\vec B}^c)}$} , \\
\left [ {\vec \Theta}  ^{(0)} ({\vec u}) \times {\vec n} \right ]_{ {Q (\Gamma)}}   &= {\vec 0}   && \hbox{on $ Q (\Gamma)$} ,\\ 
\,  \left [   \tilde{\mu}_r^{-1}   \nabla_\xi  \times {\vec \Theta}^{(0)}({\vec u} )  \times {\vec n} \right ]_{ {Q (\Gamma)}}  & = 
{\vec 0}  &&
\hbox{on $ {Q (\Gamma )}$} ,\\ 
{\vec \Theta}^{(0)} - {\vec u} \times {\vec \xi} & = O(|{\vec \xi} |^{-1})  && \hbox{as $|{\vec \xi}| \to \infty$} ,
\end{align} 
\end{subequations}
with $Q=Q^{{\mathcal R}^{\sigma_*}}$, $Q=Q^{{\mathcal I}^{\sigma_*}}$, $Q=Q^{{\mathcal N}^0 }$, respectively.
\end{theorem}
\begin{proof}
Under the action of a rotation $Q$, the MPT's coefficients transform as $({\mathcal M}')_{ij}   =({\mathcal M} [Q ({\vec B})])_{ij} = (Q)_{ip} (Q)_{jq} ({\mathcal M})_{pq}$. Thus,
\begin{align}
({\mathcal N}^{0'})_{ij}+( {\mathcal Q}^{\sigma_* ' })_{ij}+\im ({\mathcal I}^{\sigma_* ' })_{ij}= & (Q)_{ip} (Q)_{jq}( ({\mathcal N}^{0})_{pq}+ ({\mathcal R}^{\sigma_*  })_{pq}+\im( {\mathcal I}^{\sigma_* } )_{pq} ) \nonumber \\
( {\mathcal N}^{0'}+ {\mathcal R}^{\sigma_* ' })_{ij} +\im( {\mathcal I}^{\sigma_* ' })_{ij}= & (Q)_{ip} (Q)_{jq} ( {\mathcal N}^{0}+ {\mathcal R}^{\sigma_*  })_{pq}+\im (Q)_{ip} (Q)_{jq} ( {\mathcal I} ^{\sigma_* })_{pq},  \nonumber
\end{align}
 and so $( {\mathcal N}^{0'} )_{ij}=({\mathcal N}^{0}[Q({\vec B})] )_{ij} = (Q)_{ip} (Q)_{jq} ({\mathcal N}^{0})_{pq}$, $( {\mathcal R}^{\sigma_* ' })_{ij} =({\mathcal R}^{\sigma_*  } [Q({\vec B})] )_{ij}= (Q)_{ip} (Q)_{jq} ({\mathcal R}^{\sigma_*})_{pq}$ and $ ({\mathcal I}^{\sigma_* ' } )_{ij}=({\mathcal I}^{\sigma_*  } [Q({\vec B})] )_{ij}= (Q)_{ip} (Q)_{jq} ({\mathcal I}^{\sigma_*})_{pq}$. Choosing $Q=Q^{{\mathcal R}^{\sigma_*}}$, and noting that under the action of this rotation  ${\vec B}$ becomes $Q^{{\mathcal R}^{\sigma_*}}({\vec B} )$, then, by the application of Corollary~\ref{coll:diagonalcoeff} for the rotated object configuration we have
 \begin{align}
( {\mathcal R}^{\sigma_* ' } )_{ii}& = ({\mathcal R}^{\sigma_* }[Q^{{\mathcal R}^{\sigma_*}}({\vec B})])_{ii} =  - \frac{\alpha^3}{4} \left  (  \| \nabla \times {\vec \Theta}^{(1)} ({\vec e}_i) \|_{W(\mu_r^{-1},   Q^{{\mathcal R}^{\sigma_*}}(B))}^2 +  \| \nabla \times {\vec \Theta}^{(1)} ({\vec e}_i)  \|_{L^2(Q^{{\mathcal R}^{\sigma_*}}(B^c) )}^2 \right ) \nonumber \\
 {} & = (Q^{\sigma_*})_{ip} (Q^{\sigma_*})_{iq}( {\mathcal R}^{\sigma_*})_{pq} =  (\Lambda^{{\mathcal R}^{\sigma_*}})_{ii} ,\nonumber
  \end{align}
for the diagonal coefficients, where the repeated index $i$ does not imply summation.  Repeating similar steps for $Q=Q^{{\mathcal I}^{\sigma_*}}$ and $Q=Q^{{\mathcal N}^{0}} $ gives the corresponding result for $\Lambda^{{\mathcal I}^{\sigma_*}}$ and $\Lambda^{{\mathcal N}^{0}}$.
 \end{proof}
 
 \begin{remark}
When  $Q=Q^{{\mathcal R}^{\sigma_*}}$ is applied to ${\vec B}$, the resulting
$  {\mathcal R}^{\sigma_* '}={\mathcal R}^{\sigma_* }[Q^{{\mathcal R}^{\sigma_*}}({\vec B})] $
 will necessarily be diagonal and will have the eigenvalues of ${\mathcal R}^{\sigma_* }[{\vec B}]$ as its diagonal coefficients, i.e. $\Lambda^{{\mathcal R}^{\sigma_*}} = {\mathcal R}^{\sigma_* '}$. 
  Since 
$ {\mathcal R}^{\sigma_* '} $  is diagonal, the eigenvalues of  $ {\mathcal R}^{\sigma_* '}  $ are its diagonal entries, i.e. $\Lambda^{{\mathcal R}^{\sigma_*'}} =  {\mathcal R}^{\sigma_* '}$, and the eigenvectors of  $ {\mathcal R}^{\sigma_* '}  $  form the columns of  ${\mathbb I}$. However, the eigenvectors of ${\mathcal R}^{\sigma_* }[{\vec B}] $ do not, in general, form the columns of ${\mathbb I}$ unless the object has rotational or reflectional symmetries.  It follows that the eigenvalues contained in $\Lambda^{{\mathcal R}^{\sigma_*}}= \Lambda^{{\mathcal R}^{\sigma_*'}}$ are invariant under the action of rotation of an object, but the eigenvectors of ${\mathcal R}^{\sigma_* }$ are not. Using similar, arguments we also get that $\Lambda^{{\mathcal I}^{\sigma_*'}}=\Lambda^{{\mathcal I}^{\sigma_*}}$ and $\Lambda^{{\mathcal N}^{0'}}=\Lambda^{{\mathcal N}^{0}}$ are invariant under rotation.
 \end{remark}

 \begin{corollary}
Excluding the limiting cases of zero frequency  and infinite conductivity, ${\mathcal R}^{\sigma_*}$ is negative definite and ${\mathcal I}^{\sigma_*}$ is positive definite.
If $\mu({\vec \xi}) >1 $ for ${\vec \xi} \in {\vec B}$, ${\mathcal N}^{0}$ is positive definite for an inhomogeneous object.  For a homogeneous object, ${\mathcal N}^{0}$ is positive definite if $\mu_r >1 $ and negative definite if $\mu_r <1$. For the limiting case of zero frequency, ${\mathcal M} = {\mathcal N}^0$ and, thus, has the aforementioned properties of the real tensor ${\mathcal N}^0$ and, for the limiting case of infinite conductivity, ${\mathcal M} \to  {\mathcal M}(\infty)$  is real and negative definite.
\end{corollary}
\begin{proof}
Choosing $\nu \in (0, \infty)$
excludes the limiting cases of zero frequency and infinite conductivity~\cite{ledgerlionheart2016}. The definiteness of ${\mathcal R}^{\sigma_*}$, ${\mathcal I}^{\sigma_*}$, for $\nu \in (0,\infty)$,
and ${\mathcal N}^{0}$, for  $\mu({\vec \xi}) >1 $ for ${\vec \xi} \in {\vec B}$,
 follow from Theorem~\ref{thm:tensoreigs}. The results on ${\mathcal N}^{0}$ for a homogeneous object follow from~\cite[pg. 93]{ammarikangbook}, since ${\mathcal N}^0$ coincides with the P\'olya-Szeg\"o tensor for a homogenous object. The results on the limiting cases  follow from (\ref{eqn:limit0}) and (\ref{eqn:limitinf})  by considering $\nu=0$ and $\nu \to \infty$, respectively.
%
 \end{proof}

\section{Spectral analysis of ${\mathcal M}=  {\mathcal N}^0 + {\mathcal R}^{\sigma_*} + \im {\mathcal I}^{\sigma_*} $ for an object with homogeneous $\sigma_*$} \label{sect:spectrum}
In this section, we investigate how  ${\mathcal M}$ depends on $\omega$. An illustration of the typical behaviour of $\text{Re} ( {\mathcal M}(\omega)) ={\mathcal N}^0 +  {\mathcal R}^{\sigma_*}(\omega) $ and $\text{Im} ( {\mathcal M} (\omega) )={\mathcal I}^{\sigma_*}(\omega)  $ for the case of a conducting sphere $B_\alpha = \alpha B$ with radius $\alpha=0.01\text{m}$ and material parameters $\mu_r=1.5$ and $\sigma_*=5.96 \times 10^6 \text{S/m}$ is shown in Figure~\ref{fig:sphere}. This plot is obtained by evaluating the known analytical solution provided by Wait~\cite{Wait1951}.
Here,   $\Lambda^{{\mathcal R}^{\sigma_*}}$, $\Lambda^{{\mathcal I}^{\sigma_*}}$ and $\Lambda^{{\mathcal N}_0}$ each contain a single repeated eigenvalue of multiplicity three 
 as ${\mathcal R}^{\sigma_*}$, ${\mathcal N}^0$ and ${\mathcal I}^{\sigma_*}(\omega) $ are each a multiple of ${\mathbb I}$. Numerical results for other object shapes can be found in~\cite{ledgerlionheart2016,ledgerlionheart2018,ledgerlionheart2018mathmeth}.
\begin{figure}[!h]
\begin{center}
$\begin{array}{c}
\includegraphics[width=0.6\textwidth]{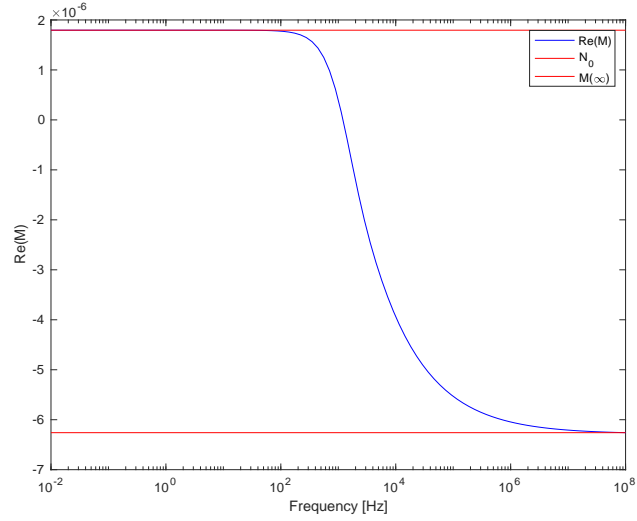} \\
 (a) \\
 \includegraphics[width=0.6\textwidth]{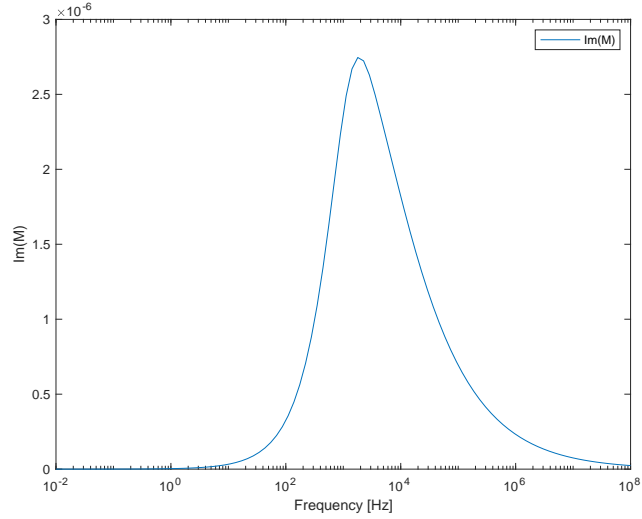} \\
 (b)  
\end{array}$
\end{center}
\caption{Conducting sphere with  $\alpha=0.01\text{m}$, $\mu_r=1.5$ and $\sigma_*=5.96 \times 10^6 \text{S/m}$: behaviour of $(a)$ $\text{Re} ( {\mathcal M}(\omega)) ={\mathcal N}^0 + {\mathcal R}^{\sigma_*}(\omega)$ and $(b)$ $\text{Im} ( {\mathcal M} (\omega) )$.} \label{fig:sphere}
\end{figure}

The matrices of eigenvalues $\Lambda^{{\mathcal R}^{\sigma_*}}$ and  $\Lambda^{{\mathcal I}^{\sigma_*}}$ are  strongly dependent on $\nu= \alpha^2 \sigma_* \omega \mu_0 $. The limiting behaviour of $ {\mathcal M}- {\mathcal N}^0 =  {\mathcal R}^{\sigma_*} + \im {\mathcal I}^{\sigma_*} = -\widecheck{\mathcal C}^{\sigma_*}+ {\mathcal N}^{\sigma_*}$ for $\nu = 0$ and $\nu \to \infty$ has already been investigated and we recall that
\begin{itemize}
\item ${\mathcal R}^{\sigma_*} \to 0$, ${\mathcal I}^{\sigma_*} \to 0$ as $\nu \to 0$ and hence $\Lambda^{{\mathcal R}^{\sigma_*} }\to 0$, $\Lambda^{{\mathcal I}^{\sigma_*}} \to 0$ as $\nu \to 0$;
\item ${\mathcal I}^{\sigma_*} \to 0$ as $\nu \to \infty$ and hence since ${\mathcal N}^0 +{\mathcal R}^{\sigma_*} \to {\mathcal M}(\infty)$  as $\nu \to \infty$ by Lemma~\ref{lemma:limitcase} then $\Lambda^{{\mathcal I}^{\sigma_*}} \to 0$ and $\Lambda^{{\mathcal N}^0 + {\mathcal R}^{\sigma_*} }\to\Lambda^{ {\mathcal M}(\infty)}$  as $\nu \to \infty$. If $\beta_1(B)=0$ then ${\mathcal M}(\infty)$ simplifies to ${\mathcal T}(0)$ and $\Lambda^{{\mathcal N}^0+{\mathcal R}^{\sigma_*} }\to\Lambda^{ {\mathcal T}(0)}$ .
\end{itemize}
Throughout this section we require that $\sigma_*=\sigma_*^{(n)}$ is constant throughout ${\vec B}$, but allow $\mu_*$ to still vary in a piecewise constant manner through  ${\vec B}$.
With this, and the above in mind,  it is beneficial to consider the dependence of ${\mathcal M}$, ${\mathcal R}^{\sigma_*}$ and ${\mathcal I}^{\sigma_*}$ on $\nu$ from which their behaviour with $\omega$ can be readily obtained by a simple change of variables. As explained previously in Section ~\ref{sect:limitcase}, our interest lies in the case in understanding the behaviour of these tensors where $\nu= O(1)$ so as not to invalidate~(\ref{eqn:asymp}).
We begin by investigating the behaviour of ${\vec \Theta}^{(1)} ({\vec u})$  with $\nu$. 

\subsection{Spectral behaviour of ${\vec \Theta}^{(1)} ({\vec u})$ with $\nu$}
We introduce the model eigenvalue problem: Find the eigenvalue--eigensolution pairs $(\lambda,{\vec \phi})$ such that
\begin{subequations} \label{eqn:modelprob}
\begin{align}
\nabla \times \mu_r^{-1} \nabla \times {\vec \phi} =& \lambda {\vec \phi} && \text{in ${\vec B}$}, \\
\nabla \times  \nabla \times {\vec \phi} =&  {\vec 0} && \text{in ${\vec B} ^c$}, \\
\nabla \cdot {\vec \phi} = & 0 && \text{in ${\vec B} \cup {\vec B}^c$}, \\
[ {\vec n} \times {\vec \phi} ]_\Gamma = & {\vec 0} && \text{ on $\Gamma$} , \\
[ {\vec n} \times \tilde{\mu}_r^{-1} \nabla \times {\vec \phi} ]_\Gamma  = & {\vec 0} && \text{ on $\Gamma$} , \\
{\vec \phi} = & O(|{\vec \xi}|^{-1}) && \text{as $|{\vec \xi}| \to \infty$} ,
\end{align}
\end{subequations}
which we will show is closely related to understanding the behaviour of ${\vec \Theta}^{(1)}( {\vec u}) $. The model eigenvalue problem can be written in weak form as: Find ${\vec \phi} \in Y$ and $\lambda$ such that
\begin{align}
\left < \mu_r^{-1} \nabla \times {\vec \phi} , \nabla \times {\vec \psi} \right >_{L^2( {\vec B} )}  = \lambda \left <  {\vec \phi} , {\vec \psi} \right >_{L^2( {\vec B} )}\qquad \forall {\vec \psi} \in Y
\label{eqn:wkform1},
\end{align}
where 
\begin{align}
Y:= & \{ {\vec \varphi} \in H(\hbox{curl}): \nabla \times \nabla \times {\vec \varphi} = {\vec 0} \text{ in ${\vec B}^c$}, \nabla \cdot {\vec \varphi}=0 \text{ in ${\vec B} \cup {\vec B}^c$}, {\vec \varphi} = O(|{\vec \xi}|^{-1} ) \text{ as $|{\vec \xi}| \to \infty$}  \} . \nonumber 
\end{align}
To analyse (\ref{eqn:wkform1}), it is useful to apply a Helmholtz decomposition~\cite[pg. 86]{monkbook} to $Y$:
\begin{align}
Y=& Y_0 \oplus \nabla S, \qquad
Y_0:=  \{ {\vec \phi} \in Y:  \left < {\vec \phi},\nabla p \right >_{L^2({\vec B})}=0 \qquad \forall p \in S \}, \nonumber
\end{align}

and, based on treatment of a similar problem in~\cite[pg. 96]{monkbook}, we summarises its properties in the following remark.
\begin{remark} \label{remark:monkremark}
%
 Repeating similar arguments to those of Monk~\cite[pg 96.]{monkbook}, the eigenvalue problem (\ref{eqn:wkform1}) can be investigated using a Helmholtz decomposition
 \begin{equation}
 {\vec \phi}({\vec \xi}) = {\vec \phi}_0 + \nabla p \qquad \text{where } {\vec \phi}_0 \in Y_0, p \in S,
 \end{equation}
for  ${\vec \xi} \in {\vec B}$. Corresponding to the eigenvalue $\lambda=0$, then it can be shown that ${\vec \phi}_0={\vec 0}$ and there are an infinite number of gradient eigenfunctions in ${\vec B}$. Corresponding to $\lambda \ne 0$ the problem (\ref{eqn:wkform1}) can be rewritten as: Find ${\vec \phi}_0 \in  Y_0$ and $\lambda\ne 0$  such that
 \begin{align}
\left < \mu_r^{-1} \nabla \times {\vec \phi}_0 , \nabla \times {\vec \psi} \right >_{L^2 ( {\vec B})}  = \lambda \left <  {\vec \phi}_0 , {\vec \psi} \right >_{L^2({\vec B})} \qquad \forall {\vec \psi} \in Y_0
\label{eqn:wkform2}.
\end{align}
 Choosing ${\vec \psi}={\vec \phi}_0$ in (\ref{eqn:wkform2}) it is possible to show that  $\lambda > 0$ . Continuing to follow Monk, then, by introducing an appropriate solution operator, the existence of eigenvalues and eigenfunctions can be established using the Hilbert-Schmidt theory leading to the following conclusions:
\begin{enumerate}
\item Corresponding to the eigenvalue $\lambda=0$ there is an infinite family of eigenfunctions, which are such that ${\vec \phi} = \nabla p$ in ${\vec B}$ for any $p \in S$.
\item There is an infinite discrete set of eigenvalues $\lambda_j > 0$, $j=1,2\, \ldots$ and corresponding eigenfunctions ${\vec \phi}_j\in Y_0$, ${\vec \phi}_j \ne 0$ such that
\begin{itemize}
\item Problem (\ref{eqn:wkform1}) is satisfied,
\item $0 < \lambda_1 < \lambda_2 < \ldots$,
\item $\lim_{j \to \infty} \lambda_j = \infty$,
\item ${\vec \phi}_j$ is orthogonal to ${\vec \phi}_k$ in the $L^2({\vec B})$ inner product $\left < {\vec \phi}_j , {\vec \phi}_k \right >_{L^2({\vec B})} = \delta_{jk}$.
\end{itemize}
\end{enumerate}
\end{remark}
 Using these properties we can deduce the following about ${\vec \Theta}^{(1)} ({\vec u})$:
\begin{lemma} \label{lemma:theta1rep}
The  weak solution to (\ref{eqn:transproblem1var}) for $\nu \in [0,\infty)$
can be expressed as the convergent series 
\begin{equation}
{\vec \Theta}^{(1)} ({\vec u})= - \sum_{n=1}^\infty \frac{\im \nu }{\im \nu  - \lambda_n} P_n({\vec \Theta}^{(0)}({\vec u}) ) = \sum_{n=1}^\infty \beta_n P_n({\vec \Theta}^{(0)} ({\vec u}) ) , \ \
\beta_n:= - \frac{\im \nu }{\im \nu  - \lambda_n}, \label{eqn:spectraltheta1}
\end{equation}
where $P_n({\vec \Theta}^{(0)} ({\vec u})) = {\vec \phi}_n \left <   {\vec \Theta}^{(0)}({\vec u}) , {\vec \phi}_n \right >_{L^2({\vec B})}$,  $(\lambda_n,{\vec \phi}_n)$ satisify (\ref{eqn:wkform1}) and
\begin{equation}
\mathrm{Re}( \beta_n) = - \frac{\nu^2}{\nu^2+ \lambda_n^2}, \qquad \mathrm{Im} (\beta_n) = \frac{\nu \lambda_n}{\nu^2+ \lambda_n^2} .
\nonumber
\end{equation}
\end{lemma}

\begin{proof}
We begin by defining $a: Y \times Y \to {\mathbb C}$ by
\begin{align}
a ({\vec f},{\vec g}):=& \left < \mu_r^{-1} \nabla \times {\vec f } , \nabla \times {\vec g } \right >_{L^2 ({\vec B})} - \im\nu \left < {\vec f},{\vec g} \right >_{L^2({\vec B})}\nonumber .
\end{align}
For $ \nu \in (0, \infty)$, it is clear that  $\im \nu$ is not an eigenvalue of (\ref{eqn:wkform1}) and, hence,
by Corollary 4.19 in Monk~\cite[pg. 98]{monkbook} the problem: Find ${\vec \Theta}^{(1)} ({\vec u}) \in Y$ such that
\begin{equation}
a ({\vec \Theta}^{(1)}({\vec u}),{\vec \psi}) =\im \nu  \left <  {\vec \Theta}^{(0)}({\vec u}),{\vec \psi} \right >_{L^2 ({\vec B} )} \qquad \forall {\vec \psi} \in Y,
\end{equation}
has a unique solution for every ${\vec \Theta}^{(0)}({\vec u})  \in L^2({\vec B}) \subset Y'$.
Defining the operator $(L - \im \nu {\mathbb I} ) : Y \to Y'$, this problem consists of finding the solution to the operator equation
\begin{equation}
(L - \im \nu {\mathbb I} ) {\vec \Theta}^{(1)}  = \im \nu  {\vec \Theta}^{(0)} \label{eqn:lminnuopeq} .
\end{equation}
Writing $A= (L -\im \nu  {\mathbb I})^{-1}:  Y' \to Y$ for the solution operator defined by
\begin{equation}
\left < \mu_r^{-1} \nabla \times (A{\vec f}), \nabla  \times {\vec g}  \right > _{L^2( {\vec B})}  - \im \nu \left <  A {\vec f} , {\vec g} \right >_{L^2({\vec B})} =
\left < {\vec f},{\vec g} \right >_{L^2({\vec B})} \qquad \forall {\vec g} \in Y \nonumber ,
\end{equation}
then it clear that $A$ is linear and we can check that it is self adjoint:
\begin{align}
\left < {\vec f},A {\vec g} \right >_{L^2({\vec B})} =& \left < \mu_r^{-1} \nabla \times (A{\vec f}), \nabla  \times (A{\vec g} ) \right > _{L^2( {\vec B})} - \im \nu \left < A {\vec f} , A{\vec g} \right >_{L^2({\vec B})}  \nonumber \\
=&\left < \mu_r^{-1} \nabla \times \overline{(A{\vec g})}, \nabla  \times \overline{(A{\vec f} )} \right > _{L^2( {\vec B})} - \im \nu \left <  \overline{A {\vec g}} , \overline{A{\vec f}} \right >_{L^2( {\vec B})} \nonumber \\
=& \left < \overline{\vec g},\overline{A {\vec f}} \right >_{L^2({\vec B})} = \overline{ \left <  {\vec g}, A {\vec f}\right >_{L^2({\vec B})} }=\left < A{\vec f} , {\vec g}\right >_{L^2( {\vec B})} \nonumber .
\end{align}
Also, using the spectral behaviour of $L$ from  Remark~\ref{remark:monkremark}, we have $L {\vec \phi}_n = \lambda_n  {\vec \phi}_n$, thus, $ (L- \im \nu{\mathbb I} ) {\vec \phi}_n = (\lambda_n - \im \nu) {\vec \phi}_n$ and, hence, $A{\vec \phi}_n =(\lambda_n - \im \nu)^{-1}{\vec \phi}_n$.
Furthermore, as $A$ is linear and self adjoint, the spectral theorem applies  to $A$, which, when combined with (\ref{eqn:lminnuopeq}), leads immediately to (\ref{eqn:spectraltheta1}).  We can extend its applicability to $ \nu\in [0,  \infty )$ since we know that ${\vec \Theta}^{(1)}({\vec u})$ vanishes for $\nu=0$.  Then, we introduce $\beta_n:= - {\im \nu }/ ( {\im \nu  - \lambda_n})$ and its real and imaginary parts are trivially computed.

To show that the series converges, we expand ${\vec \Theta}^{(0)}({\vec u})$ in terms of ${\vec \phi}_n$ as
\begin{align}
{\vec \Theta}^{(0)}({\vec u})= \sum_{n=1}^\infty {\vec \phi}_n \left <  {\vec \Theta}^{(0)} ({\vec u}), {\vec \phi}_n \right >_{L^2({\vec B})} \nonumber,
\end{align}
from which it follows that
\begin{align}&
\| \nabla \times \tilde{\mu}_r^{-1} \nabla \times {\vec \Theta}^{(0)} ({\vec u}) \|_{L^2({\vec B}\cup {\vec B}^c)}^2 =
\| \nabla \times {\mu}_r^{-1} \nabla \times {\vec \Theta}^{(0)} ({\vec u}) \|_{L^2({\vec B} )}^2 = \nonumber \\ 
&=\sum_{n=1}^\infty \sum_{m=1}^\infty \left <  {\vec \Theta}^{(0)} ({\vec u}), {\vec \phi}_n \right > _{L^2( {\vec B})}  \left < {\vec \Theta}^{(0)} ({\vec u}), {\vec \phi}_m \right >_{L^2( {\vec B})}  \int_{\vec B}
\nabla \times \mu_r^{-1} \nabla \times{\vec \phi}_n \cdot \nabla \times \mu_r^{-1}\nabla \times {\vec \phi}_m \dif {\vec \xi} \nonumber \\
& = \sum_{n=1}^\infty \sum_{m=1}^\infty \left <  {\vec \Theta}^{(0)} ({\vec u}), {\vec \phi}_n \right > _{L^2( {\vec B})}  \left <  {\vec \Theta}^{(0)} ({\vec u}), {\vec \phi}_m \right >_{L^2({\vec B})}   \lambda_n \lambda_m \left <  {\vec \phi}_n , {\vec \phi}_m \right >_{L^2({\vec B})} \nonumber\\
& = \sum_{n=1}^\infty \lambda_n^2 \left < {\vec \Theta}^{(0)} ({\vec u}), {\vec \phi}_n \right >_{L^2({\vec B})}^2 <C, \label{eqn:bdinnerprodthephi}
\end{align}
since $\| \nabla \times \tilde{\mu}_r^{-1} \nabla \times {\vec \Theta}^{(0)} ({\vec u}) \|_{L^2({\vec B}\cup {\vec B}^c)}$ is bounded and $\left <  {\vec \phi}_n , {\vec \phi}_m \right >_{L^2({\vec B})} =\delta_{mn}$. Hence, 
\begin{align}
\left | \left <  {\vec \Theta}^{(0)} ({\vec u}), {\vec \phi}_n \right >_{L^2({\vec B})} \right |< O(\lambda_n^{-s/2}),  \nonumber
\end{align}
 as $n \to \infty $ with $s >2$  and $C>0$ independent of $\lambda_n$. Combining this with $|{\vec \phi}_n| < C \lambda_n \| {\vec \phi}_n \|_{L^2({\vec B})} < C\lambda_n $, which follows, for example, from using an analogous result to that in Proposition 3.1 of~\cite{filoche}, and
\begin{equation}
\beta_n= - \frac{ (\nu/\lambda_n)^2 }{(\nu/\lambda_n)^2 +1} + \im  \frac{ \nu/\lambda_n }{(\nu/\lambda_n)^2 +1} ,  \nonumber
\nonumber
\end{equation}
 then we have that
\begin{equation}
| \beta_n P_n({\vec \Theta}^{(0)} ({\vec u}) )| \le |\beta_n|  | {\vec \phi}_n |  \left |  \left <  {\vec \Theta}^{(0)} ({\vec u}), {\vec \phi}_n \right >_{L^2({\vec B})} \right |  < C |\beta_n| \lambda_n^{1-s/2}. \nonumber 
\end{equation}
This estimate goes to zero as $n \to \infty$ and, hence, (\ref{eqn:spectraltheta1}) converges.
\end{proof}
 
 \begin{corollary} \label{coll:linkderiv}
 From the definition of $\beta_n$ in Lemma~\ref{lemma:theta1rep} it follows that
 \begin{equation}
 \frac{\dif }{\dif \log \nu} ( \mathrm{Re}(\beta_n)) = -2\frac{\lambda_n^2 \nu^2}{(\nu^2+\lambda_n^2)^2} = -2\left ( \mathrm{Im}(\beta_n)  \right )^2 \label{eqn:deriv1},
 \end{equation}
 and
  \begin{equation}
 \frac{\dif ^2}{\dif (\log \nu)^2} ( \mathrm {Re}(\beta_n)) = -\frac{4\lambda_n^2\nu^2(\lambda_n^2-\nu^2)}{(\nu^2+\lambda_n^2)^3} = -4\left ( \mathrm {Im}(\beta_n)  \right )\frac{\dif }{\dif \log \nu} ( \mathrm {Im}(\beta_n)) \label{eqn:deriv2}.
 \end{equation}
 \end{corollary}
 \begin{remark}
 The complex functions $\beta_n(\nu)$, $n=1,2,\ldots$, characterise the behaviour of ${\vec \Theta}^{(1)}({\vec u})$ with respect to $\nu$. The real part of each function,  $\mathrm{Re}( \beta_n)$, is monotonic and bounded with $\log \nu$ and the imaginary part of each function, $\mathrm{Im}( \beta_n)$, has a single local maximum with $\log \nu$. 
\end{remark}

 \subsection{Spectral behaviour of ${\mathcal R}^{\sigma_*}$ and ${\mathcal I}^{\sigma_*}$ with $\nu$}
The following Lemma, which describes the behaviour of ${\mathcal R}^{\sigma_*}$ and ${\mathcal I}^{\sigma_*}$ with $\nu$, follows from the representation of ${\vec \Theta}^{(1)}({\vec u})$ provided by Lemma~\ref{lemma:theta1rep}.

\begin{lemma} \label{lemma:tenrepintheta1sum}
The coefficients of the tensors ${\mathcal R}^{\sigma_*}$ and ${\mathcal I}^{\sigma_*}$ for an object with homogeneous $\sigma_*$, although not necessarily homogenous $\mu_*$, can be expressed as the convergent series
\begin{subequations}
\begin{align}
({\mathcal R}^{\sigma_*})_{ij} = &  - \frac{\alpha^3 \nu^2  }{4} \sum_{n=1}^\infty \frac{\lambda_n}{\nu^2+ \lambda_n^2}  \left <  {\vec \phi}_n ,{\vec \Theta}^{(0)}({\vec e}_i) \right >_{L^2( {\vec B})} \left <  {\vec \phi}_n ,{\vec \Theta}^{(0)}({\vec e}_j) \right >_{L^2({\vec B})} \nonumber \\
 =&  \frac{\alpha^3  }{4} \sum_{n=1}^\infty \mathrm{Re}(\beta_n) \lambda_n  \left < {\vec \phi}_n ,{\vec \Theta}^{(0)}({\vec e}_i) \right >_{L^2({\vec B})} \left < {\vec \phi}_n ,{\vec \Theta}^{(0)}({\vec e}_j) \right >_{L^2({\vec B})} , \label{eqn:expandrsigma}\\ 
({\mathcal I}^{\sigma_*})_{ij} = &   \frac{\alpha^3    \nu}{4} \sum_{n=1}^\infty \frac{\lambda_n^2}{\nu^2+ \lambda_n^2}  \left < {\vec u}_n ,{\vec \Theta}^{(0)}({\vec e}_i) \right >_{L^2({\vec B})} \left < {\vec \phi}_n ,{\vec \Theta}^{(0)}({\vec e}_j) \right >_{L^2({\vec B} )}  \nonumber \\
= &  \frac{\alpha^3   }{4} \sum_{n=1}^\infty \mathrm {Im}( \beta_n) \lambda_n  \left < {\vec \phi}_n ,{\vec \Theta}^{(0)}({\vec e}_i) \right >_{L^2({\vec B})} \left < {\vec \phi}_n ,{\vec \Theta}^{(0)}({\vec e}_j) \right >_{L^2({\vec B})} .\label{eqn:expandisigma}%
\end{align}
\end{subequations}
\end{lemma}
\begin{proof}
Using Theorem~\ref{thm:realandimagpts} and Lemma~\ref{lemma:theta1rep} we see that $({\mathcal R}^{\sigma_*} )_{ij}$ can be expressed as
\begin{align}
({\mathcal R}^{\sigma_*} )_{ij}=&  - \frac{\alpha^3}{4}  \left (  \left <  \mu_r ^{-1}\nabla \times {\vec \Theta}^{(1)}({\vec e}_j) ,\nabla \times {{\vec \Theta}^{(1)}}({\vec e}_i) \right >_{L^2({\vec B})} +  \left < \nabla \times {\vec \Theta}^{(1)}({\vec e}_j) , 
\nabla \times {{\vec \Theta}^{(1)}} ({\vec e}_i) \right >_{L^2({\vec B}^c)}  \right ), \nonumber \\
= & - \frac{\alpha^3}{4}  \left (  \left < \sum_{n=1}^\infty \frac{\im \nu}{\im \nu - \lambda_n} \mu_r ^{-1}   \nabla \times P_n({\vec \Theta}^{(0)}({\vec e}_i)), \sum_{m=1}^\infty \frac{\im \nu}{\im \nu - \lambda_m}  \nabla \times  P_m({\vec \Theta}^{(0)}({\vec e}_j)) \right >_{L^2({\vec B})} \right .   \nonumber  \\
& +  \left .  \left < \sum_{n=1}^\infty \frac{\im \nu}{\im \nu - \lambda_n}  \nabla \times  P_n({\vec \Theta}^{(0)}({\vec e}_i)), \sum_{m=1}^\infty \frac{\im \nu}{\im \nu - \lambda_m} \nabla \times   P_m({\vec \Theta}^{(0)}({\vec e}_j)) \right >_{L^2({\vec B}^c)} \right ) \nonumber \\
= & - \frac{\alpha^3}{4} \left ( \sum_{n=1}^\infty  \sum_{m =1}^\infty \frac{\im \nu}{\im \nu - \lambda_n} \overline{ \frac{\im \nu}{\im \nu - \lambda_m}} \left ( \left <
  \mu_r ^{-1} \nabla \times  P_n({\vec \Theta}^{(0)}({\vec e}_i)),  \nabla \times  P_m({\vec \Theta}^{(0)}({\vec e}_j)) \right > _{L^2({\vec B})} \right .   \right.  \nonumber \\
& +  \left . \left.  \left < \nabla \times  P_n({\vec \Theta}^{(0)}({\vec e}_i)),\nabla \times  P_m({\vec \Theta}^{(0)}({\vec e}_j))  \right >_{L^2({\vec B}^c)} \right )\right )  \label{eqn:expandrform1} .
\end{align}
Then, noting that
\begin{align}
\left < \right . &\left . \mu_r ^{-1} \nabla \times  P_n({\vec \Theta}^{(0)}({\vec e}_i)),  \nabla \times  P_m({\vec \Theta}^{(0)}({\vec e}_j))  \right >_{L^2({\vec B})} +   \left < \nabla \times  P_n({\vec \Theta}^{(0)}({\vec e}_i)),\nabla \times  P_m({\vec \Theta}^{(0)}({\vec e}_j))  \right >_{L^2({\vec B}^c)}  \nonumber \\
  =&  \left <  {\vec \phi}_n , {\vec \Theta}^{(0)}({\vec e}_i) \right >_{L^2({\vec B})}  \left < {\vec \phi}_m , {\vec \Theta}^{(0)}({\vec e}_j)\right >_{L^2({\vec B})}\nonumber\\
 & \left (  \left < \mu_r^{-1} \nabla \times {\vec \phi}_n , \nabla \times{\vec \phi}_m \right >_{L^2({\vec B})} +  \left < \nabla \times {\vec \phi}_n ,\nabla \times {\vec \phi}_m \right >_{L^2( {\vec B}^c) }\right ) \nonumber \\
  = &\left <  {\vec \phi}_n , {\vec \Theta}^{(0)}({\vec e}_i) \right >_{L^2({\vec B})}  \left < {\vec \phi}_m , {\vec \Theta}^{(0)}({\vec e}_j) \right >_{L^2({\vec B})}  \lambda_n  \left < {\vec \phi}_n  , {\vec \phi}_m  \right >_{L^2({\vec B})} \nonumber\\
  = &\left < {\vec \phi}_n , {\vec \Theta}^{(0)}({\vec e}_i ) \right >_{L^2({\vec B})}  \left < {\vec \phi}_m , {\vec \Theta}^{(0)}({\vec e}_j) \right >_{L^2({\vec B})}  \lambda_n \delta_{nm}  , \label{eqn:eigtranformrule}
\end{align}
and combining with (\ref{eqn:expandrform1}), gives the desired result for $({\mathcal R}^{\sigma_*})_{ij}$. For $({\mathcal I}^{\sigma_*})_{ij} $ we have
\begin{align}
({\mathcal I}^{\sigma_*} )_{ij}=&   \frac{\alpha^3}{4\nu} \left <  \nabla \times \mu_r ^{-1} \nabla \times {\vec \Theta}^{(1)} ({\vec e}_i),\nabla \times \mu_r^{-1} \nabla \times {\vec \Theta}^{(1)} ({\vec e}_j) \right >_{L^2({\vec B})}     \nonumber \\
= & \frac{\alpha^3}{4\nu}  \left < \sum_{n=1}^\infty \frac{\im \nu}{\im \nu - \lambda_n}\nabla \times \mu_r ^{-1} \nabla \times P_n( {\vec \Theta}^{(1)}({\vec e}_i)), \sum_{m=1}^\infty 
\frac{\im \nu}{\im \nu - \lambda_m}\nabla \times \mu_r ^{-1} \nabla \times P_m( {\vec \Theta}^{(1)     } ({\vec e}_i) ) \right >_{L^2({\vec B})}  , \nonumber
\end{align}
and using
\begin{align}
\nabla \times \mu_r ^{-1} \nabla \times P_n( {\vec \Theta}^{(1)} ({\vec e}_i))=& \left < {\vec \phi}_n,{\vec \Theta}^{(0)}({\vec e}_i)  \right >_{L^2({\vec B})} \nabla \times \mu_r ^{-1} \nabla \times {\vec \phi}_n\nonumber \\
= & \lambda_n  \left < {\vec \phi}_n,{\vec \Theta}^{(0)}({\vec e}_i)  \right >_{L^2({\vec B})} {\vec \phi}_n \qquad \text{in ${\vec B}$}, \nonumber
\end{align}
we have
\begin{align}
({\mathcal I}^{\sigma_*} )_{ij}=&   \frac{\alpha^3}{4\nu}   \sum_{n=1}^\infty \sum_{m =1}^\infty \frac{\im \nu}{\im \nu - \lambda_n} \overline{ \frac{\im \nu}{\im \nu - \lambda_m}}     \lambda_n \lambda_m  \left < {\vec \phi}_n,{\vec \Theta}^{(0)}({\vec e}_i) \right >_{L^2({\vec B})}   \left < {\vec \phi}_m,{\vec \Theta}^{(0)}({\vec e}_j) \right >_{L^2({\vec B})} \nonumber \\
&  \left < {\vec \phi}_n , {\vec \phi}_m \right >_{L^2({\vec B})}  . \nonumber
\end{align}
The final result for $({\mathcal I}^{\sigma_*})_{ij}$ follows from noting that $\left < {\vec \phi}_n , {\vec \phi}_m \right >_{L^2({\vec B})} = \delta_{mn}$.

The convergence of (\ref{eqn:expandrsigma}) and (\ref{eqn:expandisigma}) follows in a similar manner to that of (\ref{eqn:spectraltheta1}) by using 
\begin{align}
\left | \left <  {\vec \Theta}^{(0)} ({\vec u}), {\vec \phi}_n \right >_{L^2({\vec B})} \right | < O(\lambda_n^{-s/2} ) , \nonumber
\end{align}
 as $n \to \infty $ with $s >2$  and $C>0$ independent of $\lambda_n$.

\end{proof}
Taking in to account possible multiplicities in the eigenvalues  $\lambda_n$, we have the following:
\begin{remark} 
The result of Lemma~\ref{lemma:tenrepintheta1sum} can be rewritten to make explicit possible multiplicities in the eigenvalues  $\lambda_n$ as
\begin{subequations}
 \label{eqn:modesmult}
 \begin{align}
({\mathcal R}^{\sigma_*} )_{ij}= &   \frac{\alpha^3  }{4} \sum_{n=1}^\infty  \mathrm{Re}(\beta_n) \lambda_n \sum_{k=1}^{\mathrm{mult}(\lambda_n)} \left < {\vec \phi}_{n,k} ,{\vec \Theta}^{(0)}({\vec e}_i)  \right >_{L^2({\vec B})}   \left <  {\vec \phi}_{n,k} ,{\vec \Theta}^{(0)}({\vec e}_j)  \right >_{L^2({\vec B})} , \\
({\mathcal I}^{\sigma_*})_{ij} = &     \frac{\alpha^3   }{4} \sum_{n=1 }^\infty \mathrm{Im}( \beta_n) \lambda_n  \sum_{k=1}^{\mathrm{mult}(\lambda_n)}  \left < {\vec \phi}_{n,k} ,{\vec \Theta}^{(0)}({\vec e}_i)  \right >_{L^2({\vec B})} \left <  {\vec \phi}_{n,k} ,{\vec \Theta}^{(0)}({\vec e}_j)  \right >_{L^2({\vec B})}.
\end{align}
\end{subequations}
\end{remark}

We observe that Lemma~\ref{lemma:tenrepintheta1sum} provides a connection between the point of inflection of $({\mathcal R}^{\sigma_*})_{ij}$ with $\log \nu$ and the stationary point of $({\mathcal I}^{\sigma_*})_{ij}$ with $\log \nu$ as discussed in the following remark.
\begin{remark} \label{remark:statpoint}
Applying (\ref{eqn:deriv2}) to the results (\ref{eqn:modesmult})
then a point of inflection for $({\mathcal R}^{\sigma_*})_{ij}$ with $\log \nu$ corresponds to where
\begin{align}
 \frac{\dif ^2}{\dif (\log \nu)^2}  &(({\mathcal R}^{\sigma_*})_{ij}) =     - \alpha^3 \sum_{n=1}^\infty  \frac{\lambda_n^3\nu^2(\lambda_n^2-\nu^2)}{(\nu^2+\lambda_n^2)^3} \sum_{k=1}^{\mathrm{mult}(\lambda_n)} \left <  {\vec \phi}_{n,k} ,{\vec \Theta}^{(0)}({\vec e}_i)  \right > _{L^2({\vec B})} \left <  {\vec \phi}_{n,k} ,{\vec \Theta}^{(0)}({\vec e}_j)  \right >_{L^2({\vec B})} \nonumber \\
 =&
-\alpha^3 \sum_{n=1}^\infty  \lambda_n  \mathrm{Im}(\beta_n)  \frac{\dif }{\dif \log \nu} ( \mathrm{Im}(\beta_n)) \sum_{k=1}^{\mathrm{mult}(\lambda_n)}  \left <  {\vec \phi}_{n,k} ,{\vec \Theta}^{(0)}({\vec e}_i)  \right >_{L^2({\vec B} )} \left <  {\vec \phi}_{n,k} ,{\vec \Theta}^{(0)}({\vec e}_j)  \right >_{L^2({\vec B})} \nonumber \\
 =&0. \nonumber
\end{align}
Similarly, the stationary point for $({\mathcal I}^{\sigma_*})_{ij}$ with $\log \nu$ corresponds to where
\begin{align}
 \frac{\dif }{\dif (\log \nu)}  (({\mathcal I}^{\sigma_*})_{ij}) = &  \frac{ \alpha^3}{4} \sum_{n=1}^\infty  \frac{\lambda_n^2\nu(\lambda_n^2-\nu^2)}{(\nu^2+\lambda_n^2)^2} \sum_{k=1}^{\mathrm{mult}(\lambda_n)}  \left <  {\vec \phi}_{n,k} ,{\vec \Theta}^{(0)}({\vec e}_i) \right >_{L^2( {\vec B})}  \left <  {\vec \phi}_{n,k} ,{\vec \Theta}^{(0)}({\vec e}_j) \right >_{L^2({\vec B})} \nonumber \\
= &  -\frac{ \alpha^3}{16} \sum_{n=1}^\infty \frac{\lambda_n}{\mathrm{Im}(\beta_n)} 
 \frac{\dif^2 }{\dif ( \log \nu )^2 } ( \mathrm{Re}(\beta_n)) \nonumber \\
 & \sum_{k=1}^{\mathrm{mult}(\lambda_n)} \left <   {\vec \phi}_{n,k} ,{\vec \Theta}^{(0)}({\vec e}_i) \right >_{L^2({\vec B})} \left < {\vec \phi}_{n,k} ,{\vec \Theta}^{(0)}({\vec e}_j) \right >_{L^2({\vec B})} \nonumber \\
  =& 0. \nonumber
 \end{align}
Thus, a stationary point for $({\mathcal I}^{\sigma_*})_{ij}$ with respect to $\log \nu$ corresponds to a point of inflection for $({\mathcal R}^{\sigma_*})_{ij}$  with respect to $\log \nu$.
\end{remark}

 \subsubsection{Dominant spectral behaviour of $( {\mathcal R}^{\sigma_*}(\nu))_{ij}$, $( {\mathcal I}^{\sigma_*}(\nu))_{ij}$}

From Corollary~\ref{coll:linkderiv}, we observe, for $i=j$, that the expressions  (\ref{eqn:modesmult}a) and (\ref{eqn:modesmult}b) involve sums of terms that are each monotonically decreasing and bounded with $\log \nu$ and have a single local maximum with $\log \nu$, respectively.
For $i\ne j$,  (\ref{eqn:modesmult}a) involves sums of terms that are either monotonically decreasing and bounded or monotonically increasing and bounded with $\log \nu$, and, (\ref{eqn:modesmult}b) has terms which have either a single local  minimum or maximum with $\log \nu$. 
The difference in the behaviour of the different terms for $i\ne j$ is due to 
\begin{align}
\sum_{k=1}^{\mathrm{mult}(\lambda_n)} \left < {\vec \phi}_{n,k} ,{\vec \Theta}^{(0)}({\vec e}_i)  \right >_{L^2({\vec B})}   \left <  {\vec \phi}_{n,k} ,{\vec \Theta}^{(0)}({\vec e}_j)  \right >_{L^2({\vec B})}, \nonumber
\end{align}
whose sign can vary for different $n$.
For each $i$, $j$  we expect, amongst the terms in these summations, there is a $n=n_{dom}$, which we call the dominant mode, that provides the dominant behaviour of $( {\mathcal R}^{\sigma_*}(\nu))_{ij}$ and  $({\mathcal I}^{\sigma_*}(\nu))_{ij}$ for $\nu \in [0,\nu_{max})$.
 We confirm this behaviour by using a least squares fit of the functions
 \begin{align}
 f^{({\mathcal R}^{\sigma_*})_{ij}}(a,b) = - \frac{ab\nu^2}{\nu^2+b^2} ,\qquad
 f^{({\mathcal I}^{\sigma_*})_{ij}}(c,d) =  \frac{cd\nu}{\nu^2+d^2} , \nonumber 
 \end{align}
 to the curves of $({\mathcal R}^{\sigma_*}(\nu)) _{ij}=( \mathrm{Re}({\mathcal M}(\nu)))_{ij} - ({\mathcal N}^0)_{ij}$ and $( {\mathcal I}^{\sigma_*}(\nu) )_{ij}= (\mathrm{Im}({\mathcal M}(\nu)))_{ij}$ where $a$ and $c$ control the amplitude and sign of the functions and we expect to find that $b\approx d$ corresponds to the dominant eigenvalue $\lambda_{n_{dom}}$ for the considered coefficient. 
 
First, we consider the conducting sphere previously shown in Figure~\ref{fig:sphere}. For this object, ${\mathcal R}^{\sigma_*}(\nu)$ and $ {\mathcal I}^{\sigma_*}(\nu)$ are diagonal and a multiple of ${\mathbb I}$. We also expect the dominant mode to be $n_{dom}=1$, which has an eigenvalue with multiplicity 3. By fitting the functions $f^{({\mathcal R}^{\sigma_*})_{ii} }(a,b)$ and $f^{({\mathcal I}^{\sigma_*})_{ii}}(c,d)$ to the exact data (no summation implied) for 
$f\in [0, 10^4) \text{Hz}$, where  $f_{max} = \omega_{max}/ (  2 \pi ) =10^4 \text{Hz}$, which implies $\nu_{max} =\alpha^2 \sigma_* \mu_0 \omega_{max} \approx 47$, we find $b\approx d\approx10$, as expected. In Figure~\ref{fig:fitcurvesphere}, we observe that the functions provide a good approximation of $({\mathcal R}^{\sigma_*}(\nu))_{ii}$ and 
$({\mathcal I}^{\sigma_*}(\nu))_{ii}$. Also included are the residuals $- | ({\mathcal R}^{\sigma_*}(\nu)) _{ii} -  f^{( {\mathcal R}^{\sigma_*} )_{ii}}(a,b)|$ and $ | ({\mathcal I}^{\sigma_*}(\nu)) _{ii} -  f^{( {\mathcal I}^{\sigma_*})_{ii} }(a,b)|$, which are small for $\nu \in [0,\nu_{max})$.
 
 \begin{figure}[!h]
 \begin{center}
 $\begin{array}{c}
 \includegraphics[width=0.6\textwidth]{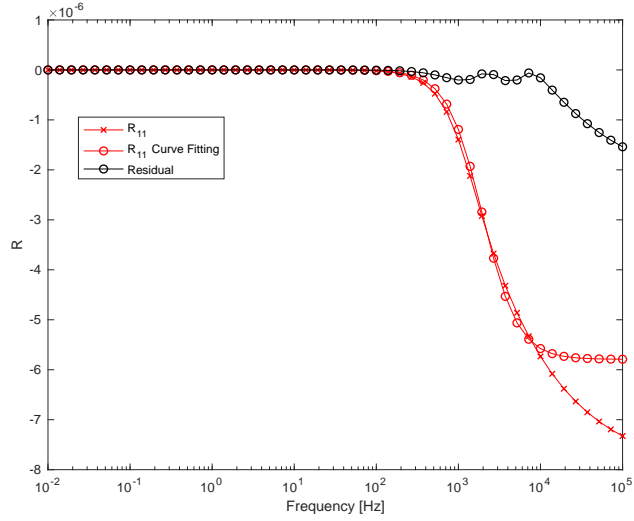} \\
 (a) \\
\includegraphics[width=0.6\textwidth]{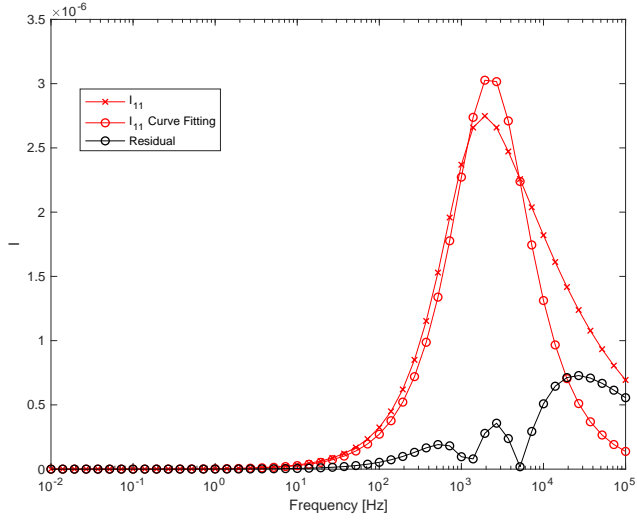} \\
(b)
\end{array}$
\end{center}
\caption{Conducting sphere with  $\alpha=0.01\text{m}$, $\mu_r=1.5$ and $\sigma_*=5.96 \times 10^6 \text{S/m}$: Curve fitting of $(a)$  $({\mathcal R}^{\sigma_*}(\nu)_{ii} = (\mathrm{Re}({\mathcal M}(\nu)))_{ii} - ({\mathcal N}^0)_{ii}$ and $(b)$ $({\mathcal I}^{\sigma_*}(\nu) )_{ii}=( \mathrm{Im}({\mathcal M}(\nu)))_{ii}$, (no summation implied, $i=1,2,3$ are identical).} \label{fig:fitcurvesphere}
\end{figure}
 
 As a second example, we consider an irregular conducting tetrahedron $B_\alpha = \alpha B$ where the object $B$ has vertices $(0,0,0)$, $(0.7,0,0)$, $(0.89,0.46,0)$ and $(1.36,1.33,1.62)$,  $\alpha=0.01\text{m}$, $\mu_r=1.5$ and $\sigma_*=5.96 \times 10^6 \text{S/m}$. For this object, ${\mathcal R}^{\sigma_*}(\nu)$ and $ {\mathcal I}^{\sigma_*}(\nu)$ have $6$ independent coefficients and, therefore, for each coefficient, the dominant mode may differ. 
The functions $f^{( {\mathcal R}^{\sigma_*})_{ij} }(a,b)$ and $f^{( {\mathcal I}^{\sigma_*})_{ij} }(c,d)$ are fitted to the curves $({\mathcal R}^{\sigma_*}(\nu))_{ij}$ and $ ({\mathcal I}^{\sigma_*}(\nu))_{ij}$
obtained using the computational procedure described in~\cite{ledgerlionheart2014,ledgerlionheart2016} for $f\in [0,  10^5)  \text{Hz}$ using a mesh of 34\,473 unstructured tetrahedra and third order finite elements. Different values of $c \approx d$ are obtained for each coefficient and we observe, in Figure~\ref{fig:fitcurvetetdiag}, for the diagonal coefficients, and in  Figure~\ref{fig:fitcurvetetoffdiag},
 for the off-diagonal coefficients, 
 that the functions describe the dominant behaviour of $({\mathcal R}^{\sigma_*}(\nu))_{ij}$ and 
$({\mathcal I}^{\sigma_*}(\nu))_{ij}$ for $f\in [0, 10^5) \text{Hz}$, where  $f_{max} = \omega_{max}/ (  2 \pi ) =10^5 \text{Hz}$, which implies $\nu_{max} =\alpha^2 \sigma_* \mu_0 \omega_{max} \approx 470$.

  \begin{figure}[!h]
 \begin{center}
 $\begin{array}{c}
 \includegraphics[width=0.6\textwidth]{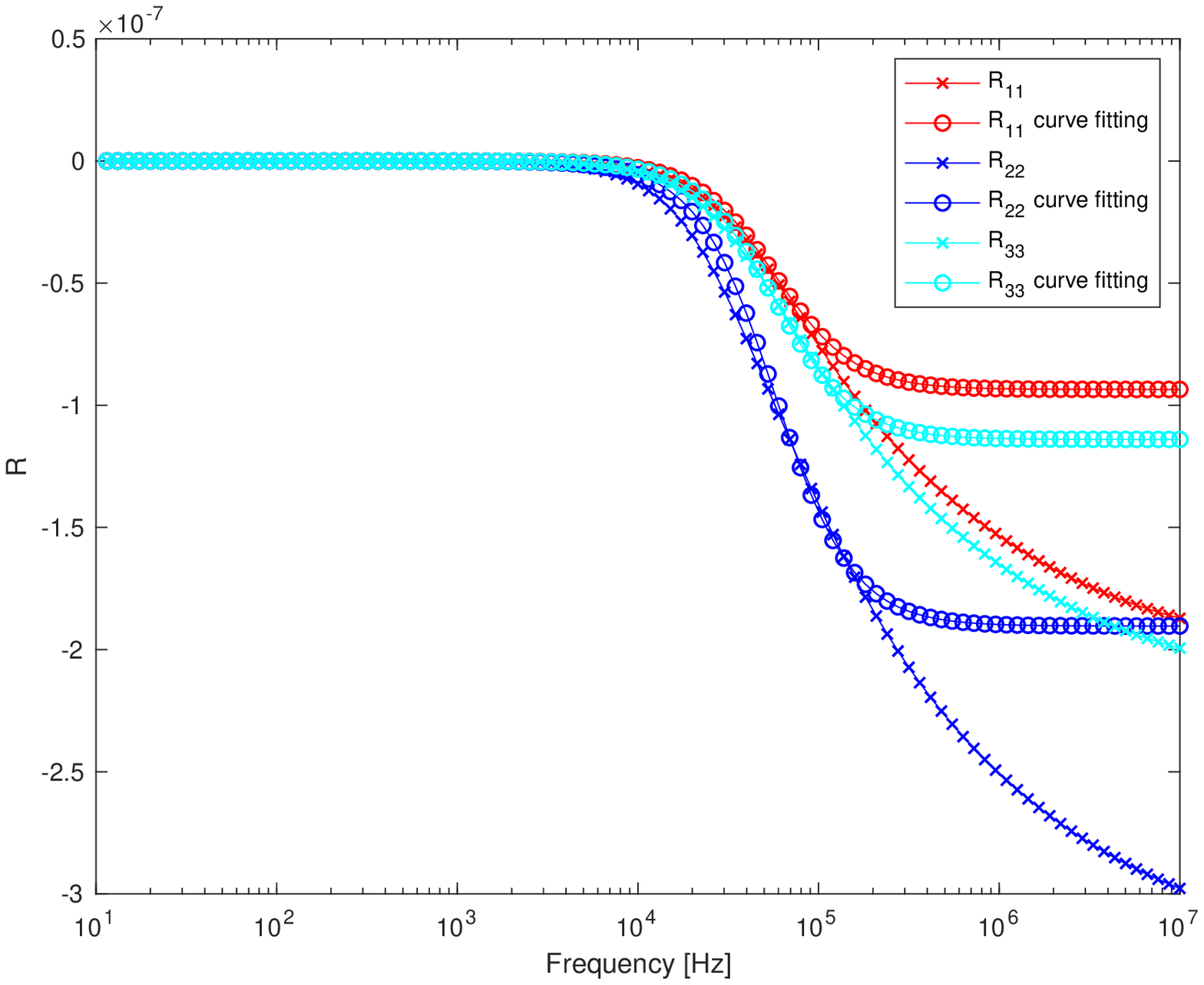} \\
 (a) \\
\includegraphics[width=0.6\textwidth]{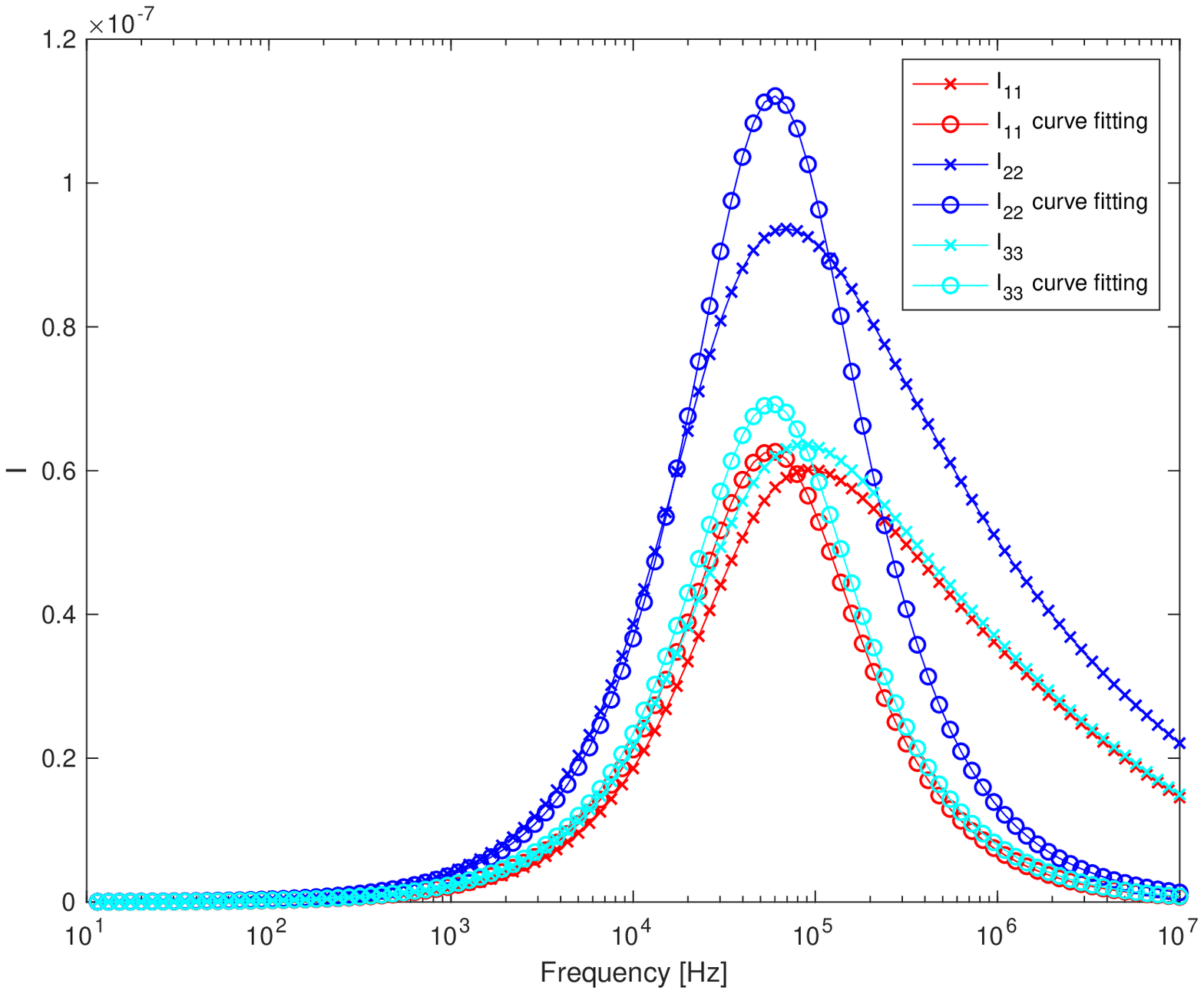} \\
(b)
\end{array}$
\end{center}
\caption{Conducting irregular tetrahedron with  $\alpha=0.01\text{m}$, $\mu_r=1.5$ and $\sigma_*=5.96 \times 10^6 \text{S/m}$: Curve fitting of $(a)$  ${\mathcal R}_{ii}^{\sigma_*}(\nu) = \mathrm{Re}({\mathcal M}(\nu))_{ii} - {\mathcal N}_{ii}^0$ and $(b)$  ${\mathcal I}_{ii}^{\sigma_*}(\nu) = \mathrm{Im}({\mathcal M}(\nu))_{ii}$, $i=1,2,3$ (no summation implied).} \label{fig:fitcurvetetdiag}
\end{figure}
 
   \begin{figure}[!h]
 \begin{center}
 $\begin{array}{c}
 \includegraphics[width=0.6\textwidth]{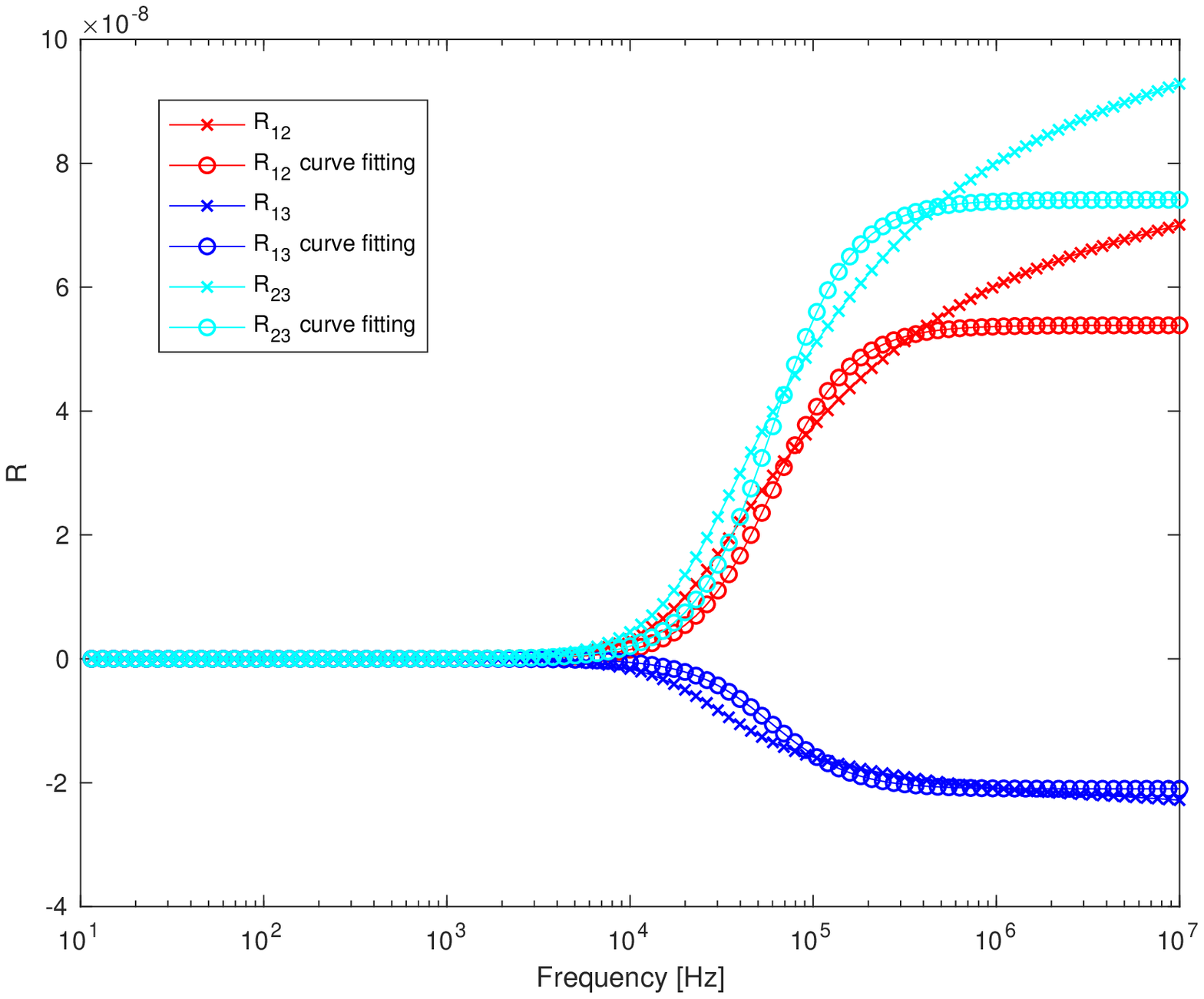} \\
 (a) \\
\includegraphics[width=0.6\textwidth]{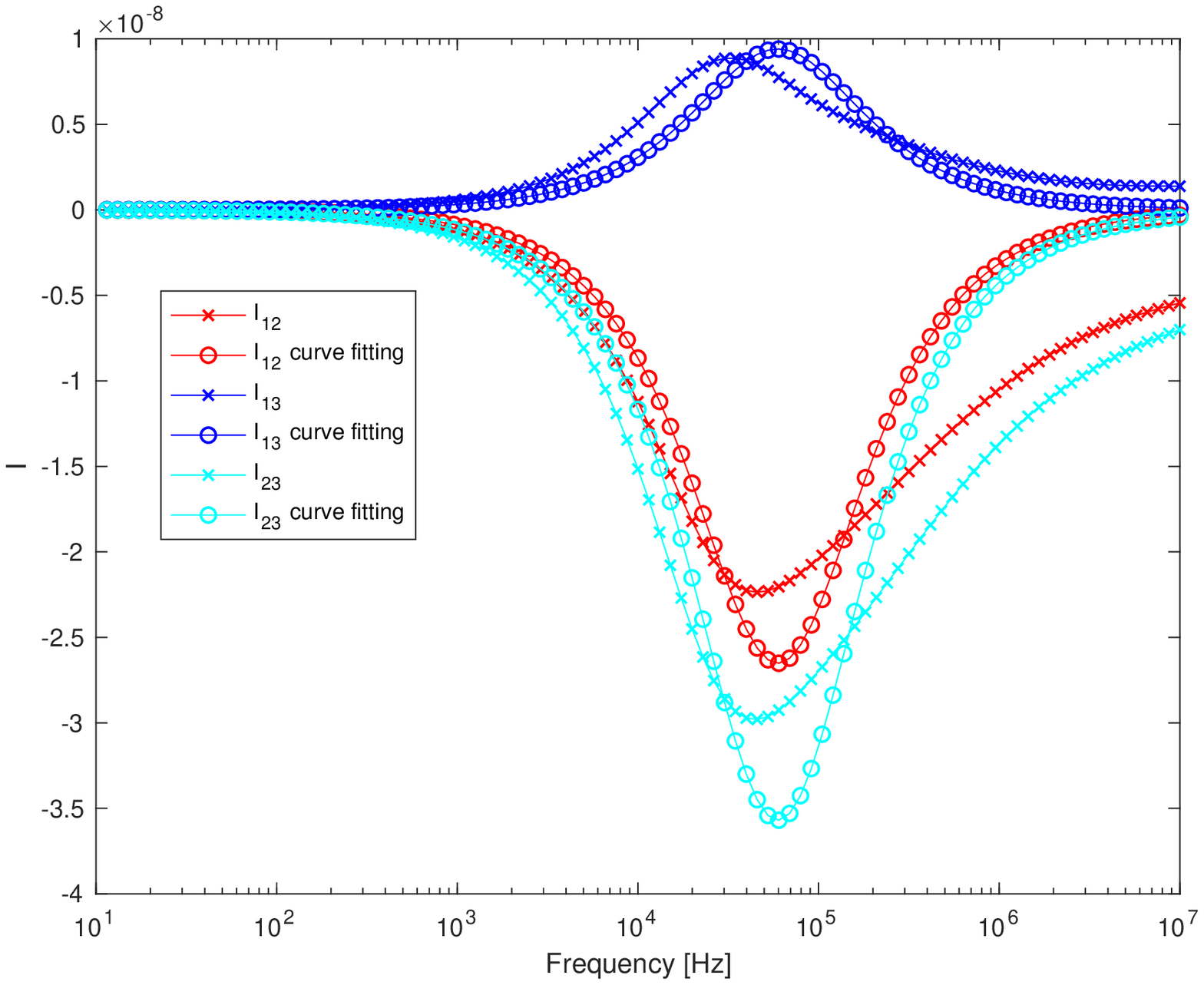} \\
(b) 
\end{array}$
\end{center}
\caption{Conducting irregular tetrahedron with  $\alpha=0.01\text{m}$, $\mu_r=1.5$ and $\sigma_*=5.96 \times 10^6 \text{S/m}$: Curve fitting of $(a)$  $({\mathcal R}^{\sigma_*}(\nu) )_{ij} = (\mathrm{Re}({\mathcal M}(\nu)))_{ij} -(  {\mathcal N}^0)_{ij}$ and $(b)$ $( {\mathcal I}^{\sigma_*}(\nu) )_{ij}=( \mathrm{Im}({\mathcal M}(\nu)))_{ij}$, $i\ne j $ .}
\label{fig:fitcurvetetoffdiag}
\end{figure}
%
%
%

The presence of dominant modes also provides further insights in to how  $({\mathcal R}^{\sigma_*})_{ii} $ and $({\mathcal I}^{\sigma_*})_{ii} $ are connected as described in the following remark.
\begin{remark}
For $\nu \in [0, \nu_{max})$ then, given a dominate mode $n_{dom}$, and applying  Corollary~\ref{coll:linkderiv} to Lemma~\ref{lemma:tenrepintheta1sum}, 
we have the following
\begin{align}
\left |  \frac{\dif}{\dif \log \nu} ( ({\mathcal R}^{\sigma_*})_{ii} )   \right | =  &  \frac{\alpha^3}{2}  \sum_{n=1}^\infty  (\mathrm{Im}(\beta_n)^2 \lambda_n \sum_{k=1}^{\mathrm{mult}(\lambda_n)} \left <  {\vec \phi}_{n,k} ,{\vec \Theta}^{(0)}({\vec e}_i) \right >_{L^2({\vec B})}^2   \nonumber \\
\approx &   \frac{\alpha^3}{2}   (\mathrm{Im}(\beta_{n_{dom}})^2 \lambda_{n_{dom}} \sum_{k=1}^{\mathrm{mult}(\lambda_{n_{dom}})} \left <  {\vec \phi}_{n_{dom},k} ,{\vec \Theta}^{(0)}({\vec e}_i) \right >_{L^2({\vec B})}^2\nonumber \\
\approx & C ({\mathcal I}^{\sigma_*})_{ii}^2, \label{eqn:correctrel}
\end{align}
where $C>0$ depends on $\lambda_{n_{dom}}$, $ {\vec \phi}_{n_{dom},k}$ and ${\vec \Theta}^{(0)}$, but is independent of $\nu$, which reveals insights in to how $\left |  \frac{\dif}{\dif \log \nu} ( ({\mathcal R}^{\sigma_*})_{ii} )   \right | $ and $({\mathcal I}^{\sigma_*})_{ii}^2$ are connected. 
From frequency sweeps of the computed tensor coefficients for different objects with homogenous $\sigma_*$ (e.g.~\cite{ledgerlionheart2016,ledgerlionheart2018,ledgerlionheart2018mathmeth} ), and from broadband measurements of tensorial coefficients (e.g.~\cite{rehimpeyton,davidson2018,norton2001}), $({\mathcal R}^{\sigma_*})_{ii}$ has been found to exhibit a monotonic and bounded behaviour with $\log \nu$ and ${\mathcal I}_{ii}^{\sigma_*} $ has a single local maximum with $\log \nu$ for a large range of objects. Thus, one might be tempted to conjecture that
\begin{equation}
\left | \frac{\dif}{\dif \log \nu} ( ({\mathcal R}^{\sigma_*})_{ii} ) \right | \le C   ({\mathcal I}^{\sigma_*})_{ii},  \nonumber
\end{equation}
however, this is not true, the correct behaviour being  of the type stated in (\ref{eqn:correctrel}).
 \end{remark}

\subsubsection{Reduction in the number of coefficients in $ {\mathcal R}^{\sigma_*}(\nu)$, $ {\mathcal I}^{\sigma_*}(\nu)$ due to object symmetries}

The important role played by $\left < {\vec \phi}_{n,k} ,{\vec \Theta}^{(0)}({\vec e}_i)  \right>_{L^2({\vec B})}   \left < {\vec \phi}_{n,k} ,{\vec \Theta}^{(0)}({\vec e}_j) \right >_{L^2({\vec B})}$ in the transformation of ${\mathcal R}^{\sigma_*}$ and ${\mathcal I}^{\sigma_*}$ is understood through the following lemma.

\begin{lemma} \label{lemma:transformintphitheta0}
Under the action of an orthogonal transformation matrix $Q$  
\begin{align}&
\left < {\vec \phi}_{n,k} ,{\vec \Theta}^{(0)}({\vec e}_i) \right >_{L^2({Q({\vec B})})}   \left <  {\vec \phi}_{n,k} ,{\vec \Theta}^{(0)}({\vec e}_j) \right >_{L^2({Q({\vec B})})}
\nonumber \\
&= (Q)_{ip} (Q)_{jq} \left <  {\vec \phi}_{n,k} ,{\vec \Theta}^{(0)}({\vec e}_p) \right >_{L^2({\vec B})}   \left < {\vec \phi}_{n,k} ,{\vec \Theta}^{(0)}({\vec e}_q) \right >_{L^2({\vec B})}, \label{eqn:transformintphitheta0}
\end{align}
transforms like the coefficients of a rank 2 tensor. Consequently, the coefficients of ${\mathcal R}^{\sigma_*}$ and ${\mathcal I}^{\sigma_*}$  expressed in the form (\ref{eqn:modesmult}) obey the transformations
\begin{align}
({\mathcal R}^{\sigma_*}[Q({\vec B})] )_{ij}= & (Q)_{ip} (Q)_{jq} ({\mathcal R}^{\sigma_*}[{\vec B}])_{pq}, \nonumber \\
({\mathcal I}^{\sigma_*}[Q({\vec B})])_{ij} = & (Q)_{ip} (Q)_{jq} ({\mathcal I}^{\sigma_*}[{\vec B}])_{pq}, \nonumber 
\end{align}
as expected.
\end{lemma}

\begin{proof}
Using the notation ${\vec \Theta}_{\vec B}^{(0)}({\vec u})$ to denote the solution of (\ref{eqn:transproblem0var}) and ${\vec \phi}_{n,k,{\vec B}}$ to denote the $n,k$ eigenmode of (\ref{eqn:modelprob}), where the dependence on ${\vec B}$ has been made explicit, we have, from Proposition 4.3 of~\cite{ammarivolkov2013b}, the transformations
\begin{align}
{\vec \Theta}_{Q({\vec B})}^{(0)}({\vec u}) = |Q| Q {\vec \Theta}_{\vec B}^{(0)}(Q^T{\vec u}), \qquad  {\vec \phi}_{n,k,Q({\vec B})} =|Q| Q {\vec \phi}_{n,k,{\vec B}} ,\nonumber
\end{align}
for an orthogonal transformation matrix $Q$. Observe that ${\vec \phi}_{n,k,R({\vec B})}$ does not depend on auxiliary vector and so its transformation is simpler. Following similar arguments to the proof of Theorem 3.1 of~\cite{ledgerlionheart2014} we have
\begin{align}
\left < {\vec \phi}_{n,k} ,{\vec \Theta}^{(0)}({\vec e}_i) \right >_{L^2({Q({\vec B})})}  = & \int_{Q({\vec B}) } {\vec \phi}_{n,k,Q({\vec B})} \cdot {\vec \Theta}_{Q({\vec B})} ^{(0)}({\vec e}_i) \dif {\vec \xi} \nonumber \\
= & |Q|^2 \int_{\vec B} Q_{pq} ( {\vec \phi}_{n,k,{\vec B}})_q Q_{ps} ( {\vec \Theta}_{{\vec B}} ^{(0)}(Q^T{\vec e}_i) )_s\dif {\vec \xi} \nonumber \\
= &   \delta_{qs} \int_{\vec B} ( {\vec \phi}_{n,k,{\vec B}})_q ( {\vec \Theta}_{{\vec B}} ^{(0)}(Q^T{\vec e}_i) )_s\dif {\vec \xi} \nonumber \\
= &   (Q)_{ip} \int_{\vec B}  {\vec \phi}_{n,k,{\vec B}}  \cdot {\vec \Theta}_{{\vec B}} ^{(0)}({\vec e}_p)  \dif {\vec \xi} , \nonumber 
\end{align}
where we have used ${\vec \Theta}_{\vec B}^{(0)}(Q^T{\vec e}_i) = \sum_{p=1}^3 (Q)_{ip} {\vec \Theta}_{\vec B}^{(0)}({\vec e}_p)$ in the final step. Repeating similar steps for $\left <  {\vec \phi}_{n,k} ,{\vec \Theta}^{(0)}({\vec e}_j) \right>_{L^2({Q({\vec B})})}$ gives the result in (\ref{eqn:transformintphitheta0}). On consideration of (\ref{eqn:modesmult}) the 
 transformations of the coefficients of  ${\mathcal R}^{\sigma_*}[Q({\vec B})]$ and ${\mathcal I}^{\sigma_*}[Q({\vec B})]$ immediately follow.
\end{proof}
\begin{remark}
Suppose, due to reflectional or rotational symmetries of an object, that $({\mathcal R}^{\sigma_*})_{ij}=0$ and $({\mathcal I}^{\sigma_*})_{ij}=0$ for some $i \ne j$. According to Lemma~\ref{lemma:transformintphitheta0}, we have already seen \\  $\left <  {\vec \phi}_{n,k} ,{\vec \Theta}^{(0)}({\vec e}_i) \right >_{L^2({\vec B})} \left <  {\vec \phi}_{n,k} ,{\vec \Theta}^{(0)}({\vec e}_j) \right>_{L^2({\vec B})}$ transforms like the coefficients of a rank 2 tensor. This then implies 
\begin{equation}
\sum_{k=1}^{\text{mult}(\lambda_n)} \left < {\vec \phi}_{n,k} ,{\vec \Theta}^{(0)}({\vec e}_i) \right >_{L^2( {\vec B})} \left < {\vec \phi}_{n,k} ,{\vec \Theta}^{(0)}({\vec e}_j) \right >_{L^2({\vec B})} =0,
\end{equation}
must hold for each $n$ to ensure that  (\ref{eqn:modesmult}) results in $({\mathcal R}^{\sigma_*})_{ij}=0$ and $({\mathcal I}^{\sigma_*})_{ij}=0$, independent of the object's materials and the frequency. It is impossible to have $\left <  {\vec \phi}_{n,k} ,{\vec \Theta}^{(0)}({\vec e}_i) \right >_{L^2({\vec B})} =0$ or $\left < {\vec \phi}_{n,k} ,{\vec \Theta}^{(0)}({\vec e}_j) \right >_{L^2({\vec B})} =0$  for all $ k$ since this would then imply that all of the $i$th row  or the $j$th column of the tensor was $0$, which contradicts Lemma~\ref{lemma:limitcase} where the diagonal coefficients only go to $0$ for extreme values.
Furthermore, this also implies that if we have a rotational, or reflectional symmetries resulting in  $({\mathcal R}^{\sigma_*})_{ij}=0$ and $({\mathcal I}^{\sigma_*})_{ij}=0$ for some $i \ne j$, then we must also have $\text{mult} (\lambda_n) \ge 2$ for all $n$.
\end{remark}

\subsubsection{Spectral behaviour of the eigenvectors of $ {\mathcal R}^{\sigma_*}(\nu)$, $ {\mathcal I}^{\sigma_*}(\nu)$}
For objects with rotational and/or reflectional symmetries, such that $ {\mathcal R}^{\sigma_*}$ and $ {\mathcal I}^{\sigma_*}$ are diagonal, then all of the coefficients of the commutators satisfy
\begin{subequations}\label{eqn:commutatordiag}
\begin{align}
 ( {\mathcal R}^{\sigma_*} (\nu))_{ij} ( {\mathcal I}^{\sigma_*} (\nu) )_{jk}-  ({\mathcal I}^{\sigma_*} (\nu) )_{ij}( {\mathcal R}^{\sigma_*} (\nu))_{jk}  =& 0,\\
 ({\mathcal R}^{\sigma_*} (\nu_1))_{ij} ( {\mathcal R}^{\sigma_*} (\nu_2) )_{jk}- ( {\mathcal R}^{\sigma_*} (\nu_2) )_{ij}( {\mathcal R}^{\sigma_*} (\nu_1))_{jk}  =& 0, \\
  ({\mathcal I}^{\sigma_*} (\nu_1))_{ij} ( {\mathcal I}^{\sigma_*} (\nu_2))_{jk} - ( {\mathcal I}_{ij}^{\sigma_*} (\nu_2))_{ij} ( {\mathcal I}^{\sigma_*} (\nu_1) )_{jk} =& 0 ,
\end{align}  %
\end{subequations} 
for any choice of $0<\nu_1,\nu_2,\nu< \infty$ and, hence, the eigenvectors of $ {\mathcal R}^{\sigma_*}(\nu_1)$ and  $ {\mathcal I}^{\sigma_*}(\nu_2)$ are the same for any $0<\nu_1,\nu_2< \infty$.

To understand how the eigenvectors of $ {\mathcal R}^{\sigma_*}$ and $ {\mathcal I}^{\sigma_*}$ for a general object can differ, we consider the following Lemma that provides estimates on the off-diagonal elements of the commutators using the alternative form of the tensors provided  by  Lemma~\ref{lemma:altformiandr}.  We note that it is easy to show that the diagonal elements of the commutators, corresponding to $i=k$ in (\ref{eqn:commutatordiag}), always vanish for any object. 

\begin{lemma} \label{lemma:commutator}
The off-diagonal elements of the commutators of
$ {\mathcal R}^{\sigma_*} (\nu ) $ and  $ {\mathcal I}^{\sigma_*} (\nu)$, $ {\mathcal R}^{\sigma_*} (\nu_1)  $ and   $ {\mathcal R}^{\sigma_*} (\nu_2)  $, as well as $ {\mathcal I}^{\sigma_*} (\nu_1) $ and  $ {\mathcal I}^{\sigma_*} (\nu_2) $, for $0 < \nu_1,\nu_2,\nu< \infty$ for $\nu_1\ne \nu_2$, for a general object, can be estimated as follows
\begin{subequations}
\begin{align}
\left |  ( {\mathcal R}^{\sigma_*} (\nu))_{ij} ( {\mathcal I}^{\sigma_*} (\nu) )_{jk}-  ({\mathcal I}^{\sigma_*} (\nu) )_{ij}( {\mathcal R}^{\sigma_*} (\nu))_{jk}  \right |  \le &
C\alpha^6 \nu , \label{eqn:bdcomrinew} \\
 \left | ({\mathcal R}^{\sigma_*} (\nu_1))_{ij} ( {\mathcal R}^{\sigma_*} (\nu_2) )_{jk}- ( {\mathcal R}^{\sigma_*} (\nu_2) )_{ij}( {\mathcal R}^{\sigma_*} (\nu_1))_{jk}   \right |
  \le & C\alpha^6,   \\
 \left |  ({\mathcal I}^{\sigma_*} (\nu_1))_{ij} ( {\mathcal I}^{\sigma_*} (\nu_2) )_{jk}- ( {\mathcal I}^{\sigma_*} (\nu_2) )_{ij}( {\mathcal I}^{\sigma_*} (\nu_1))_{jk}   \right |
   \le &C\alpha^6 \nu_1 \nu_2,
\end{align}
\end{subequations}
where $i\ne k$, $C>0$ is independent of $\nu$, $\nu_1$, $\nu_2$ and $\alpha$.
\end{lemma}

\begin{proof}
Using (\ref{eqn:raltform})  we estimate that
\begin{align}
| ( {\mathcal R}^{\sigma_*} (\nu))_{ij} | & \le  \nu \frac{\alpha^3}{4} \| \mathrm{Im}( {\vec \Theta}^{(1)} ({\vec e}_i)) \|_{L^2({\vec B})} \| {\vec \Theta}^{(0)} ({\vec e}_j  ) \|_{L^2 ({\vec B})}  \nonumber \\
& \le  \nu \frac{\alpha^3}{4} \left ( \sum_{p=1}^3 \|  \mathrm{Im}( {\vec \Theta}^{(1)} ({\vec e}_p)) \|_{L^2({\vec B})}^2  \right )^{1/2}
\left ( \sum_{p=1}^3 \|   {\vec \Theta}^{(0)} ({\vec e}_p) \|_{L^2({\vec B})}^2  \right )^{1/2} \nonumber\\
& \le C  \nu \alpha^3  \left ( \sum_{p=1}^3 \|  \mathrm{Im}( {\vec \Theta}^{(1)} ({\vec e}_p)) \|_{L^2({\vec B})}^2  \right )^{1/2} \label{eqn:bdrij} ,
\end{align}
where $C>0$ does not depend on $\nu$ or $\alpha$ and, from (\ref{eqn:ialtform}), we estimate 
\begin{align}
|  ({\mathcal I}^{\sigma_*} (\nu))_{ij} |  \le & \nu \frac{\alpha^3}{4} \left ( \| \mathrm{Re}( {\vec \Theta}^{(1)} ({\vec e}_i)) \|_{L^2({\vec B})}  \| {\vec \Theta}^{(0)} ({\vec e}_j) \|_{L^2({\vec B})}  +  \| {\vec \Theta}^{(0)} ({\vec e}_i )  \|_{L^2({\vec B})}   \| {\vec \Theta}^{(0)} ({\vec e}_j )  \|_{L^2({\vec B})}  \right ) \nonumber \\
\le & \nu \frac{\alpha^3}{4} \left ( \sum_{p=1}^3 \|   {\vec \Theta}^{(0)} ({\vec e}_p) \|_{L^2({\vec B})}^2  \right )^{1/2} \left (
\left ( \sum_{p=1}^3 \|  \mathrm{Re}( {\vec \Theta}^{(1)} ({\vec e}_p)) \|_{L^2({\vec B})}^2  \right )^{1/2}
   \right . \nonumber \\
& \left .  + \left ( \sum_{p=1}^3 \|   {\vec \Theta}^{(0)} ({\vec e}_p) \|_{L^2({\vec B})}^2  \right )^{1/2}  \right ) \nonumber \\
\le & C \nu \alpha^3    \left (
\left ( \sum_{p=1}^3 \|  \mathrm{Re}( {\vec \Theta}^{(1)} ({\vec e}_p)) \|_{L^2({\vec B})}^2  \right )^{1/2} +1 \right ) \label{eqn:bdiij} .
\end{align}
Furthermore, using (\ref{eqn:spectraltheta1}), we obtain
\begin{align}
\|  \mathrm{Im}&( {\vec \Theta}^{(1)} ({\vec e}_p)) \|_{L^2({\vec B})}^2 =  \int_{\vec B} \mathrm{Im}( {\vec \Theta}^{(1)} ({\vec e}_p))  \cdot \mathrm{Im}( {\vec \Theta}^{(1)} ({\vec e}_p)) \dif {\vec \xi} \nonumber \\
& = \sum_{n=1}^\infty \sum_{m=1}^\infty \mathrm{Im}(\beta_n) \mathrm{Im}(\beta_m)  \left < {\vec \phi}_n , {\vec \Theta}^{(0)} ({\vec e}_p) \right >_{L^2({\vec B})}
 \left < {\vec \phi}_m , {\vec \Theta}^{(0)} ({\vec e}_p)\right >_{L^2({\vec B} )} \left <{\vec \phi}_n,{\vec \phi}_m \right >_{L^2({\vec B})} \nonumber \\
 & = \sum_{n=1}^\infty (\mathrm{Im}(\beta_n))^2 \left < {\vec \phi}_n , {\vec \Theta}^{(0)} ({\vec e}_p) \right >_{L^2({\vec B})}^2 \nonumber \\
& = \frac{1}{\nu^2} \sum_{n=1}^\infty \frac{\lambda_n^2}{(1+(\lambda_n/\nu)^2)^2} \left < {\vec \phi}_n , {\vec \Theta}^{(0)} ({\vec e}_p)  \right >_{L^2({\vec B})}^2 \label{eqn:expiml2int} ,
 \end{align}
where we have used $\mathrm{Im}( {\vec \Theta}^{(1)} ({\vec e}_p))\in {\mathbb R}$ and $\left < {\vec \phi}_n,{\vec \phi}_m \right >_{L^2({\vec B})}= \delta_{mn}$. In a similar way, we can show that
\begin{align}
\|  \mathrm{Re}( {\vec \Theta}^{(1)} ({\vec e}_p)) \|_{L^2({\vec B})}^2 = &  \sum_{n=1}^\infty \frac{1}{(1+(\lambda_n/\nu)^2)^2} \left < {\vec \phi}_n , {\vec \Theta}^{(0)} ({\vec e}_p) \right >_{L^2){\vec B})}^2 \label{eqn:exprel2int}.
\end{align}
Next, we use
\begin{align}
&\left |  ( {\mathcal R}^{\sigma_*} (\nu))_{ij} ( {\mathcal I}^{\sigma_*} (\nu) )_{jk}-  ({\mathcal I}^{\sigma_*} (\nu) )_{ij}( {\mathcal R}^{\sigma_*} (\nu))_{jk}   \right |  \le\nonumber \\
&|  ({\mathcal R}^{\sigma_*} (\nu) )_{ij} |  |  ( {\mathcal I}^{\sigma_*} (\nu) )_{jk} | + |  ({\mathcal I}^{\sigma_*} (\nu) )_{ij}| | ({\mathcal R}^{\sigma_*} (\nu))_{jk}| , \nonumber
\end{align}
and substitute (\ref{eqn:bdrij}) and (\ref{eqn:bdiij}) followed by (\ref{eqn:expiml2int}) and (\ref{eqn:exprel2int}) to obtain
\begin{align}
\left |  ( {\mathcal R}^{\sigma_*} (\nu))_{ij} ( {\mathcal I}^{\sigma_*} (\nu) )_{jk}-  ({\mathcal I}^{\sigma_*} (\nu) )_{ij}( {\mathcal R}^{\sigma_*} (\nu))_{jk}  \right |  \le &
C\alpha^6 \nu  E_1(\nu) E_2(\nu) , \label{eqn:bdcomri} 
\end{align}
where
\begin{align}
E_1(\nu) := & \left ( \sum_{n=1}^\infty \frac{\lambda_n^2}{\left (1+ (\lambda_n/\nu)^2 \right )^2} \sum_{p=1}^3 \left < {\vec \phi}_n , {\vec \Theta}^{(0)} ({\vec e}_p) \right >_{L^2({\vec B})}^2 \right )^{1/2}, \nonumber \\
E_2(\nu) : =&  \left ( \sum_{n=1}^\infty \frac{1}{\left (1+ (\lambda_n/\nu)^2 \right )^2} \sum_{p=1}^3 \left < {\vec \phi}_n , {\vec \Theta}^{(0)} ({\vec e}_p) \right >_{L^2({\vec B})}^2 \right )^{1/2}+1  \nonumber .
\end{align}
Still further, using (\ref{eqn:bdinnerprodthephi}), we obtain that
\begin{align}
E_1(\nu)^2 & \le \sum_{n=1}^\infty \lambda_n^2  \sum_{p=1}^3 \left < {\vec \phi}_n , {\vec \Theta}^{(0)} ({\vec e}_p) \right >_{L^2({\vec B})}^2  \le C, \nonumber
\end{align} 
independent of $\nu$. Since $E_2(\nu) < E_1(\nu)$  then we also have $E_2(\nu) <C $ independent of $\nu$, which, together with (\ref{eqn:bdcomri}), leads immediately to (\ref{eqn:bdcomrinew}).
The other two bounds are found in a similar way.

\end{proof}
 \begin{remark}
 In~\cite{ledgerlionheart2018mathmeth} we have proposed to use the eigenvalues of   $ {\mathcal R}^{\sigma_*} $ and  $ {\mathcal I}^{\sigma_*}  $ for the classification of objects, as they are known to be invariant under an object rotation, and their eigenvectors for determining an object's orientation. Lemma~\ref{lemma:commutator} 
 shows that the off-diagonal elements of the commutator  between $ {\mathcal R}^{\sigma_*} (\nu ) $ and $ {\mathcal I}^{\sigma_*} (\nu ) $, for general objects, grows at most linearly with $\nu$. Recalling that $\nu = \omega \sigma_*\mu \alpha^2$, then, by using
 \begin{equation}
 \frac{\dif}{\dif \omega} \left |  ({\mathcal R}^{\sigma_*} (\omega ))_{ij} ( {\mathcal I}^{\sigma_*} (\omega) )_{jk}- ( {\mathcal I}^{\sigma_*} (\omega) )_{ij} ({\mathcal R}^{\sigma_*} (\omega ) )_{jk} \right | ,
 \end{equation}
over a range of $\omega$, will also provide useful information and allow cases where the eigenvectors of  $ {\mathcal R}^{\sigma_*} (\omega ) $ and  $ {\mathcal I}^{\sigma_*} (\omega)$ are the same and where they differ to be distinguished.
 As an illustration, we include, in Figure~\ref{fig:comirregrandi}, the numerical results for $|(Z(\omega))_{ik}| =  |  ({\mathcal R}^{\sigma_*} (\omega ))_{ij} ( {\mathcal I}^{\sigma_*} (\omega) )_{jk}- ( {\mathcal I}^{\sigma_*} (\omega) )_{ij} ({\mathcal R}^{\sigma_*} (\omega ) )_{jk}  | $, $i\ne k$, for the irregular tetrahedron 
 previously considered in Figures~\ref{fig:fitcurvetetdiag} and~\ref{fig:fitcurvetetoffdiag}.
 We observe that the behaviour of $|(Z(\omega))_{ik}|$ tracks $\| {\mathcal R}^{\sigma_*} \|_F \| {\mathcal I}^{\sigma_*}\|_F $ and this behaviour is similar, in turn, to the estimate in (\ref{eqn:bdcomri}). 
 \begin{figure}[!h]
 \begin{center}
 \includegraphics[width=0.6\textwidth]{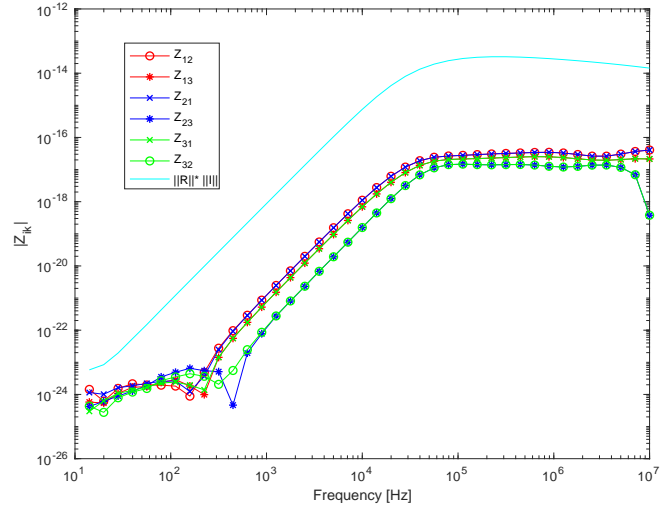}
 \end{center}
\caption{Conducting irregular tetrahedron with  $\alpha=0.01\text{m}$, $\mu_r=1.5$ and $\sigma_*=5.96 \times 10^6 \text{S/m}$: Behaviour of $|(Z(\omega))_{ik}| =  |  ({\mathcal R}^{\sigma_*} (\omega ))_{ij} ( {\mathcal I}^{\sigma_*} (\omega) )_{jk}- ( {\mathcal I}^{\sigma_*} (\omega) )_{ij} ({\mathcal R}^{\sigma_*} (\omega ) )_{jk}  | $ as a function of $\omega$ for $i\ne k$} \label{fig:comirregrandi}
 \end{figure} 
 \end{remark}

 \subsection{Mittag-Leffler expansion of ${\mathcal M}$}
Given a meromorphic function $f(w)$ in a region $\Omega$  with poles $a_n$, then Ahlfors~\cite[pg. 187]{ahlfors} explains how it can be expressed in the form
\begin{equation}
f(w) =g(z) + \sum_{n} P_n \left ( \frac{1}{w-a_n} \right ) ,
 \label{eqn:orgformml}
\end{equation}
where $P_n \left ( 1/(w-a_n) \right ) $ is a polynomial in $ 1/(w-a_n) $ for each pole $a_n$ and $g(w)$ is analytic in $\Omega$. Unfortunately, the sum on the right hand side is infinite and so there is no garuntee that it will converge in general. However, as described by Ahlfors, it is possible to modify (\ref{eqn:orgformml}) by  subtracting an analytic function $p_n$ from each singular part $P_n$, where each $p_n$ can be chosen as a polynomial. In the case where $\Omega$ is the complex plane, then, in Theorem 4 of~\cite[pg. 187]{ahlfors}, Ahlfors proves that every meromorphic function has a development in partial fractions and that the singular parts can be described  arbitrarily, with this being a particular case of a more general result due to Mittag-Leffler.
In particular, he explains that the modified series
\begin{equation}
f(w) =g(z) + \sum_{n} P_n  \left ( \frac{1}{w-a_n} \right )  -p_n(w),
 \label{eqn:modformml}
\end{equation}
can constructed by taking $p_n$ to be the Taylor series expansion of $ P_n  \left ( \frac{1}{w-a_n} \right ) $ expanded about $0$ and truncated at some sufficient degree $n_v$. Still further, he explains that the series in (\ref{eqn:modformml}) can be made absolutely convergent in the whole complex plane, apart from the poles, by choosing $n_v$ sufficiently large, in particular such that $2^{n_v} \ge M_n 2^n$ for all $n$ where $M_n = \max |P_n(w)|$ for $|w| < a_n/2$.

\vspace{0.1in}

We apply this result to ${\mathcal M}_{ij} (w)$ with $w:= \im \nu$ and obtain the following theorem, which is the main result of this section.

%
%
%
%
%

\begin{theorem} \label{thm:expandmittag}
The coefficients of ${{\mathcal M}} (w)$ are meromorphic in the whole complex plane with simple poles at $\lambda_n$ on the positive real axis, where $0 < \lambda_1 < \lambda_2 < \ldots $,
and is analytic at $w =0$ with ${{\mathcal M}} (0) = {\mathcal N}^0$. Thus, the coefficients of ${{\mathcal M}} (w)$  admit a Mittag-Leffler type expansion for simple poles in the form
\begin{align}
({\mathcal M}(w ))_{ij}  = & ( {\mathcal N}^0)_{ij} +   \sum_{n=1}^\infty   \left ( \frac{\lambda_n}{w -\lambda_n} +1   \right )  ({\mathcal A}^{(n)})_{ij}-p_n(w) 
, \label{eqn:mittagnu1}
\end{align}
where
\begin{align}
p_n(w) & =   - \left ( \frac{w}{\lambda_n} + \frac{w^2}{\lambda_n^2} + \ldots+\frac{w^{n_v}}{\lambda_n^{n_v}}  \right )({\mathcal A}^{(n)})_{ij} , \label{eqn:mittagni1}\\
({\mathcal A}^{(n)})_{ij} :=& -  \frac{\alpha^3 \lambda_n  }{4} \sum_{k=1}^{\mathrm{mult}(\lambda_n)} \left <  {\vec \phi}_{n,k} ,{\vec \Theta}^{(0)}({\vec e}_i) \right >_{L^2({\vec B})}   
\left <  {\vec \phi}_{n,k} ,{\vec \Theta}^{(0)}({\vec e}_j) \right >_{L^2({\vec B})} . \label{eqn:mittagni2}
\end{align}
In the above, $(\lambda_n,{\vec \phi}_n)$ are the eigenvalue--eigensolution pairs of (\ref{eqn:modelprob}).
The series can be made absolutely convergent in the complex plane, apart from at the poles, by choosing 
 $n_v$ sufficiently large, in particular such that $2^{n_v} \ge M_n 2^n$ for all $n$ where $M_n = \max |\left (\lambda_n/(w-\lambda_n) +1 \right ) A_{ij}^{(n)}|$ for $|w| < \lambda_n/2$.
\end{theorem}

\begin{proof}
Recall  ${\mathcal R}^{\sigma_*} + \im {\mathcal I}^{\sigma_*}= {\mathcal M} - {\mathcal N}^0 $ and from (\ref{eqn:modesmult}),
 that
\begin{align}
({{\mathcal M}} (w))_{ij} = & ({\mathcal N}^0)_{ij} + \frac{\alpha^3  }{4} \sum_{n=1}^\infty \beta_n(w) \lambda_n  \sum_{k=1}^{\mathrm{mult}(\lambda_n)} \left < {\vec \phi}_{n,k} ,{\vec \Theta}^{(0)}({\vec e}_i) \right >_{L^2({\vec B})}   \left < {\vec \phi}_{n,k} ,{\vec \Theta}^{(0)}({\vec e}_j) \right >_{L^2({\vec B})} , \nonumber
\end{align}
for objects ${\vec B}$ with homogenous $\sigma_*$, and possibly inhomogeneous $\mu_*$, where we have introduced  $w=\im \nu$.
Thus, by introducing (\ref{eqn:mittagni2}), we have
\begin{align}
({{\mathcal M}} (w) )_{ij}
=& ({\mathcal N}^0)_{ij} +   \sum_{n=1}^\infty   \frac{w }{w -\lambda_n} ({\mathcal A}^{(n)} )_{ij}
=( {\mathcal N}^0)_{ij} +   \sum_{n=1}^\infty \left ( \frac{\lambda_n}{w -\lambda_n} +1 \right ) ({\mathcal A}^{(n)})_{ij}  \label{eqn:pfractmittagg}.
\end{align}
We recall from Lemma~\ref{lemma:limitcase} that for the limiting case of $\nu =0$ we have ${\mathcal M}(0 ) ={\mathcal N}^0$, which, by Corollary~\ref{corollary:n0tops}, reduces to the P\'oyla-Szeg\"o tensor when considering a single object with homogeneous $\mu_*$, and as its coefficients  are independent of $\nu$ they are clearly analytic.
Thus, $f(w) =({{\mathcal M}} (w))_{ij} $ is of the form of (\ref{eqn:orgformml})
with $g(w) =  ({\mathcal N}^0)_{ij}$ and $P_n(w) = \left ( \frac{\lambda_n}{w -\lambda_n} +1 \right ) ({\mathcal A}^{(n)})_{ij} $ where the poles are simple. 
We already know from Lemma ~\ref{lemma:tenrepintheta1sum} that (\ref{eqn:pfractmittagg}) is convergent for $\nu\in [0,  \infty)$, i.e. when $w$ lies on the positive imaginary axis, away from the poles in the real axis. We can extend this further by applying the Mittag-Leffler Theorem, described above,
and constructing a modified expansion (\ref{eqn:mittagnu1}) where our $p_n(w) $ in (\ref{eqn:mittagni1}) is the Taylor series expansion of our $P_n(w) $ about $0$ and truncated at $n_v$ in such a way to ensure that it is convergent at all points in the complex plane away from the poles. This then immediately leads to our quoted result.

%
%
\end{proof}

\begin{corollary}\label{coll:speccasemittag}
Expanding ${\mathcal M}(s) = {\mathcal N}^0 + {\mathcal R}^{\sigma_*}(s)+ \im {\mathcal I}^{\sigma_*}(s)$ in terms of $s= -\im \omega$ we have that
$({\mathcal M} (s) )_{ij} $ is meromorphic in the whole complex plane with simple poles at $s_n=-\lambda_n / ( \mu_0 \sigma_* \alpha^2)$ on the negative real axis where $0 < |s_1 | < |s_2| < \ldots $ and is analytic at $s =0$ with ${{\mathcal M}} (0) = {\mathcal N}^0$ and, hence, in the case of $n_v=0$, admits the expansion
\begin{align}
({{\mathcal M}} (s) )_{ij}= &( {\mathcal N}^0)_{ij} +\sum_{n=1}^\infty \left ( \frac{s_n}{s-s_n} + 1 \right)  ( {\mathcal A}^{(n)})_{ij},   \label{eqn:expandmins}
\end{align}
which is absolutely convergent in the whole complex plane, apart from the poles, provided that $\max |\left ( s_n / ( s-s_n) +1 \right )  {\mathcal A}^{(n)})_{ij} |$, for $|s| < |s_n|/2$, decays faster than $2^{-n}$.
\end{corollary}

\begin{proof}
For $n_v=0$ the result stated in (\ref{eqn:mittagnu1}) in Theorem~\ref{thm:expandmittag} becomes
\begin{align}
({\mathcal M} (w ) )_{ij}= & 
( {\mathcal N}^0)_{ij} +   \sum_{n=1}^\infty   \left ( \frac{\lambda_n}{w -\lambda_n} + 1 \right ) ( {\mathcal A}^{(n)})_{ij}, \label{eqn:expandw0}
\end{align}
which is convergent for $\nu\in [0, \infty) $, i.e. when $w$ lies on the positive imaginary axis and is absolutely convergent in the whole complex plane, apart from the poles provided that $M_n = \max |\left ( \lambda_n / ( w-\lambda_n) +1 \right )  {\mathcal A}^{(n)})_{ij} |$ for $|w| < \lambda_n/2$ decays faster than $2^{-n}$ . Still further, using a simple change of variables, we can obtain an expansion of ${\mathcal M}(s) = {\mathcal N}^0 + {\mathcal R}^{\sigma_*}(s)+ \im {\mathcal I}^{\sigma_*}(s)$ in terms of $s =-\im \omega$ and find that the poles are at  $s_n=-\lambda_n / ( \mu_0 \sigma_* \alpha^2)$ on the negative real axis where $0 < |s_1 | < |s_2| < \ldots $. Making the change of variables $ w = -s \mu_0 \sigma_* \alpha^2 $ in (\ref{eqn:expandw0}) gives (\ref{eqn:expandmins}).
\end{proof}
\begin{remark}
Wait and Spies~\cite{wait1969} obtained an analytical solution for a conducting permeable sphere and obtained explicit expressions for the tensor coefficients and the negative real values of the poles $s_n$ for this case. Their choice of $s_n$ corresponds to our $ s_n\mu_r$ in the case of a permeable homogeneous object, but ours is more general as it can also be applied to inhomogeneous objects where $\mu_r$ is no longer a constant. For other shapes with homogeneous parameters,
 Baum~\cite{baum499} has predicted that $ {\mathcal M}(s)$ has simple poles on the negative real axis and quoted
\begin{align}
{{\mathcal M}} (s) =&  {\mathcal M}(0) + \sum_{n=1}^\infty \frac{s}{s_n(s-s_n)} M_n {\vec M}_n \otimes {\vec M}_n  + \text{possible entire function} ,\label{eqn:baum} \nonumber \\
=& {\mathcal M}(0) +  \sum_{n=1}^\infty \left ( \frac{1}{s-s_n} + \frac{1}{s_n} \right)  M_n {\vec M}_n \otimes {\vec M}_n + \text{possible entire function}, 
\end{align}
without a formal proof and without explicit expressions for the tensor coefficients or the scalars $M_n$. He proposes a numerical approximation approach for calculation of $s_n$ and the eigenvectors ${\vec M}_n$, but does not make reference to eigenvalue problem (\ref{eqn:modelprob}), which is fundamental to their correct computation. His prediction uses $s=\im \omega$ as he applies $e^{\im \omega t}$ to obtain the time harmonic equations instead of $e^{-\im \omega t} $ used in this work.
Indeed, the subject of Baum's prediction was of subject of some considerable debate see e.g. ~\cite{ramm1982}.
 His prediction can be seen as a special case of Theorem~\ref{thm:expandmittag} discussed in Corollary~\ref{coll:speccasemittag},  which by comparing with (\ref{eqn:expandmins}) makes clear the definition of all the terms and makes explicit the correct eigenvalue problem (\ref{eqn:modelprob}) that needs to be solved. Although, importantly, (\ref{eqn:expandmins}) will only be absolutely convergent in the complex plane, apart from the poles, if $\max | ( s_n / ( s-s_n ) + 1  )   {\mathcal A}^{(n)}_{ij}|$, for $|s| <| s_n| /2$, decays faster than $2^{-n}$.
\end{remark}

 \section{Transient response of $({\vec H}_\alpha - {\vec H}_0)({\vec x})$ } \label{sect:tranresp}
  Building on the earlier work of Wait and Spies~\cite{wait1969}, who have obtained an analytical expression for transient response from a conducting permeable sphere, we can now apply Theorem~\ref{thm:expandmittag} to obtain explicit expressions for the transient response from an inhomogeneous conducting permeable object with $\sigma_*$ fixed.
 
\begin{theorem} \label{thm:stepfun}
 The transient perturbed magnetic field response to ${\vec B}_\alpha$ with fixed $\sigma_*$ placed in a background field ${\vec H}_0^{step}({\vec x},t)= {\vec H}_0({\vec x}) u(t)$  is
 \begin{align}
 ({\vec H}_\alpha - {\vec H}_0^{step})({\vec x},t)_i = & ( {\vec D}^2 G({\vec x},{\vec z}))_{ij} (  {\mathcal M} ^{step}(t))_{jk}
  ({\vec H}_0({\vec z}))_k  + ({\vec R}({\vec x},t))_i , \nonumber \\
 ( {\mathcal M}^{step}(t))_{jk} := &\left (({\mathcal N}^0)_{jk} + \sum_{n=1}^\infty e^{s_n t}( {\mathcal A}^{(n)} )_{jk} \right )u(t),  \nonumber
 \end{align}
  where ${\vec H}_0({\vec x})$ is  real valued and $u(t)$ a unit step function, generated by a divergence free current source of the form ${\vec J}_0^{step}({\vec x},t) = {\vec J}_0({\vec x}) u(t)$ with real valued ${\vec J}_0({\vec x}) $.
 In the above, ${\vec H}_\alpha ({\vec x},t)$ is the transient magnetic interaction field, which satisfies the transient version of (\ref{eqn:eddyqns}), $s_n = - \lambda_n/(\mu_0 \sigma_* \alpha^2)$, $\lambda_n$ is an eigenvalue of (\ref{eqn:modelprob}) and ${\mathcal A}_{ij}^{(n)}$ is as defined in (\ref{eqn:mittagni2}). 
 If the conditions of the asymptotic formula (\ref{eqn:asymp}) are met then ${\vec R}({\vec x},t)= {\vec 0}$.
 \end{theorem}
 
 \begin{proof}
For consistency with Wait and Spies~\cite{wait1969} we set $s=\im \omega$ and apply
 \begin{align}
 ({\vec H}_\alpha - {\vec H}_0^{step})({\vec x},t)_i  = \frac{1}{2\pi \im} \int_{c-\im \infty}^{c+\im \infty} \frac{1}{s} \overline{({\vec H}_\alpha - {\vec H}_0) ({\vec x})_i} e^{s t} \dif s = {\mathcal L}^{-1} \left ( \frac{1}{s} \overline{({\vec H}_\alpha - {\vec H}_0) ({\vec x})_i} \right ), \nonumber 
 \end{align}
where $c$ is a positive constant and ${\mathcal L}^{-1}$ denotes the inverse Laplace transform. The complex conjugate of $({\vec H}_\alpha - {\vec H}_0) ({\vec x})_i$ is taken as Wait and Spies use $e^{\im \omega t}$ rather than $e^{-\im \omega t } $ used here.
Now, substituting the asymptotic formula (\ref{eqn:asymp}), we have, assuming ${\vec H}_0$ is real, that
  \begin{align}
 ({\vec H}_\alpha - {\vec H}_0^{step})({\vec x},t)_i  = & ( {\vec D}^2 G({\vec x},{\vec z}))_{ij}  \frac{1}{2\pi \im} \int_{c-\im \infty}^{c+\im \infty} \frac{1}{s} \overline{( {\mathcal M})_{jk}}  e^{s t} \dif s ({\vec H}_0({\vec z}))_k  \nonumber \\
&  +  \frac{1}{2\pi \im} \int_{c-\im \infty}^{c+\im \infty} \frac{1}{s} \overline{{\vec R}({\vec x},s)_i} e^{st} \dif s \nonumber .
 \end{align}
 By considering (\ref{eqn:mittagnu1}), applying the change of variables from $w= \im \nu$ to $s=- \im \omega$, so that the poles lie on the negative real axis at $s_n=-\lambda_n / ( \mu_0 \sigma_* \alpha^2)$, as discussed in Corollary~\ref{coll:speccasemittag}, and closing the contour by an infinite  semicircle in the left hand $s$ plane, we find, 
for $t > 0$, that
 \begin{align}
 \frac{1}{2\pi \im}\int_{c-\im \infty}^{c+\im \infty} \frac{1}{s} \overline{( {\mathcal M}) _{jk}}  e^{s t} \dif s  =&  \frac{1}{2\pi \im}\int_{c-\im \infty}^{c+\im \infty} \frac{1}{s}  \left (
({\mathcal N}^0)_{ij} +\sum_{n=1}^\infty \left ( \frac{s_n}{\overline{s}-s_n}+1 - \overline{q_n(s) } \right)  ( {\mathcal A}^{(n)})_{ij}
\right )e^{st} \dif s \nonumber \\
 = & \Res_{s=0,  \overline{s}=s_n  } \left  ( \frac{1}{s}  
\left (
({\mathcal N}^0)_{ij} +\sum_{n=1}^\infty \left ( \frac{s_n}{\overline{s}-s_n} +1 -\overline{q_n(s)} \right) ( {\mathcal A}^{(n)})_{ij} \right )e ^{st}  \right )  u(t)\nonumber \\
=& \left ( \left ( ( {\mathcal N}^0)_{ij} + \sum_{n=1}^\infty \left ( \frac{s_n}{0-s_n}+1  -\overline{q_n(0)} \right ) ( {\mathcal A}^{(n)})_{ij} \right )e^{0t}\right . \nonumber \\
&+ \left . \sum_{n=1}^\infty \frac{s_n}{s_n }   ( {\mathcal A}^{(n)})_{ij} e^{s_nt} \right ) u(t) \nonumber \\
 = & \left (  ( { \mathcal N}^0)_{ij} + \sum_{n=1}^\infty  (  {\mathcal A}^{(n)} )_{ij} e^{s_nt} \right ) u(t). \nonumber
 \end{align} 
where $q_n(s): = -(  \frac{s}{s_n} +\frac{s^2}{s_n^2}+ \ldots + \frac{s^{n_v}}{s_n^{n_v}})$ and we have used the fact that $s_n$ is real. 
For $t < 0 $, we close the integral by an infinite semicircle in the righthand $s$ plane and find that the integral vanishes in this case as there are no poles in the right hand plane. 
From~\cite{ammarivolkov2013}, under the conditions of (\ref{eqn:asymp}) are met,  then ${\vec R}({\vec x},s) \le C \nu \alpha^4 \|{\vec H}_0 \|_{W^{2,\infty}(B_\alpha)}= C |s| \mu_0 \sigma_* \alpha^6 \|{\vec H}_0 \|_{W^{2,\infty}(B_\alpha)}$ and, hence, 
\begin{equation}
  ({\vec R}({\vec x},t))_i = \frac{1}{2\pi \im} \int_{c-\im \infty}^{c+\im \infty} \frac{1}{s} \overline{({\vec R}({\vec x},s) )_i}e^{st} \dif s = {\mathcal L}^{-1} \left ( \frac{1}{s} \overline{({\vec R}({\vec x},s) )_i} \right )  = 0.
 \end{equation}
Thus, the result immediately follows.
  \end{proof}
 
 \begin{remark}
 Theorem~\ref{thm:stepfun} shows the long-time response of the perturbed field for a step function characterises of an inhomogeneous object by the ${\mathcal N}^0$ tensor, which describes the magnetostatic characteristics of ${\vec B}_\alpha$. Similar observations were found for a conducting permeable sphere by Wait and Spies~\cite{wait1969}. Despite the issues with the convergence of Baum's~\cite{baum499} for a homogenous object it leads to a predication that is similar to that obtained in (\ref{eqn:baum}) when the correct form of Mittag-Leffler theorem is used. However, importantly, all terms are now explicitly defined and, 
 under the conditions of (\ref{eqn:asymp}),
 can be computed.
  \end{remark}
 
 \begin{theorem} \label{thm:impfun}
The transient perturbed magnetic field response to ${\vec B}_\alpha$ with fixed $\sigma_*$ placed in a background field ${\vec H}_0^{imp}({\vec x},t)= {\vec H}_0({\vec x}) \delta (t)$ is
 \begin{align}
 ({\vec H}_\alpha - {\vec H}_0^{imp })({\vec x},t)_i = & ({\vec D}^2 G({\vec x}, {\vec z}))_{ij} 
({\mathcal M}^{imp}(t) )_{jk}
 ({\vec H}_0({\vec z}))_k + \left  ( \tilde{\vec R}({\vec x},t)\right )_i , \nonumber \\
(  {\mathcal M}^{imp}(t))_{jk} : = &
 ({\mathcal M}(\infty))_{jk} \delta(t) + \sum_{n=1}^\infty s_n e^{s_n t} ( {\mathcal A}^{(n)} )_{jk} u(t) , \nonumber
 \end{align}
 where ${\vec H}_0({\vec x})$ is  real valued and $\delta(t)$ is a delta function associated with an impulse at $t=0$, generated by a divergence free current source of the form ${\vec J}_0^{imp}({\vec x},t) = {\vec J}_0({\vec x}) \delta(t)$ with real valued ${\vec J}_0({\vec x}) $.
 In the above, ${\vec H}_\alpha ({\vec x},t)$ is the transient magnetic interaction field, which satisfies the transient version of (\ref{eqn:eddyqns}), $s_n = - \lambda_n/(\mu_0 \sigma_* \alpha^2)$, $\lambda_n$ is an eigenvalue of (\ref{eqn:modelprob}) and $( {\mathcal A}^{(n)})_{ij}$ is as defined in (\ref{eqn:mittagni2}).  If the conditions of the asymptotic formula (\ref{eqn:asymp}) are met then ${\vec R}({\vec x},t)= {\vec 0}$.
 \end{theorem}

  \begin{proof}
 Applying~\cite{wait1969} we have
 \begin{align}
 ({\vec H}_\alpha - {\vec H}_0^{imp})({\vec x},t)_i  = \frac{\partial }{\partial t} ( ({\vec H}_\alpha - {\vec H}_0^{step })({\vec x},t)_i ) ,
 \end{align}
then, since $\delta(t) = \dif / \dif t ( u(t) )$, it follows that
  \begin{align}
 ({\vec H}_\alpha - {\vec H}_0^{imp })({\vec x},t)_i =& ({\vec D}^2 G({\vec x}, {\vec z}))_{ij} \left ( 
 \left  ( ( { \mathcal N}^0)_{jk} + \sum_{n=1}^\infty e^{s_n t} ( {\mathcal A}^{(n)})_{jk} \right ) \delta(t) + \sum_{n=1}^\infty s_n e^{s_n t}  ({\mathcal A}^{(n)})_{jk} u(t) 
 \right)  \nonumber \\
 &  ({\vec H}_0({\vec z}))_k + \left  ( \tilde{\vec R}({\vec x},t)\right )_i .  \nonumber
 \end{align}
 Considering that the first term in parenthesis is only present at time $t=0$ we have, using (\ref{eqn:mittagni2}), that
 \begin{align}
( {\mathcal N}^0)_{jk} + \sum_{n=1}^\infty  ({\mathcal A}^{(n)})_{jk} = & ( {\mathcal N}^0)_{jk}   - \frac{\alpha^3 \lambda_n  }{4} \sum_{k=1}^{\mathrm{mult}(\lambda_n)} \left < {\vec \phi}_{n,k} ,{\vec \Theta}^{(0)}({\vec e}_i) \right >_{L^2({\vec B})}  \left < {\vec \phi}_{n,k} ,{\vec \Theta}^{(0)}({\vec e}_j) \right >_{L^2({\vec B})} \nonumber \\
 =& (  {\mathcal N}^0)_{jk} - \frac{\alpha^3}{4}  \sum_{n=1}^\infty \sum_{m=1}^\infty  \left ( \left <   \mu_r ^{-1} \nabla \times  P_n({\vec \Theta}^{(0)}({\vec e}_i)),  \nabla \times  P_m({\vec \Theta}^{(0)}({\vec e}_j))  \right >_{L^2({\vec B})} \right . \nonumber\\
 &\left . +   \left < \nabla \times  P_n({\vec \Theta}^{(0)}({\vec e}_i)),\nabla \times  P_m({\vec \Theta}^{(0)}({\vec e}_j))  \right >_{L^2({\vec B}^c)} \right ) \nonumber ,
   \end{align}
  where (\ref{eqn:eigtranformrule}) has been applied. Notice that 
  $  \lim_{\nu \to \infty} \beta_n=  \lim_{\nu \to \infty} -\frac{\im \nu}{\im \nu - \lambda_n} = -1 $
  and, thus, from Lemma~\ref{lemma:theta1rep} we have
  \begin{align}
  \lim_{\nu \to \infty } {\vec \Theta}^{(1)} ({\vec u}) = - \sum_{n=1}^\infty P_n({\vec \Theta}^{(0)}({\vec u}) )\nonumber ,
\end{align}
and so
\begin{align}
( {\mathcal N}^0)_{jk} + \sum_{n=1}^\infty  ( {\mathcal A}^{(n)})_{jk} =  &( {\mathcal N}^0)_{jk} - \frac{\alpha^3}{4}  \lim_{\nu \to \infty}  \left ( \left <  \mu_r ^{-1} \nabla \times {\vec \Theta}^{(1)}({\vec e}_i),  \nabla \times  {\vec \Theta}^{(1)}({\vec e}_j)  \right>_{L^2({\vec B})} \right . \nonumber\\
&\left . + \left < \nabla \times {\vec \Theta}^{(1)}({\vec e}_i),  \nabla \times  {\vec \Theta}^{(1)}({\vec e}_j)  \right>_{L^2({\vec B}^c)} \right ) \nonumber .
 \end{align}
 Still further, $ {\vec \Theta}^{(1)} ({\vec u}) = {\vec \Theta}({\vec u}) + \tilde{\vec \Theta}^{(0)}({\vec u})$,  and since $\tilde{\vec \Theta}^{(0)}({\vec u}) $ is independent of $\nu$, we have, using Theorem~\ref {thm:formsn0}, and writing in terms of inner products, that
 \begin{align}
 ( {\mathcal N}^0)_{jk}& + \sum_{n=1}^\infty (  {\mathcal A}^{(n)} )_{jk}=
 \alpha^3  \int_{{\vec B}}  \left ( 1- \mu_r^{-1}\right )   \delta_{jk}  \dif {\vec \xi}  + \frac{\alpha^3}{4} \left ( \left <  \mu_r^{-1}  \nabla  \times \tilde{\vec \Theta}^{(0)}({\vec e}_j) , \nabla  \times \tilde{\vec \Theta}^{(0)} ({\vec e}_k) \right >_{L^2{\vec B})}  \right . \nonumber \\
 & \left . + \left <  \nabla  \times \tilde{\vec \Theta}^{(0)}({\vec e}_j) , \nabla  \times \tilde{\vec \Theta}^{(0)} ({\vec e}_k) \right >_{L^2({\vec B}^c)}  \right ) 
- \frac{\alpha^3}{4}  \lim_{\nu \to \infty}  \left ( \left <   \mu_r ^{-1} \nabla \times {\vec \Theta}^{(1)}({\vec e}_j),  \nabla \times  {\vec \Theta}^{(1)}({\vec e}_k)  \right >_{L^2({\vec B})}  \right . \nonumber\\
&\left . + \left <  \nabla \times {\vec \Theta}^{(1)}({\vec e}_j),  \nabla \times  {\vec \Theta}^{(1)}({\vec e}_k)  \right >_{L^2({\vec B}^c)} \right ) \nonumber \\
=& \alpha^3  \int_{{\vec B}}  \left ( 1-  \mu_r^{-1} \right )  \delta_{jk} \dif {\vec \xi}  
 - \frac{\alpha^3}{4} \lim_{\nu \to \infty} \left (  \left <  \tilde{\mu}_r^{-1} \nabla \times {\vec \Theta} ({\vec e}_j),  \nabla \times  {\vec \Theta}({\vec e}_k) \right >_{L^2({\vec B} \cup {\vec B}^c) }  \right .   \nonumber \\
&\left . + \left <  \tilde{\mu}_r^{-1} \nabla \times {\vec \Theta}({\vec e}_j),  \nabla \times  {\vec \Theta}^{(0)}({\vec e}_k) \right >_{L^2({\vec B} \cup {\vec B}^c) } 
 +\left <  \tilde{\mu}_r^{-1} \nabla \times {\vec \Theta}^{(0)}({\vec e}_j),  \nabla \times  {\vec \Theta}({\vec e}_k) \right >_{L^2({\vec B} \cup {\vec B}^c) }
   \right ) \nonumber ,
\end{align}
and since $\lim_{\nu \to \infty} {\vec \Theta}({\vec u}) = - {\vec u} \times {\vec \xi} $ in $B$ then  
\begin{align}
\lim_{\nu \to \infty} \left <  {\mu}_r^{-1} \nabla \times {\vec \Theta} ({\vec e}_j),  \nabla \times  {\vec \Theta} ({\vec e}_k)  \right > _{L^2({\vec B})} = 4 \int_{\vec B} \mu_r^{-1} \delta_{jk} \dif {\vec \xi}. \nonumber
\end{align}
We can also show, by integration by parts,  that 
\begin{align}
 \lim_{\nu \to \infty} &\left <  \tilde{\mu}_r^{-1} \nabla \times {\vec \Theta}({\vec e}_j),  \nabla \times  {\vec \Theta}^{(0)}({\vec e}_k) \right >_{L^2({{\vec B} \cup {\vec B}^c })} \nonumber \\
  = & \lim_{\nu \to \infty}  \left (\int_{\partial {\vec B}} {\vec n}^- \cdot  {\vec \Theta}({\vec e}_j) \times {\mu}_r^{-1}  \nabla \times  {\vec \Theta}^{(0)}({\vec e}_k) |_- \dif {\vec 
 \xi} + 
 \int_{\partial {\vec B}} {\vec n}^+ \cdot  {\vec \Theta}({\vec e}_j) \times {\mu}_r^{-1}  \nabla \times  {\vec \Theta}^{(0)}({\vec e}_k) |_+ \dif {\vec 
 \xi} \right )\nonumber \\
 =& - \lim_{\nu \to \infty}  \int_{\partial {\vec B}}{\vec \Theta}({\vec e}_j) \cdot [ \tilde{\mu}_r^{-1} \nabla \times  {\vec \Theta}^{(0)}({\vec e}_k) \times {\vec n}^- ]_{\partial {\vec B}} \dif {\vec \xi}   = 2 \lim_{\nu \to \infty}  \int_{\partial {\vec B}}  [\tilde{\mu}_r^{-1}]_{\partial {\vec B}}  {\vec \Theta} ({\vec e}_j) \cdot {\vec e}_k\times {\vec n}^- \dif{\vec \xi}
 \nonumber \\
 =& 
 4\int_{\vec B} ( 1- \mu_r^{-1}) \delta_{jk} \dif {\vec \xi} , \nonumber 
 \end{align}
 and similarly obtain 
 \begin{align}
 \lim_{\nu \to \infty} &\left <  \tilde{\mu}_r^{-1} \nabla \times {\vec \Theta}^{(0)}({\vec e}_j),  \nabla \times  {\vec \Theta} ({\vec e}_k) \right >_{L^2({{\vec B} \cup {\vec B}^c })} =  \lim_{\nu \to \infty} \left <  \tilde{\mu}_r^{-1} \nabla \times \overline{{\vec \Theta} ({\vec e}_k)} ,  \nabla \times  {\vec \Theta} ^{(0)} ({\vec e}_j) \right >_{L^2({{\vec B} \cup {\vec B}^c }) }  \nonumber \\
  = & \lim_{\nu \to \infty}  \left (\int_{\partial {\vec B}} {\vec n}^- \cdot  \overline{{\vec \Theta}({\vec e}_k  )} \times {\mu}_r^{-1}  \nabla \times  {\vec \Theta}^{(0)}({\vec e}_ j ) |_- \dif {\vec 
 \xi} + 
 \int_{\partial {\vec B}} {\vec n}^+ \cdot \overline{ {\vec \Theta}({\vec e}_k)} \times {\mu}_r^{-1}  \nabla \times  {\vec \Theta}^{(0)}({\vec e}_j) |_+ \dif {\vec 
 \xi} \right )\nonumber \\
 =& - \lim_{\nu \to \infty}  \int_{\partial {\vec B}} \overline{{\vec \Theta}({\vec e}_k)} \cdot [ \tilde{\mu}_r^{-1} \nabla \times  {\vec \Theta}^{(0)}({\vec e}_j) \times {\vec n}^- ]_{\partial {\vec B}} \dif {\vec \xi}   = 2 \lim_{\nu \to \infty}  \int_{\partial {\vec B}}  [\tilde{\mu}_r^{-1}]_{\partial {\vec B}} \overline{  {\vec \Theta} ({\vec e}_k) }\cdot {\vec e}_j\times {\vec n}^- \dif{\vec \xi}
 \nonumber \\
 =& 
 4\int_{\vec B} ( 1- \mu_r^{-1}) \delta_{jk} \dif {\vec \xi}  \nonumber .
 \end{align}
 Thus, we finally obtain that  
 \begin{align}
( {\mathcal N}^0)_{jk} + \sum_{n=1}^\infty  ( {\mathcal A}^{(n)})_{jk} =&- \alpha^3  \int_{\vec B} \delta_{jk} \dif {\vec \xi} - \frac{\alpha^3}{4} \left < 
 \nabla \times {\vec \Theta}^{(\infty)} ({\vec e}_j),  \nabla \times  {\vec \Theta}^{(\infty)}({\vec e}_k) \right >_{L^2({\vec B}^c)}=(  {\mathcal M}(\infty) )_{jk} ,
  \end{align}
  where we used $\lim_{\nu \to \infty} {\vec \Theta}({\vec u}) =  {\vec \Theta}^{(\infty)}({\vec u})$, which satisfies the transmission problem (\ref{eqn:tpthetainf}). 
 \end{proof}

 \begin{remark}
 Theorem~\ref{thm:impfun} shows that the short-time response of the perturbed field for an impulse function characterises an inhomogeneous object by the coefficients of the ${\mathcal M}(\infty)$ tensor, which describes a perfectly conducting object ${\vec B}_\alpha$. Similar observations were found for a conducting permeable sphere by Wait and Spies~\cite{wait1969}. This is also confirms Baum's~\cite{baum499}  predication for homogeneous conducting objects and makes explicit all of the terms if the  conditions of the asymptotic formula (\ref{eqn:asymp}) are met.
  \end{remark}
 
 \begin{remark}
 Theorems~\ref{thm:stepfun}  and~\ref{thm:impfun} rely on the conditions of the asymptotic formula (\ref{eqn:asymp}) in order for ${\vec R}({\vec x},t)$ to vanish. In general, when these conditions are not met,  we do not have an estimate of  ${\vec R}({\vec x},t)$. Quantifying its behaviour for more general circumstances will form part of our future work.
 \end{remark}

\section*{Acknowledgement}

The authors are grateful for the useful discussions with Professor Habib Ammari, ETH Zurich, Switzerland, and Professor Faouzi Triki, Universit\'e Grenoble Alpes, France,  whose insight aided the proof of convergence of   the series in (\ref{eqn:spectraltheta1}).
The authors would like to thank EPSRC for the financial support received from the grants EP/R002134/1 and EP/R002177/1. The second author would like to thank the Royal Society for the financial support received from a Royal Society Wolfson Research Merit Award.  All data are provided in full in Section~\ref{sect:spectrum} of this paper.

\bibliographystyle{plain}
\bibliography{ledgerlionheart}

\end{document}